\documentclass[english,10pt]{amsart}
\usepackage[english]{babel}

\usepackage[hidelinks]{hyperref}

\usepackage{verbatim}
\usepackage{amsfonts,amssymb,mathtools}
\usepackage{amsmath}
\usepackage{bbm}
\usepackage[foot]{amsaddr}

\usepackage{todonotes}

\usepackage{graphicx,xcolor}
\definecolor{green}{RGB}{89,169,58}
\definecolor{red}{RGB}{224,61,42}
\definecolor{blue}{RGB}{63,78,181}

\graphicspath{{fig/}}
\usepackage{subcaption, tikz}
\usetikzlibrary{decorations.markings}
\usetikzlibrary{arrows}

\newcommand{\N}{\ensuremath{\mathbb{N}}}
\newcommand{\R}{\ensuremath{\mathbb{R}}}
\newcommand{\Z}{\ensuremath{\mathbb{Z}}}
\newcommand{\E}{\ensuremath{\mathbb{E}}}
\renewcommand{\P}{\ensuremath{\mathbb{P}}}

\newcommand{\ind}[1]{\ensuremath{\mathbbm{1}_{\left\{#1\right\}}}}
\newcommand{\diff}{\mathop{}\mathopen{}\mathrm{d}}
\newcommand{\cal}[1]{\ensuremath{\mathcal{#1}}}
\newcommand\croc[1]{\left\langle #1\right\rangle}

\def\eps{\varepsilon}

\newtheorem{proposition}{Proposition}
\newtheorem{definition}[proposition]{Definition}
\newtheorem{lemma}[proposition]{Lemma}
\newtheorem{remark}[proposition]{Remark}
\newtheorem{theorem}[proposition]{Theorem}

\newtheorem{corollary}[proposition]{Corollary}

\newtheorem{assumption}{Assumption}

\numberwithin{proposition}{section}

\title{Bulk stability of bimolecular stochastic chemical reaction networks}

\author{Andrea Agazzi}
\email{andrea.agazzi@unibe.ch}
\author{Lucie Laurence}
\address{Institute of Mathematical Statistics and Actuarial Science, Department of Mathematics and Statistics, University of Bern, Alpeneggstrasse 22, 3012 Bern (CH)}
\email{lucie.laurence@unibe.ch}

\date{\today}

\begin{document}

\begin{abstract}
We establish criteria for stability in the bulk of planar stochastic chemical reaction networks with mass-action kinetics. Using Poincaré compactification, we characterize the asymptotic drift and show how stochastic fluctuations can prevent escape along transversally repulsive branches of the limiting deterministic flow. We construct a Lyapunov function by combining local estimates near these branches with dissipative estimates in the neighboring regions. The criteria cover a wide class of networks and in particular apply beyond weakly reversible and strongly endotactic networks. Together with compatible boundary estimates, these criteria imply positive recurrence.
\end{abstract}

\maketitle

\section{Introduction}\label{s:Intro}

Stability is a fundamental requirement for the study of the long-time behavior of chemical
reaction networks. For stochastic models, stability is often understood as (positive) recurrence, giving an invariant
probability measure on each closed irreducible class and allowing one to study the system at
stationarity. This is especially important in the multiscale setting, where the effective dynamics
of slow variables are often obtained by averaging over the stationary distribution of
a fast subsystem~\cite{Popovic06,Kang2}. Establishing such a reduction requires control
of the fast dynamics, including their stability and the relevant moment bounds.

Chemical reaction networks admit two classical descriptions. The deterministic model
follows concentrations through a system of mass-action ordinary differential equations.
The stochastic model tracks molecular copy numbers through a continuous-time Markov chain, retaining the fluctuations that are particularly important at low copy numbers~\cite{Gillespie1977}. Under the natural volume scaling, the deterministic equations arise as a law of large numbers limit of the stochastic dynamics on finite time intervals~\cite{Kurtz70,Kurtz71}. However, this connection does not, by itself, identify their long-time stability properties.

Stability in this context has been extensively studied in both deterministic and stochastic settings. Starting with the theory of
complex balance~\cite{HornJackson,Feinberg}, the deterministic literature has developed
structural conditions relating the reaction graph to persistence and permanence,
notably weak reversibility and endotacticity~\cite{CraciunNazarovPantea2013,Craciun2019, GopalkrishnanMillerShiu2014}.
Related conditions also play a central role in the stochastic theory. Complex balance
yields product-form stationary distributions \cite{Anderson1,Anderson9}, while the
geometry of strongly endotactic networks underlies stochastic Lyapunov estimates and
large-deviation results~\cite{Agazzi1,Agazzi2,Cappelletti4}.

The two descriptions can nevertheless behave very differently near the boundary of the
state space, where some species have low copy numbers. Extinction, trapping near the
boundary and changes of recurrence class can occur in the stochastic model without a
corresponding deterministic behavior~\cite{AndersonEncisoJohnston2014,Anderson8,Agazzi4,Cappelletti4}.
In the bulk, where all species have large copy numbers, the stochastic drift is instead
well approximated by the deterministic vector field. Much of the stability analysis in
this region has consequently treated the stochastic dynamics as a perturbation of the
deterministic one, using deterministic inward drift to construct stochastic Lyapunov
functions \cite{Agazzi2,Cappelletti4}. This suggests a close correspondence between the
two models in the bulk, but does not establish an equivalence of their stability
conditions.

In our previous article~\cite{Agazzi25}, we showed that this correspondence can fail
even without boundary effects. We constructed a network whose deterministic system contains a trajectory that diverges in finite time in the interior of the state space, while the stochastic model is positive recurrent. The divergent trajectory is repulsive in the transverse direction, and stochastic fluctuations push the process away towards regions where the dynamics dissipate. While an earlier example based on a special choice of rate constants appears in~\cite[Example~4]{Anderson9}, the system from~\cite{Agazzi25} displays the stabilization by noise for generic choices of such constants. These examples show that deterministic bulk stability is not necessary for stochastic stability.

This paper continues our program by giving general criteria for bulk stability of two-species networks. The goal is to give conditions on the deterministic vector field that guarantee sufficient stability (encoded through Lyapunov-type dissipation inequalities) in the bulk of the stochastic system. Building on the ideas of~\cite{Agazzi25}, we focus on its behavior at infinity: which curves lead to infinity, and whether they attract or repel neighboring trajectories. Our aim is to introduce new ideas and a new approach to studying stochastic dynamics in the bulk, without seeking a complete classification of bulk stability.

To obtain this information, we carry out a \emph{compactification} of the deterministic vector field, using Poincaré and Poincaré-Lyapunov compactifications. These compactifications are particularly well suited to polynomial planar vector fields, see~\cite{Dumortier2006}, and allow us to ``spread'' infinity along a circle, so that different asymptotic directions can be studied separately. 
Different scaling regimes must be analysed, determined by the graph of the CRN studied. 

The singular points at infinity of the compactified vector field describe the asymptotic behavior of invariant curves. We distinguish \emph{convergent} and \emph{divergent} curves, along which the flow goes towards smaller and larger states, respectively. Any such (convergent or divergent) curve can be further classified as \emph{attractive} or \emph{repulsive} according to whether neighboring trajectories move towards it or away from it. The example in~\cite{Agazzi25} shows that a repulsive divergent curve in the deterministic system can still allow stability of the stochastic system. Under the non-degeneracy and drift assumptions stated below, we show stability in the bulk when all attractive curves are convergent. 

To prove this result we construct, for each such system, an energy function dissipated in the bulk, which can be
combined with compatible boundary estimates to obtain positive recurrence. The proof of the dissipation inequality relies on a careful approximation of the infinitesimal generator of the process in each scaling regime. These approximations, given in Propositions~\ref{p:ReducedGenerator11} and~\ref{p:ReducedGeneratorNeq11}, are used throughout the paper, both to show pointwise dissipation of the energy function and to derive scaling limits for the process near repulsive curves. We then patch the local energy functions together to obtain dissipation throughout the bulk.

Our result significantly extends previous work on the stability of chemical reaction systems with stochastic mass-action kinetics by exploiting the purely stochastic phenomenon of stabilization by noise. In particular, when combined with complementary dissipation estimates on the boundary, it allows us to establish positive recurrence for networks outside the endotactic family, thereby extending the class of stochastically stable networks well beyond the deterministically stable ones. This further clarifies the relationship between the asymptotic behavior of deterministic and stochastic models of chemical reaction networks.

At the conceptual level, the main novelty of our approach lies in using the Poincaré compactification of the deterministic vector field to derive criteria for stability in the bulk of stochastic CRNs. 
The main technical challenge stems from the generality of our results: the broad family of networks considered gives rise to a number of qualitatively different cases requiring separate, ad hoc treatment. 
To address this challenge, we develop a principled approach based on a careful decomposition of the generator in various scaling regimes, allowing us to obtain the desired dissipation estimates for all the different asymptotic behaviors of the model. 
To combine these estimates, we also develop a conceptually straightforward but rather general patching lemma that may be of independent interest. We believe our assumptions to be close to optimal, apart from some asymptotically degenerate cases that we explicitly discuss and leave for future work.

The paper is organized as follows. Section~\ref{Model} presents the general setting of CRNs. Section~\ref{s:Poincare} introduces Poincaré compactification and illustrates it with two examples. The main results and their scope of application are presented in Section~\ref{s:Theorem}. The rest of the paper is devoted to the proof of the main theorem. In Sections~\ref{s:R11} and~\ref{s:RNeq11}, we construct energy functions near repulsive curves in the diagonal and sub-diagonal regions, respectively, and prove their dissipation. Section~\ref{s:A} deals with attractive curves, completing the analysis within each compactification. Section~\ref{s:Drift} treats the regions between the compactifications, called \emph{drift regions}. Finally, Section~\ref{Gluing} patches the local energy functions together to complete the proof.
\section{Model and notations}\label{Model}

\subsection{General notations}
We write $\Z$ for the integers and $\N=\Z_{\geq 0}$ for the nonnegative integers.
As $r\to\infty$, $O(r)$ denotes a quantity whose ratio to $r$ is bounded, and $o(r)$ one whose ratio to $r$ tends to zero. The notation $x\sim r$ means that $x$ is ``of order $r$'', i.e., $|x/r|$ is upper and lower bounded (away from 0) when $r\to+\infty$. 
We use the Euclidean norm $\|\cdot\|$ and scalar product $\croc{\cdot,\cdot}$.

\subsection{Chemical reaction networks}

We introduce here the standard notations used in the theory of chemical reaction networks (CRN). 

A CRN is a triplet of finite sets $(\cal{S}, \cal{C}, \cal{R})$ with 
\begin{itemize}
	\item $\cal{S}$ the set of species. We restrict the study to planar CRNs, with only two species. Therefore, $\cal{S}=\{S_1, S_2\}$. 
	\item $\cal{C}$ is the set of complexes. It is a subset of $\N^{|\cal{S}|}$, each coordinate giving the number of molecules of the corresponding species in the complex. 
	\item $\cal{R}$ is the set of reactions. It is a subset of $\{(y_1, y_2)\in \cal{C}^2: y_1\neq y_2\}$. 
\end{itemize}
A given CRN is usually represented as a directed graph, with vertices $\cal{C}$ and edges $\cal{R}$. 
An element $y=(y_1,y_2)\in \cal{C}$ is often written as $y_1S_1+y_2S_2$, and $(0,0)$ as $\emptyset$. An element $(y^-, y^+)\in \cal{R}$ is written $y^-\to y^+$.

We introduce the set of input vertices, or source complexes, 
\[
	\cal{C}^-:=\{y\in \cal{C}: \exists y\to y^+\in \cal{R}\}.
\]
We assume that each of the species $S_1$ and $S_2$ is present in at least one of the source complexes. 
For a reaction $\nu:y_\nu^-\to y_\nu^+$, we write $\xi^\nu=y_\nu^+-y_\nu^-$ for the reaction vector, giving the net change in molecule numbers when the reaction occurs.
We set $K_0:=\max_{\nu\in\cal{R}}\|\xi^\nu\|$. For $S\subset\Z_{\geq 0}^2$, its two-sided $K_0$-boundary is
\[
    \partial^{K_0}(S):=\{x\in\Z_{\geq 0}^2: \operatorname{dist}(x,S)\leq K_0,\ \operatorname{dist}(x,S^c)\leq K_0\},
\]
where $S^c=\Z_{\geq 0}^2\setminus S$. This contains both endpoints of any reaction jump crossing the boundary of $S$.

\subsubsection*{Deterministic system}

The state of the deterministic system is given by the concentrations $x(t)\in \R_{\geq 0}^2$. The kinetics of the reactions follow the law of mass action. Given a set of positive reaction rate constants $\kappa=(\kappa_\nu)_{\nu \in \cal{R}}$, the ODE associated to the system $(\cal{S}, \cal{C}, \cal{R}, \kappa)$ is 
\begin{equation}\label{ODE}
	\dot x = \Delta(x):=\sum_{\nu\in \cal{R}}\kappa_{\nu}x_1^{y_{\nu,1}^-}x_2^{y_{\nu,2}^-}\xi^\nu.  
\end{equation}
For $x\in \R_{\geq 0}^2$ and $y\in \cal{C}$, we use the notation $x^y:=x_1^{y_1}x_2^{y_2}$. 

\subsubsection*{Stochastic system}

The stochastic system tracks the number of molecules of each species. At time $t$, its state is given by $X(t)\in \Z_{\geq 0}^2$. The reaction kinetics again follow the law of mass action. 
The rate of a reaction $\nu$ depends on the state of the system, and is zero when there are not enough reactant molecules. Explicitly,
\begin{equation}\label{e:SDE}
	\lambda_\nu(x):=
	\begin{cases}
		0&\text{   if   }x_1<y_{\nu, 1}^- \text{  or  } x_2<y_{\nu, 2}^-\\
		\kappa_\nu \frac{x_1!}{(x_1-y_{\nu,1}^-)!}\frac{x_2!}{(x_2-y_{\nu, 2}^-)!} &\text{   otherwise}\\
	\end{cases}
\end{equation}
Equivalently, 
\begin{equation}
	\lambda_\nu(x)=\kappa_\nu x_1(x_1-1)\ldots(x_1-y_{\nu, 1}^-+1)x_2(x_2-1)\ldots(x_2-y_{\nu, 2}^-+1)
\end{equation}
for $x\in \Z_{\geq 0}^2$, with empty products equal to one.

Given a set of reaction rate constants $\kappa=(\kappa_\nu)_{\nu \in \cal{R}}$, the stochastic process associated to the system $(\cal{S}, \cal{C}, \cal{R}, \kappa)$ is a continuous time Markov chain whose infinitesimal generator acts on a test function $g$ as
\begin{equation}\label{e:IGOriginal}
	\cal{A}(g)(x)=\sum_{\nu\in\cal{R}}\lambda_\nu(x)(g(x+\xi^\nu)-g(x)). 
\end{equation}

\subsection{Polytope}

The source polytope is the convex hull of $\cal{C}^-$. To determine which reactions dominate at large states, we consider scalings $x_i\sim N^{\omega_i}$, with $N\to\infty$, and normalize the directions by $\max\{\omega_1,\omega_2\}=1$. Thus, we introduce
	\begin{equation*}
		\Omega=\{(1, \omega_2), \omega_2\in [0, 1]\}\cup \{(\omega_1, 1), \omega_1\in [0, 1]\}.
	\end{equation*}
	We order $\Omega$ in the trigonometric sense, i.e. for $\omega, \omega'\in \Omega$, 
	\[
		\omega\leq \omega' \text{   if and only if   } \omega_2\leq \omega_2' \text{  and  } \omega_1\geq \omega_1'. 
	\]
	Also, $\omega<\omega'$ if $\omega\leq \omega'$ and $\omega\neq \omega'$. 
	For $\omega^-\leq\omega^+$, 
	\begin{align*}
		[\omega^-, \omega^+]:= \{\omega\in \Omega: \omega^-\leq \omega\leq \omega^+\},\\
		(\omega^-, \omega^+):= \{\omega\in \Omega: \omega^-< \omega<\omega^+\}. 
	\end{align*}
	Finally, we introduce $\Omega_{\leq}=\{\omega\in\Omega: \omega\leq (1,1)\}$.

Along the scaling $x_i\sim N^{\omega_i}$, the source monomial $x^{y^-}$ has order $N^{\croc{y^-,\omega}}$. The largest exponent therefore identifies the dominant source complexes and their reactions.

\begin{definition}
	The dominant source complexes and reactions in direction $\omega\in \Omega$ are 
	\begin{align*}
		\cal{C}_\omega&:=\arg\max\{\croc{y^-, \omega}: y^-\in \cal{C}^-\}\\
		\cal{R}_\omega&:= \{y^-\to y^+\in \cal{R}: \, y^-\in \cal{C}_\omega\}. 
	\end{align*}
	The dominant order is 
	\begin{equation}\label{s:RhoDef}
		\rho_0(\omega):= \max\{\croc{y^-, \omega}: y^-\in \cal{C}^-\}. 
	\end{equation}
	If each species occurs in at least one source complex, then $\rho_0(\omega)\geq 1$ for all $\omega\in \Omega$. We will use the simplified notation $\rho_0=\rho_0(\omega)$ when $\omega$ is clear from the context.  
	
	We identify the nontrivial exposed faces of the source polytope with their directions,
	\begin{equation}\label{e:DefP}
		\cal{P}:= \{\omega\in \Omega: \, |\cal{C}_\omega|>1\}. 
	\end{equation}
\end{definition}

Note that $\cal{P}$ is finite. States at order of magnitude $N^\omega$ for $\omega\in \cal{P}$ have to be studied more carefully, because in this case several reactions compete. Reactions starting from the same complex have proportional rates and combine into a single drift vector. When this drift vector is $0$, the complex is called a \emph{ghost complex}. This cancellation requires specific values of $\kappa$, which we exclude.  

For $\omega\in \Omega$, we define the drift contributed by the dominant reactions, 
\begin{equation}\label{e:fDelta}
	\Delta^\omega(u, v):= \sum_{y^-\to y^+\in \cal{R}_\omega}\kappa_{y^-\to y^+} (y^+-y^-)u^{y_1^-}v^{y_2^-}. 
\end{equation}

At states of order $(N^{\omega_1}, N^{\omega_2})$ with $N\gg 1$, the leading deterministic drift is described by the ODEs
\begin{equation}\label{ODEOmega}
	\begin{cases}
		\dot{x}_1= \Delta_1^{\omega}(x_1, x_2)\\
		\dot{x}_2= \Delta_2^{\omega}(x_1, x_2)\\
	\end{cases}
\end{equation}
We finally introduce the sets of states used to separate these scaling regimes. 
For $\omega^-<\omega^+\in \Omega$ and $m=(m^-, m^+)\in\R_{>0}^2$, we set 
\begin{equation}\label{e:DefSetT}
	\cal{T}(\omega^-, \omega^+; m^-, m^+):= \left\{x\in \R^2_{>0}: \frac{x_2^{\omega_1^-}}{x_1^{\omega_2^-}}\geq m^- \text{  and  } \frac{x_2^{\omega_1^+}}{x_1^{\omega_2^+}}\leq m^+ \right\}
\end{equation}
The set $\cal{T}(\omega^-,\omega^+; m^-, m^+)$ lies between the curves $x_2^{\omega_1^-}=m^-x_1^{\omega_2^-}$ and ${x_2^{\omega_1^+}=m^+x_1^{\omega_2^+}}$. When no direction of $\cal{P}$ lies between $\omega^-$ and $\omega^+$, these are the ``in between'' states forming a \emph{drift region}, at which the process is essentially deterministic, and the generator is well approximated by its first order terms. Here, the leading vector field follows the direction of the sum of the reaction vectors starting from one single complex, weighted by their rate constants. 
In the following, we will sometimes use the concise notation $\cal{T}(\omega_-, \omega_+)$ for states at these orders of magnitude. 

For $\omega\in \Omega$ and $0<u_-<u_+<\infty$, we set 
\begin{equation}\label{e:DefSetH}
	\cal{H}_\omega(u_-, u_+):= \left\{x\in \R^2_{>0}: u_-\leq \frac{x_2^{\omega_1}}{x_1^{\omega_2}}\leq u_+\right\}
\end{equation}
Thus, $\cal{H}_\omega(u_-,u_+)$ is a strip around the scaling direction $\omega$. In particular, for $\omega=(1,\gamma)$, it is given by $u_-\leq x_2/x_1^\gamma\leq u_+$. For the specific $\omega\in \cal{P}$, we will need to partition the states of order $(N^{\omega_1}, N^{\omega_2})$ in these strips defined above, according to the directions of the vector field given by the dominant reactions at these states. In the following, we will sometimes use the concise notation $\cal{H}_\omega$ to denote states at these orders of magnitude. 

\subsection{Example}

We will use the following example throughout the paper to illustrate the results. It is (almost) the example given in~\cite{Agazzi25}. The graph of the CRN is as follows: 
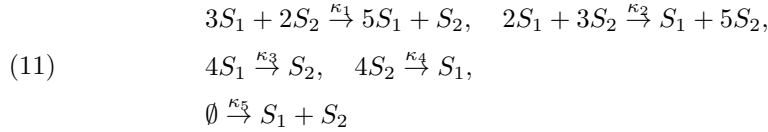
\begin{equation}\label{e:ExCRN1}
	\begin{aligned}
		&3S_1+2S_2\stackrel{\kappa_1}{\rightarrow} 5S_1+S_2, \quad 2S_1+3S_2\stackrel{\kappa_2}{\rightarrow} S_1+5S_2,\\
		&4S_1\stackrel{\kappa_3}{\rightarrow} S_2, \quad 4S_2\stackrel{\kappa_4}{\rightarrow} S_1,\\
		&\emptyset\stackrel{\kappa_5}{\rightarrow} S_1+S_2
	\end{aligned}
\end{equation}
Here $\cal{S}=\{S_1, S_2\}$, $\cal{C}$ contains 10 complexes, including for example $\emptyset$ or $3S_1+2S_2$, and there are 5 different reactions. The $\kappa_i$ written on top of each arrow correspond to the reaction rate constants, all positive. In what follows, we will call reaction $i$ the reaction associated to $\kappa_i$. 

The associated deterministic system is defined by the following ODE (see Relation~\eqref{ODE}): 
\[
	\begin{aligned}
		\dot{x}_1(t)&=2\kappa_1x_1^3x_2^2-\kappa_2x_1^2x_2^3-4\kappa_3x_1^4+\kappa_4x_2^4+\kappa_5\\
		\dot{x}_2(t)&=-\kappa_1x_1^3x_2^2+2\kappa_2x_1^2x_2^3+\kappa_3x_1^4-4\kappa_4x_2^4+\kappa_5
	\end{aligned}
\]
In the stochastic representation, the associated Markov process has 5 transitions, one for each reaction, defined according to Relation~\eqref{e:SDE}. For example, the transition associated to reaction $1$ is $(x_1, x_2)\mapsto (x_1+2, x_2-1)$ and has rate
\[
	\kappa_1x_1(x_1-1)(x_1-2)x_2(x_2-1). 
\]
The faces of the polytope associated to this CRN, defined by the set of normal vectors in Relation~\eqref{e:DefP}, are 
\[
	\cal{P}:=\{(1,1), (1,1/2), (1/2, 1)\}.
\]
On the diagonal, i.e. for $\omega=(1,1)$, the dominant reactions are reactions $1$ and $2$: $\cal{C}_\omega=\{3S_1+2S_2, 2S_1+3S_2\}$ and $\rho_0(\omega)=5$. In this case, the reduced ODE is given by Relation~\eqref{e:fDelta}:
\[
	\dot{y}=\Delta^{(1,1)}(y)= \kappa_1y_1^3y_2^2\binom{2}{-1}+\kappa_2y_1^2y_2^3\binom{-1}{2}.
\]
For $\omega=(1, 1/2)$, the reactions $1$ and $3$ compete: $\cal{C}_\omega=\{3S_1+2S_2, 4S_1\}$ and $\rho_0(\omega)=4$. 

\subsection{Dissipation and energy function}
The main purpose of this article is, in the spirit of Lyapunov-type criteria (see Section~\ref{s:Lyapunov}), to show a dissipation inequality for an energy function constructed in the bulk of the state space. The energy will be defined separately on subsets of the state space of the form $\cal{H}_\omega$ and $\cal{T}(\omega_-, \omega_+)$ as introduced above in Relations~\eqref{e:DefSetH} and~\eqref{e:DefSetT}. One key property needed to patch the different sets together is that the functions ``match'' in terms of order of magnitude near the boundary between two sets. CRNs are known to dissipate functions with controlled growth: the Entropy function for example, see~\cite{Cappelletti4}, or linear and polynomial functions, see~\cite{Howells2025} or Section 5 of~\cite{LR24} for example. In this spirit, our functions are all ``of order $\|x\|^b$'' at the boundary of their set, hence the following definition: 

\begin{definition}\label{d:b-dominance}
	Let $b>0$. 
	\begin{itemize}
		\item A function $V$ defined on a set $\cal{H}_\omega(u_-, u')$, resp. on a set $\cal{T}(\omega, \omega'; u_-, u')$ is said to be $b$-dominated on its lower boundary if we can find $v_->u_-$ and some $C, c>0$, such that for $x\in \cal{H}_\omega(u_-, v_-)$ with $\|x\|$ large enough, 
		\[
			c\leq \frac{V(x)}{\|x\|^{b}}\leq C. 
		\]
		\item Symmetrically, a function $V$ defined on a set $\cal{H}_\omega(u', u_+)$, resp. on a set $\cal{T}(\omega', \omega; u', u_+)$ is said to be $b$-dominated on its upper boundary if we can find $v_+<u_+$ and some $C, c>0$, such that for $x\in \cal{H}_\omega(v_+, u_+)$ with $\|x\|$ large enough, 
		\[
			c\leq \frac{V(x)}{\|x\|^{b}}\leq C. 
		\]
		\item We say the function $V$ is $b$-dominated on its boundaries when both are true. 
	\end{itemize}
\end{definition}

Note that these ``boundaries'' are boundaries of the subsets, and are not related to the boundary of the state space which correspond to the two strips $\{x_1\leq K'\}$ and $\{x_2\leq K'\}$ for some $K'\geq K_0$.

\begin{remark}\label{r:SwitchSpecies}
	In the rest of the paper, we present the estimates and proofs for $\omega=(1,\gamma)$, $0<\gamma\leq1$. For $\omega=(\gamma,1)$, the same arguments apply after interchanging the two species and the corresponding coordinates of the reaction complexes. This gives the required bounds on the other side of the state space. 
\end{remark}
\section{The Poincar\'e compactification}\label{s:Poincare}

As mentioned in the introduction, the main goal of this article is to approach the stability properties of the stochastic process associated to a CRN through the behavior of the vector field given by the deterministic representation of the same CRN. We are interested only in what happens in the bulk of the system, i.e., when both species are present in large numbers, and all reactions can happen. The question of boundary behaviors is not dealt with here. 

In~\cite{Agazzi25}, we have seen that the deterministic vector field can contain one divergent trajectory, as long as it is \emph{repulsive}, in the sense that a neighboring trajectory tends to move away from it. In other words, this direction at infinity leads to an ``equilibrium at infinity'' which is unstable in the ``angular direction''. This is a property of the vector field in the limit of large states. One natural way to understand how the vector field behaves ``at infinity'' is to carry out a \emph{compactification} of $\R_{+}^2$. 
In the spirit of~\cite{Agazzi25}, we would like to be able to distinguish the different directions leading to infinity, and describe their attractive or repulsive character. Using the usual stereographic projection would not give us enough information on the system, because the ``infinity'' is represented as a single point, and therefore we can't make sense of a ``repulsive'' property of a point at infinity. Instead, we will use Poincar\'e-compactification (and Poincaré--Lyapunov compactification), for which the states at infinity are spread along a circle, see~\cite{Perko2013, Dumortier2006} and \cite{Poincare1881} for the historical reference. 
The choice of this compactification is actually quite natural with polynomial planar vector fields as ours are. To our knowledge, it hasn't been used for the study of CRN yet, except on the Brusselator system, see~\cite{Llibre2020}. 

For each face of the polytope, i.e., each $\omega\in \cal{P}$, a different compactification will be carried out, taking into account the different orders of magnitude of the states. For $\omega=(1,1)$, the classical Poincaré-compactification is sufficient, which roughly corresponds to making the change of variables $$(u,v)=(x_2/x_1, (x_1^2+x_2^2)^{-1/2}),$$ and the ``states at infinity'' are represented by the states $v=0$. For $\omega<(1,1)$, we will use the Poincaré--Lyapunov compactification instead, see~\cite{Dumortier2006}, that takes into account the difference in order of magnitude between $x_1$ and $x_2$. In this case, the relevant change of variables is 
$$(u, v)=(x_2/x_1^{\omega_2}, 1/x_1).$$
The goal of this article is to describe the stability of the stochastic CRN according to the data of its associated compactifications. 

This part is meant to introduce the main quantities defined by the compactification that will play a role throughout the article. The changes of variables and limits when $v\to 0$ are presented here at a heuristic level, to explain the origin of the functions used in the proofs.

\subsection{Poincaré compactification on the face $\omega=(1,1)$}
When $\omega=(1,1)$, we introduce the change of variables 
\[
	u(x)=\frac{x_2}{x_1} \quad\text{and}\quad v(x)=(x_1^2+x_2^2)^{-1/2}.
\]
We are interested in the states for which $0<u<+\infty$. When $u$ approaches $0$, to see what happens, we need to ``zoom in'' the zone $u=0$, which corresponds to carrying out a different compactification (taking into account a smaller order of magnitude for $x_2$). 
Making this change in variables in the dynamical system defined by equation~\eqref{ODE}, we get: 
\begin{align*}
	\dot u &= v\sqrt{1+u^2}\left(-u\Delta_1\left(\frac{1}{v\sqrt{1+u^2}},\frac{u}{v\sqrt{1+u^2}}\right)+\Delta_2\left(\frac{1}{v\sqrt{1+u^2}},\frac{u}{v\sqrt{1+u^2}}\right) \right),\\
	\dot v &= -\frac{v^2}{\sqrt{1+u^2}}\left(\Delta_1\left(\frac{1}{v\sqrt{1+u^2}},\frac{u}{v\sqrt{1+u^2}}\right)+u\Delta_2\left(\frac{1}{v\sqrt{1+u^2}},\frac{u}{v\sqrt{1+u^2}}\right) \right).
\end{align*}
When $v\to 0$, we can reduce this equation to the leading order terms, replacing $\Delta_1$ and $\Delta_2$ by $\Delta_1^{\omega}$ and $\Delta_2^{\omega}$. Both these polynomial functions are sums of monomials of the same degree, given by $\rho_0:=\rho_0(\omega)$. Therefore, we get: 
\begin{align*}
	\dot u &\approx \frac{1}{(1+u^2)^{(\rho_0+1)/2}v^{\rho_0-1}}(1+u^2)\left(-u\Delta_1^\omega\left(1, u\right)+\Delta_2^\omega\left(1,u\right) \right)\\
	\dot v &\approx -v\frac{1}{(1+u^2)^{(\rho_0+1)/2}v^{\rho_0-1}}\left(\Delta_1^\omega\left(1,u\right)+u\Delta_2^\omega\left(1,u\right) \right)
\end{align*}
and after a change in time, we get: 
\begin{align*}
	\dot u &\approx (1+u^2)(-u\Delta_1^\omega\left(1, u\right)+\Delta_2^\omega\left(1,u\right))\\
	\dot v &\approx -v\left(\Delta_1^\omega\left(1,u\right)+u\Delta_2^\omega\left(1,u\right) \right)
\end{align*}
Using this last ODE, we can classify the behavior of the system. Invariant lines of the leading homogeneous system are given by the equilibrium points $(u_*,0)$ of the compactified ODE (the equation of the line is then $x_2=u_*x_1$). Looking at the Jacobian matrix of the dynamical system, we can classify the lines as \emph{attractive} or \emph{repulsive}, and \emph{convergent} or \emph{divergent}. 

\begin{definition}\label{d:Comp11}
	The compactification at $\omega= (1,1)$ is given by the data of the two functions 
	\begin{align*}
			F_\omega(u)&:=(1+u^2)(-u\Delta_1^\omega(1, u) + \Delta_2^\omega(1, u))\\
			G_\omega(u)&:= \Delta_1^\omega(1, u) + u\Delta_2^\omega(1, u). 
	\end{align*}
\end{definition}

\subsection{Poincaré-Lyapunov compactification on the face $(1,0)<\omega<(1,1)$}
A similar calculation can be done for the faces where $\omega_2<1$. Let's fix $\omega=(1,\gamma)$ with $\gamma\in (0,1)$. In this case, the change of variable is given by 
\[
	u(x)=\frac{x_2}{x_1^{\gamma}} \quad\text{and}\quad v(x)=\frac{1}{x_1}.
\]
Similar calculations as in the previous section (reducing the ODE to the leading terms and changing time) lead to the equations for $(u,v)$:
\begin{align*}
	\dot u &\approx \Delta_2^\omega\left(1,u\right)\\
	\dot v &\approx -v^{2-\gamma}\Delta_1^\omega\left(1,u\right)
\end{align*}
As in the previous section, we define the compactification for $\omega\neq (1,1)$. 
\begin{definition}\label{d:CompNeq11}
	The compactification at $(1,0)< \omega< (1,1)$ is given by the data of the two functions 
	\begin{align*}
			F_\omega(u)&:=\Delta_2^\omega(1, u)\\
			G_\omega(u)&:= \Delta_1^\omega(1, u). 
	\end{align*}
\end{definition}

\subsection{Singular points and invariant lines/curves of the dynamical system}
We now introduce the \emph{singular points} of the compactifications, that describe the asymptotic directions of invariant curves of the leading order vector field. 
\begin{definition}\label{d:CompSingPoint}	
	Let $\omega\in ((1,0),(1,1)]$. A singular point of the compactification at $\omega$ is a point $(u_*, v=0)$ such that $F_\omega(u_*)=0$. 	
	We define the set of singular values of $u$ as 
	\[
		\cal{P}_\omega=\{u\in (0, +\infty): \, F_\omega(u)=0\}. 
	\]
	When $1\leq p:=|\cal{P}_\omega|<\infty$, we represent the set $\cal{P}_\omega$ as
	\[
		\cal{P}_\omega=\{u_1, \ldots, u_p\},
	\]
	where $0=: u_0<u_1<u_2<\ldots u_p<u_{p+1}:=\infty$. 
	The singular point $u_*\in \cal{P}_\omega$ is said to be 
	\begin{itemize}
		\item \emph{attractive} (resp. \emph{repulsive}) if $F'_\omega(u_*)<0$ (resp. $F'_\omega(u_*)>0$)
		\item \emph{convergent} (resp. \emph{divergent}) if $G_\omega(u_*)<0$ (resp. $G_\omega(u_*)>0$)
	\end{itemize}
	When $p=0$, the compactification is called \emph{trivial}. 
\end{definition}
For $\omega=(1,1)$, a singular point of the compactification corresponds to an invariant line of the (leading) system, defined by the equation $x_2=u_*x_1$. For $\omega<(1,1)$, it corresponds to an invariant \emph{curve}, whose equation verifies $x_2=(u_*+o(1))x_1^{\omega_2}$ when $x_1\to +\infty$. The line/curve associated to $u_*$ is divergent if the points on the invariant lines go towards infinity. In this case, in the compactified coordinates, $v\to 0$, hence the sign of $G_\omega(u_*)$. All these definitions can be very simply represented on a graph on the Poincar\'e Disk, see the following examples.
 
\begin{remark}\label{r:NonTrivialCompDrift}
	Intuitively, non-trivial compactifications can only occur when multiple reactions compete, i.e., on faces $\omega\in \cal{P}$, where $\cal{P}$ has been defined in Relation~\eqref{e:DefP}. This intuition is mainly true: for $\omega<(1,1)$, when $\omega\notin \cal{P}$ (and $\Delta_{2}^{\omega}(1,1)\neq 0$) we have $\cal{P}_\omega=\emptyset$. The vector field in the region $\cal{H}_\omega$ is dominated by a drift coming from \emph{one} dominant complex, and the compactification at $\omega$ is trivial. 
	However, in the case of $\omega=(1,1)\notin\cal{P}$, we can still have a non-trivial compactification at $\omega$: the drift has a direction $\Delta^\omega(1,1)$, and if both $\Delta_{i}^{(1,1)}(1,1)$ have the same sign, it is in a direction of an invariant line $x_2=u^*x_1$ of the plane. This can be seen in the expression of $F_\omega(u)$ given in Definition~\ref{d:Comp11}. When the system goes towards infinity, i.e., when $\Delta_{i}^{(1,1)}(1,1)>0$ for both $i=1,2$, the singular point in the compactification is attractive and divergent, which we will see makes the system unstable. 
\end{remark}

\subsection{First example}
For the CRN given in Relation~\eqref{e:ExCRN1}, we have two diagonal/sub-diagonal regions to study: $\omega=(1,1)$ and $\omega=(1,1/2)$. 
For $\omega=(1,1)$, the polynomials defined by the Poincar\'e compactification in Definition~\ref{d:Comp11} are: 
\begin{multline*}
	F_\omega(u)= \kappa_2(1+u^2)u^2(u^2-2(\alpha-1)u-\alpha)\\
	\text{and} \quad G_\omega(u)= \kappa_2u^2(2u^2-(\alpha+1)u+2\alpha), 
\end{multline*}
with $\alpha=\kappa_1/\kappa_2$. 
$F_\omega(u)$ has only one positive root, at $u_*= \alpha-1+\sqrt{(\alpha-1)^2+\alpha}$. We can show that $F'_\omega(u_*)>0$ and $G_\omega(u_*)>0$, and therefore the only singular point of this compactification is repulsive and divergent. In particular, the calculations made in~\cite{Agazzi25} correspond to the case $\kappa_1=\kappa_2=1$, where $u_*=1$, $F'_\omega(u_*)=4$ and $G_\omega(u_*)=2$. 

For $\omega=(1, 1/2)$, the polynomials defined by the Poincar\'e compactification in Definition~\ref{d:CompNeq11} are 
\[
	F_\omega(u)=\kappa_3-\kappa_1u^2 \quad \text{and}\quad G_\omega(u)=2\kappa_1u^2-4\kappa_3. 
\]
Here again, $F_\omega(u)$ has one single positive root at $u_*=\sqrt{\kappa_3/\kappa_1}$. $F'_\omega(u_*)=-2\sqrt{\kappa_1\kappa_3}$ and $G_\omega(u_*)=-2\kappa_3<0$: the singular point is attractive and convergent. 

This is represented concisely in Figure~\ref{fig:Agazzi25Comp2}.

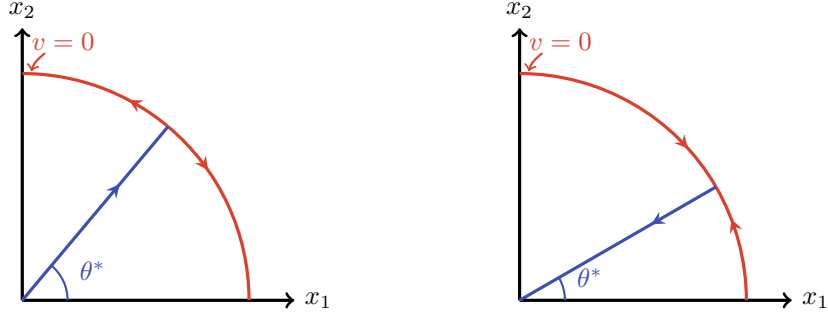
\begin{figure}[h]
    \centering
    \begin{subfigure}{0.48\textwidth}
        \centering
				\begin{tikzpicture}[scale=3]
						\def\ts{50}  
						\draw[->, very thick] (0, 0) -- (1.2, 0) node[right] {$x_1$};
						\draw[->, very thick] (0, 0) -- (0, 1.2) node[above] {$x_2$};
						\draw[red, very thick,
								postaction={decorate, decoration={markings,
										mark=at position 0.7 with {\arrowreversed{stealth}}}}]
								(0,1) arc (90:\ts:1);
						\draw[red, very thick,
								postaction={decorate, decoration={markings,
										mark=at position 0.7 with {\arrowreversed{stealth}}}}]
								(1,0) arc (0:\ts:1);
						\draw[blue, very thick,
								postaction={decorate, decoration={markings,
										mark=at position 0.67 with {\arrow{stealth}}}}]
								(0,0) -- ({cos(\ts)}, {sin(\ts)});
						\draw[blue, thick] (0.2,0) arc (0:\ts:0.2);
						\node[blue] at (0.31, 0.13) {$\theta^*$};
						\node[red] at (0.18, 1.14) {$v = 0$};
						\draw[->, red, thick] (0.10, 1.09) to[out=220, in=75] (0.04, 1.01);
				\end{tikzpicture}
    \end{subfigure}
    \hfill
    \begin{subfigure}{0.48\textwidth}
        \centering
        \begin{tikzpicture}[scale=3]
						\def\ts{30}  
						\draw[->, very thick] (0, 0) -- (1.2, 0) node[right] {$x_1$};
						\draw[->, very thick] (0, 0) -- (0, 1.2) node[above] {$x_2$};
						\draw[red, very thick,
								postaction={decorate, decoration={markings,
										mark=at position 0.8 with {\arrow{stealth}}}}]
								(0,1) arc (90:\ts:1);
						\draw[red, very thick,
								postaction={decorate, decoration={markings,
										mark=at position 0.7 with {\arrow{stealth}}}}]
								(1,0) arc (0:\ts:1);
						\draw[blue, very thick,
								postaction={decorate, decoration={markings,
										mark=at position 0.67 with {\arrowreversed{stealth}}}}]
								(0,0) -- ({cos(\ts)}, {sin(\ts)});
						\draw[blue, thick] (0.2,0) arc (0:\ts:0.2);
						\node[blue] at (0.31, 0.1) {$\theta^*$};
						\node[red] at (0.18, 1.14) {$v = 0$};
						\draw[->, red, thick] (0.10, 1.09) to[out=220, in=75] (0.04, 1.01);
				\end{tikzpicture}
    \end{subfigure}
    \caption{Schematic view of the two Poincar\'e compactifications associated to CRN~\eqref{e:ExCRN1}. We plot $\theta=\arctan(u)$ and $v'=1-v$. Infinity lies on the red circle, and $\theta_*$ is the singular point.\\
		(Left: $\omega=(1,1)$; Right: $\omega=(1,1/2)$)}
		\label{fig:Agazzi25Comp2}
\end{figure}

\subsection{Second example}
We introduce the following CRN: 
\begin{equation}\label{e:ExCRN2}
	2S_1+S_2\stackrel{4}{\rightarrow} 4S_1,\quad S_1+3S_2\stackrel{5}{\rightarrow} 4S_2, \quad 5S_2\stackrel{1}{\rightarrow} 4S_2. 
\end{equation}
Here, $\cal{P}=\{(1,1/2)\}$, and for $\omega=(1,1/2)$, the polynomials of the Poincar\'e compactification are given by
\[
	F_\omega(u)=-4u+5u^3-u^5=-u(u^2-1)(u^2-4)\quad \text{and}\quad G_\omega(u)=8u-5u^3.
\]
There are two singular points: $u_1=1$ is repulsive and divergent, and $u_2=2$ is attractive and convergent.
 
The map on the Poincar\'e disk is given in Figure~\ref{fig:Example2}. 
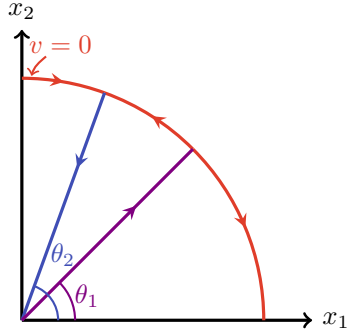
\begin{figure}[h!]
	\centering
	\begin{tikzpicture}[scale=3.2]
			\def\tone{70}  
			\def\ttwo{45}  
			\draw[->, very thick] (0, 0) -- (1.2, 0) node[right] {$x_1$};
			\draw[->, very thick] (0, 0) -- (0, 1.2) node[above] {$x_2$};
			\draw[red, very thick,
					postaction={decorate, decoration={markings,
							mark=at position 0.5 with {\arrow{stealth}}}}]
					(0,1) arc (90:\tone:1);
			\draw[red, very thick,
					postaction={decorate, decoration={markings,
							mark=at position 0.5 with {\arrowreversed{stealth}}}}]
					({cos(\tone)},{sin(\tone)}) arc (\tone:\ttwo:1);
			\draw[red, very thick,
					postaction={decorate, decoration={markings,
							mark=at position 0.5 with {\arrow{stealth}}}}]
					({cos(\ttwo)},{sin(\ttwo)}) arc (\ttwo:0:1);
			\draw[blue, very thick,
					postaction={decorate, decoration={markings,
							mark=at position 0.67 with {\arrowreversed{stealth}}}}]
					(0,0) -- ({cos(\tone)},{sin(\tone)});
			\draw[violet, very thick,
					postaction={decorate, decoration={markings,
							mark=at position 0.67 with {\arrow{stealth}}}}]
					(0,0) -- ({cos(\ttwo)},{sin(\ttwo)});
			\draw[blue, thick] (0.15,0) arc (0:\tone:0.15);
			\node[blue] at (0.17, 0.27) {$\theta_2$};
			\draw[violet, thick] (0.22,0) arc (0:\ttwo:0.22);
			\node[violet] at (0.27, 0.1) {$\theta_1$};
			\node[red] at (0.16, 1.14) {$v = 0$};
			\draw[->, red, thick] (0.10, 1.09) to[out=220, in=75] (0.04, 1.01);
		\end{tikzpicture}
		\caption{Schematic view of the Poincar\'e compactification of CRN~\eqref{e:ExCRN2} at $\omega=(1,1/2)$. We draw $\theta_i=\arctan u_i$.  }
	\label{fig:Example2}
\end{figure} 
\section{Main results and strategy of the proof}\label{s:Theorem}

\subsection{Main result}
The main result consists of a dissipation inequality, i.e., in showing that one of the inequalities required for the Lyapunov criterion (Foster-Lyapunov or Filonov's criterion, see Section~\ref{s:Lyapunov}) is verified. 

The assumptions made on the system are of two types: Assumptions~\ref{a:OmegaNeq1} and~\ref{a:OmegaEq1} concern the non-trivial compactifications: using Definition~\ref{d:CompSingPoint}, all \emph{attractive} singular points of the compactifications must be \emph{convergent}.
Assumptions~\ref{a:PFinite},~\ref{a:Delta2Neq0} and~\ref{a:DeltaNeq0_11} concern the \emph{drift regions}, i.e., the regions in which the system is well approximated by its leading order ODE. We restrict our theorem to the study of systems for which the leading order term in the drift region does not leave the $S_2$ coordinate unchanged. \emph{The generality of these assumptions as well as their interpretation are discussed in Section~\ref{s:ScaleAppl}.} 

The first assumption guarantees that the relevant compactifications have only finitely many singular points:
\begin{assumption}\label{a:PFinite}
	For $\omega\in (\cal{P}\cap((1,0), (1,1)))\cup\{(1,1)\}$, $|\cal{P}_\omega|<+\infty$. 
\end{assumption}

We will see that the behavior of the system is more complicated in the region associated to $\omega=(1, 1)$, where both species $S_1$ and $S_2$ evolve at the same timescale. The assumptions related to this order of magnitude are stated seperately. 

\subsubsection{Sub-diagonal regions}
We first state the assumptions for the regions associated to $\omega<(1,1)$. 
\begin{assumption}\label{a:OmegaNeq1}
	For $\omega\in \cal{P}$ with $(1,0)<\omega<(1,1)$, for all $u\in \cal{P}_\omega$, $F'_\omega(u)G_\omega(u)\neq 0$ and 
	\begin{equation*}
		\text{either}\quad  G_\omega(u)<0,\quad \text{or}\quad F'_\omega(u)>0.
	\end{equation*}
\end{assumption}
This assumption states that any singular point in the compactification is either repulsive, or attractive \emph{and} convergent. 
 
For the directions in Assumption~\ref{a:PFinite}, $F_\omega$ and $G_\omega$ are polynomials with $F_\omega\not\equiv0$. Therefore, for $t=0^+$ and $t=+\infty$, we can find some $\ell_{\omega}(t)\in \R\cup\{\pm \infty\}$ such that
\[
	\lim_{\substack{u\to t\\u>0}} \frac{G_\omega(u)}{F_{\omega}(u)}=\ell_\omega(t).
\]
The following assumption concerns the drift regions.   
\begin{assumption}\label{a:Delta2Neq0}
	For all $\omega\in \cal{P}$ such that $(1,0)<\omega<(1,1)$, the limits $\ell_\omega(0)$ and $\ell_\omega(+\infty)$ are finite. 
\end{assumption}
This guarantees that at the boundary of a compactification, the leading drift does not keep $x_2$ constant, which allows it to ``move to the next region of compactification''. Together with Assumption~\ref{a:DeltaNeq0_11} and the absence of ghost complexes, it ensures that for $(1,0)<\omega<(1,1)$ and $\omega\notin\cal{P}$, $\Delta_{2}^\omega(1,1)\neq 0$, see Lemma~\ref{l:BetweenCompact}. The interpretation on the graph of the CRN is that the drift coming from the dominant complex at the end of the compactification is not horizontal. 

\subsubsection{Diagonal region}
The diagonal region $\cal{H}_\omega$ with $\omega=(1,1)$ requires specific assumptions. 
When the compactification at $\omega=(1,1)$ is not trivial, the requirements are similar to Assumption~\ref{a:OmegaNeq1},  but with an additional inequality that bounds how much the norm of the process has increased while getting away from a repulsive line.
Here $b>0$ is the exponent in the boundary growth condition of Definition~\ref{d:b-dominance}.
  
\begin{assumption}\label{a:OmegaEq1}
	For $\omega=(1,1)$, for all $i\in \{1, \ldots, p\}$, $F'_\omega(u_i)G_\omega(u_i)\neq0$ and either 
	\[
		G_\omega(u_i)<0
	\]
	or 
	\[
		 F'_\omega(u_i)>0 \quad \text{and}\quad \frac{F'_\omega(u_i)}{G_\omega(u_i)}>b
	\]
\end{assumption}
Similarly to Assumption~\ref{a:Delta2Neq0}, to ensure that the drift in the sub-diagonal regions is not horizontal, we state the following assumption. 
\begin{assumption}\label{a:DeltaNeq0_11}
	For $\omega=(1,1)$, the limit $\ell_\omega(0)$ is finite. 
\end{assumption}

\subsubsection{Main theorem}
We can now state the main theorem:
\begin{theorem}\label{t:Main}
	Let $b\geq 1$ and $M>0$. If the CRN with mass action kinetics $(\cal{S}, \cal{C}, \cal{R}, \kappa)$ has no ghost complexes and verifies Assumptions~\ref{a:PFinite},~\ref{a:OmegaNeq1}, \ref{a:Delta2Neq0}, \ref{a:OmegaEq1} and~\ref{a:DeltaNeq0_11}, then we can find a function $V$ that verifies the $b$-dominance on the boundary of the state space, some constants $K>K'>0$ and $\beta>0$, $C>0$, and an integrable stopping time $\tau\neq 0$, such that for $x\in \cal{T}((1, 0), (1,1); K', M)$ with $\|x\|\geq K$,  
	\[
		\E_x\left(V(X(\tau))- V(x)\right)\leq -\beta\E_x(\tau), \quad \text{and}\quad \E_x(\tau)\leq C,
	\] 
	and $(X(t\wedge \tau))$ hasn't left $\cal{T}((1, 0), (1,1); K', M)$.
	Moreover, the dissipation is point-wise on the boundary states, i.e., we can find some $\omega>(1,0)$ such that, for $x\in \cal{T}((1, 0), \omega; K', M)$ with $\|x\|\geq K$, 
	\[
		\cal{A}(V)(x)\leq -\beta. 
	\]
\end{theorem}

We defer the discussion of the genericity of the assumptions to Section~\ref{s:ScaleAppl}. 

\subsection{Implications for positive recurrence}
Theorem~\ref{t:Main} shows dissipation of an energy function in the bulk of the state space, where both species are present in large numbers ($x_1, x_2\geq K'$). Combined with a dissipation on the boundary of the state space, this guarantees the positive recurrence of the process, using Lyapunov's criterion, see Section~\ref{s:Lyapunov} of the Appendix. 
The energy functions defined on the boundaries have to match the energy function we defined in the bulk. 

\begin{definition}
	For an energy function $V$, we say that the process is \emph{$V$-compatibly dissipative on the $S_2$-boundary} if we can find 
	\begin{itemize}
		\item some $\beta\in \R_{>0}$, $\gamma\in \R_{\geq 0}$ and $K_2\geq K'$
		\item some set $S_2\subset\Z_{\geq 0}^2$ that verifies
		\[
			\{z\in \Z_{\geq 0}^2: z_2\leq K_2\}\subset S_2\subset \{z\in \N^2: z_2\leq K_2+\gamma \log(1+z_1)\}
		\]
		\item some energy function $W$,  
		\[
			W: z\in \N^2\mapsto 
			\begin{cases}
				V(z) \quad\text{if   } z\notin S_2, \\
				W_2(z) \quad \text{if   } z\in S_2
			\end{cases}\, , 
		\]
	\end{itemize}
	and define a stopping time $\tau\leq \inf\{t>0: X_1(t)\leq X_2(t)\}$ such that 
	\begin{equation}\label{e:DissipBoundary}
	\forall z\in S_2, \quad \|z\|\geq K, \quad \E_z\left(W(X(\tau))\right)-W(z)\leq -\beta \E_z(\tau), 
	\end{equation}
	$\E(\tau)\leq C$ and for all $z\in \partial^{K_0}(S_2)$ with $\|z\|\geq K$, $\cal{A}(W)(z)\leq -1$. 
\end{definition}
Equation~\eqref{e:DissipBoundary} ensures dissipation near the boundary, while the final point-wise inequality ensures dissipation across the patching region.
The notion of \emph{compatibly dissipative on the $S_1$-boundary} is defined symmetrically.

\begin{corollary}
	If conditions of Theorem~\ref{t:Main} hold for the initial CRN and the CRN where the species are switched, and the process is also compatibly dissipative on the $S_2$-boundary and the $S_1$-boundary, then it is positive recurrent in its closed irreducible classes. 
\end{corollary}

The notion of $b$-dominance is introduced because the dissipation on the boundary sometimes requires a function at least of order $b$, the process being slowed down by boundary behaviors. See Section 5 of~\cite{LR23} for example. However, this study in the bulk shows that in some cases, the power $b$ can't be taken as large as we want, because of Assumption~\ref{a:OmegaEq1}.

The real number $b\in [1, +\infty)$ is fixed throughout. 

\begin{remark}
To keep the statements of the point-wise dissipation results simple, we do not explicitly distinguish states within one jump of the boundary of the regions considered. The functions and estimates remain valid for these states, up to a slight adjustment of the boundaries when needed. Boundary states are treated explicitly when they are required, i.e., in the patching results, Lemmas~\ref{l:GluyingJonathan} and~\ref{l:SmoothingGluing}.
\end{remark}

\subsection{Comments on the scale of application of the Assumptions}\label{s:ScaleAppl}
The main structural condition made on the CRN is that there is no lower-diagonal region where the leading drift is horizontal. This condition is required to have Assumptions~\ref{a:PFinite} for $\omega<(1,1)$, as well as Assumptions~\ref{a:Delta2Neq0} and~\ref{a:DeltaNeq0_11}.

This assumption can be relaxed, but we would need to carry out another careful classification of the different possible vector fields in the region where this drift dominates. We decided to focus on the classification related to the Poincar\'e compactifications, and therefore chose to make this strong assumption for the drift regions to avoid any additional complexity. 
The counterpart of this choice is that the original CRN from~\cite{Agazzi25} doesn't verify this assumption because the drift from reaction $4S_1\to \emptyset$, dominant for $\omega<(1,1/2)$, is horizontal. However, a slightly modified version of it, presented in Relation~\eqref{e:ExCRN1}, verifies the Assumption. 

For a fixed reaction graph, Assumption~\ref{a:PFinite} at $\omega=(1,1)$ holds for almost every choice of positive reaction rate constants. More generally, whenever $F_\omega$ is not identically zero, its positive roots are generically simple. 
\begin{proposition}\label{p:POmegaFinite}
	Fix a reaction graph. There exists a set $\cal{K}_0\subset\R_{>0}^{\cal{R}}$ of Lebesgue measure $0$ such that, for every $\kappa\in\R_{>0}^{\cal{R}}\setminus\cal{K}_0$,
	\begin{itemize}
		\item  $|\cal{P}_{(1,1)}|<+\infty$.  
		\item For every $\omega\in (\cal{P}\cap((0,1), (1,1)))\cup\{(1,1)\}$ such that $|\cal{P}_\omega|<+\infty$, each positive root is simple, i.e., $F'_\omega(u_*)\neq0$ for all $u_*\in\cal{P}_\omega$. 
	\end{itemize}
\end{proposition}
The proposition follows from standard arguments for polynomial families depending on the reaction rate constants. 

The main assumption that can be relaxed is the inequality $G_\omega(u_*)\neq 0$ (in Assumption~\ref{a:OmegaEq1} and~\ref{a:OmegaNeq1}). 
Making this assumption means that the drift towards smaller states or infinity is determined by the dominant reactions. In particular, CRN using lower order reactions to dissipate do not verify this inequality. Because of this, the theorems can't be applied to some simple weakly reversible CRNs (CRN for which each reaction belongs to a directed cycle, see~\cite{Anderson1} for example) such as the following: 
\[
	S_1+S_2\rightleftharpoons 2S_1, \quad S_2\rightleftharpoons \emptyset. 
\]
For this particular network, with the constants of reactions equal to $1$, for $\omega=(1,1)$, $F_\omega(u)=-(u-1)(u+1)(1+u^2)$ and $G_\omega(u)=-(u-1)^2$. The root $u_*=1$ verifies $F'_\omega(u_*)<0$ but $G_\omega(u_*)=0$. To see the dissipation, we have to take into account the next order in the rate of reaction.

One assumption that, to our knowledge, \emph{can't} be relaxed is: if $F_\omega'(u_i)<0$, stability requires $G_\omega(u_i)\leq 0$. In other words, if an invariant curve for the deterministic process is \emph{attractive}, it has to bring the system back towards a compact set. Otherwise, intuitively, if  $F_\omega'(u_i)<0$ and $G_\omega(u_i)> 0$, the system is attracted by the curve and will follow it to infinity (note that in this article, no result properly confirming this intuition is shown). 

We restrict the study to the states $x_2\leq M x_1$ because the problem is symmetrical with respect to a switch of the two coordinates. This is absolutely not restrictive in terms of application, see Remark~\ref{r:SwitchSpecies}. 

\subsection{Examples}
Example~\eqref{e:ExCRN1} verifies Assumption~\ref{a:OmegaNeq1}, and Assumption~\ref{a:OmegaEq1} for $b$ small enough. When $\kappa_1=\kappa_2=1$, we have $u_*=1$, $F'_\omega(u_*)=4$ and $G_\omega(u_*)=2$, and the theorem gives us a function that will verify the $b$-dominance criterion at the boundary of the state space, with $1\leq b<2$. The condition required for the sub-diagonal drift region is also verified. 

Example~\eqref{e:ExCRN2} verifies Assumption~\ref{a:OmegaNeq1}.
The diagonal is a drift region, where the only reaction that dominates is $5S_2\to 4S_2$, so Assumption~\ref{a:DeltaNeq0_11} is verified. However, when switching the two species, as mentioned in Remark~\ref{r:SwitchSpecies}, on the above-diagonal region, the dominant reaction is also $5S_2\to 4S_2$ which leads to a vertical drift, in contradiction with Assumption~\ref{a:DeltaNeq0_11} after switching the species. 

\subsection{Strategy of the proof}

The essential part of the proof relies on what happens in the regions where a non-trivial Poincar\'e compactification takes place. 
The dissipation in the rest of the state space is captured by a point-wise estimate, meaning that choosing the right energy function (which in the sub-diagonal drift regions is a power of a linear function), we can show rather easily that it is dissipated on average after one jump of the process, see Lyapunov's criterion in Section~\ref{s:Lyapunov}. 

In this part, we describe what happens in these regions, anticipating the results shown in Section~\ref{s:R11}, \ref{s:RNeq11} and~\ref{s:A}. 

We fix an $\omega\in \Omega_\leq$, $\omega=(1, \gamma)$ such that $p:= |\cal{P}_\omega|\geq 1$ and introduce its Poincar\'e compactification as in definition~\ref{d:CompSingPoint}: the functions $F_\omega$, $G_\omega$, and the ordered positive roots $u_1<\ldots<u_p$ of $F_\omega$, with $u_0=0$ and $u_{p+1}=+\infty$.  
We study the system in a region $\cal{H}_\omega(u_1/2, 2u_p)$, that contains every singular axis. 
The key to our approach is that the behavior of the system is almost \emph{entirely contained in the approximation of the generator} to the right order of magnitude in $r$.
Here $r=\|x\|$ when $\gamma=1$ and $r=x_1$ when $\gamma<1$. 
 
More precisely, we will see that following the steps from the compactification presented in Section~\ref{s:Poincare}, changing coordinates and making a time change, the generator takes the form 
\[
	\widehat{\cal{A}}_\omega(f)(u, r)\approx F_\omega(u)\frac{\partial f}{\partial u}+ r^\gamma G_\omega(u)\frac{\partial f}{\partial r} + \frac{\sigma_\omega(u)^2}{2r^\gamma}\frac{\partial^2 f}{\partial u^2} 
\]
where we recognize the functions $F_\omega(u)$ and $G_\omega(u)$ given by the Poincar\'e compactification (and $\sigma_\omega^2(u)$ is nonnegative and bounded on each compact interval of $u$).

To better see the different parts of the compactification, we further divide the cone $\cal{H}_\omega(u_1/2, 2u_p)$ into sub-cones, each containing one of the singular curves. Each set is centered at a $u_*=u_i$: for some $u_-, u_+$ such that $u_{i-1}<u_-<u_*<u_+<u_{i+1}$, we study the system in $\cal{H}_\omega(u_-, u_+)$. 

We can now push the approximation further and expand around $u_*$ where $F_\omega(u_*)=0$:  
\begin{equation}\label{e:SketchApproxA}
	\widehat{\cal{A}}_\omega(f)(u, r)\approx (u-u_*)F'_\omega(u_*)\frac{\partial f}{\partial u}+ r^\gamma G_\omega(u_*)\frac{\partial f}{\partial r} + \frac{\sigma_\omega(u_*)^2}{2r^\gamma} \frac{\partial^2 f}{\partial u^2}.  
\end{equation}
Using this approximation, we can understand the behavior of the system and choose the strategy we will use to show the dissipation. 
The sign of $F'_\omega(u_*)$ tells us if the curve $u_*$ is attractive (and the process will stay in the region) or repulsive (and the process will eventually exit the region). 
The sign of $G_\omega(u_*)$ tells us if the process goes towards larger or smaller states. If the curve $u_*$ is attractive, by Assumptions~\ref{a:OmegaEq1} and~\ref{a:OmegaNeq1}, the process has to go towards smaller states. If it is repulsive however, the process can grow until it exits the region. 
The second order term in derivative of $u$ stands for the noise. The scaling in $r^{-\gamma}$ means that when $|u-u_*|$ is large enough, the first order in derivative dominates in the infinitesimal generator, and the noise can be neglected. But when $|u-u_*|$ is small, more precisely when it is of the order $r^{-\gamma/2}$, the noise and the linearized drift both contribute at the leading order.

Three different cases are dealt with separately. We give here an overview of the proof. 
\begin{enumerate}
	\item When $F'_\omega(u_*)<0$ (and $G_\omega(u_*)<0$ by assumption), the region is ``stable'': the process is attracted by the curve at $u_*$, and then follows it to come back to smaller states. In these regions, showing some dissipation will not be difficult, since the process naturally decreases its norm. We show a stronger result than a point-wise dissipation, to facilitate the patching of the different regions, and show that after ``a while'', $r(t)\leq \eps r(0)$. See Proposition~\ref{p:DissipAC}. One main difficulty is that when $\gamma<1$, since $\dot{r}(t)\approx  G_\omega(u_*) r(t)^\gamma$, the norm doesn't decrease on the timescale $t\mapsto t$. We have to accelerate the time to see $r(t)$ reach $\eps r(0)$. More precisely, we look at the process on the timescale $t\mapsto t r(0)^{1-\gamma}$ where the decrease of $r(t)$ happens. The main difficulty is to ensure that $(u(t))$ stays sufficiently close to $u_*$ on this timescale. The proof uses the fast restoring drift that brings $u(t)$ towards $u_*$ to compensate its martingale fluctuations. 
	\item When $F'_\omega(u_*)>0$, the region is ``unstable'': the process is pushed away from the curve at $u_*$. In this case, even if the curve is divergent (i.e., if $G_\omega(u_*)>0$) we can show dissipation. When $|u-u_*|$ is large enough, the noise can be neglected and the process approximately behaves deterministically, $(u(t)-u_*)\approx (u(0)-u_*)e^{F_\omega'(u_*) t}$. In this case, taking an energy function that decreases when $|u-u_*|$ increases allows a point-wise dissipation. 
	However, when the noise can't be neglected, we have to study further the behavior of the deviation to $u_*$ to show some dissipation. We will show that after rescaling, this deviation behaves like a divergent Ornstein-Uhlenbeck process, which confirms the fact that the process escapes a neighborhood of the axis, even when starting on it. 
	\item Finally, the approximation given in Relation~\eqref{e:SketchApproxA} omits terms that must be retained when $\gamma<1$. Some other terms are present between the first order terms and the noise term, which adds a difficulty to this case, see Lemma~\ref{l:Defh}.
	Moreover, as mentioned in the case $F'_\omega(u_*)<0$, when $\gamma<1$, the norm moves at a slower timescale than $u(t)$. This time it makes things easier, because the norm of the process hasn't had time to grow before it exits the region. However, when $\gamma=1$, we have to make sure that the norm hasn't increased too much at the exit time. 
	For these two reasons, we separate the study of the cases $\gamma=1$ and $\gamma<1$ for the repulsive axis, see Sections~\ref{s:R11} and \ref{s:RNeq11}.
\end{enumerate}

\section{Repulsive curve, $\omega=(1,1)$}\label{s:R11}
The goal of this section is to show the following proposition: 
\begin{proposition}\label{p:RepLine11}
	Let $b>0$. Assume $\omega=(1,1)$, and that $u_*=u_i\in \cal{P}_\omega$ is such that 
	\begin{equation}\label{e:RepLine11}
		F'(u_*)>0\quad \text{and}\quad F'(u_*)>bG(u_*).
	\end{equation}
	Then for any choice of $u_-, u_+$ such that $u_{i-1}<u_-<u^*<u_+<u_{i+1}$, we can find an energy function $V_{u^*}^\omega$ that verifies the $b$-dominance at the boundary of the set (see Definition~\ref{d:b-dominance}), an integrable stopping time $\tau$ and some constants $\beta, R>0$ such that, \\
	for all $x\in \cal{H}_\omega(u_-, u_+)$ with $\|x\|\geq R$, 
	\[
		\E_x\left( V_{u^*}^\omega(X(\tau))-V_{u^*}^\omega(x)\right)\leq -\beta\E_x(\tau), 
	\]
	and $X(t\wedge\tau)\in \cal{H}_\omega(u_-, u_+)$ for all $t\geq 0$. 
\end{proposition}

The proof is carried out in several steps. In a first part, we give the Taylor expansion of the generator valid in the cone, and apply it to the function $f_a^{u_*}$ used throughout the proof: for $a>0$, we define for $(u, r)\in (0, +\infty)^2$ with $u\neq u_*$, 
\begin{equation}\label{d:fauStar}
	f_{a}^{u_*}(u, r):= \frac{r^b}{|u-u^*|^a}. 
\end{equation}
Instead of working with $v\to 0$ as in the compactification, we will use the variable $r(x)=\|x\|_2\to +\infty$, but the two approaches are equivalent. 
Two separate proofs are given according to the sign of $G(u_*)$, which gives the leading growth rate of the norm $r(x)$. When $G(u_*)<0$, the decrease of $r$ gives point-wise dissipation near the line; when $G(u_*)>0$, the inner cone requires a stopped estimate based on the Ornstein-Uhlenbeck limit.

\subsection{Change of variables and infinitesimal generator}
\begin{definition}[Change of coordinates and generator]
	We introduce the following change of coordinates:
	\begin{equation}\label{e:ChangeVar11}
	(u, r): (x_1, x_2)\mapsto \left(\frac{x_2}{x_1}, \sqrt{x_1^2+x_2^2}\right). 
\end{equation}
	Reversely, we define the function $\underline{x}(u,r)$ as 
	\begin{equation}\label{e:RevChangeVarNot11}
		\underline{x}: (u, r)\mapsto \left(\frac{r}{\sqrt{1+u^2}}, \frac{ru}{\sqrt{1+u^2}}\right). 
	\end{equation}
	We define the generator $\widehat{\mathcal A}_\omega$ of the process $(u(X(t)), r(X(t)))$ after a time-change. For any function $f\in \cal{C}^3$, and $(u, r)\in (0, +\infty)^2$, we set 
	\begin{multline}\label{e:GenUR11}
			 \widehat{\cal{A}}_\omega f(u,r):=\frac{(1+u^2)^{(\rho_0+1)/2}}{r^{\rho_0-1}} \\
			 \sum_{\nu\in\mathcal R}\lambda_\nu(\underline{x}(u, r))\left[f\left(u(\underline{x}(u, r)+\xi_\nu),r(\underline{x}(u, r)+\xi_\nu)\right)-f(u,r)\right]. 
	\end{multline}
\end{definition}
Because $\rho_0\geq 1$ (see Relation~\eqref{s:RhoDef}), this time change \emph{slows down} the system. Therefore, a dissipation at this timescale can easily be translated to original timescale of the process. See Section~\ref{s:TimeChange} for the rigorous argument.

The next proposition gives an approximation of this ``local'' generator in the limit ${r\to +\infty}$, with $u$ in a compact set. The first order involves the functions $F_\omega(u)$ and $G_\omega(u)$ given by the Poincaré compactification, the second order involves a function $D_\omega(u)$ defined as 
\begin{equation*}
	D_\omega(u)=\sum_{r=y^-\to y^+\in \cal{R}_\omega}\kappa_ru^{y_2^-}((y_2^+-y_2^-)-u(y_1^+-y_1^-))^2,
\end{equation*}

This proposition is the central result on which the proof of dissipation holds. We will see that this approximation of the generator contains all the information we need on the behavior of the process in this region of the space. Applied to the right function, it will be enough to show some point-wise dissipation when far enough from the repulsive line. Near the line (in an inner cone of width $1/\sqrt{r}$), it will be used to show that the deviation to the axis converges in distribution to a divergent Ornstein Uhlenbeck. 
\begin{proposition}\label{p:ReducedGenerator11}
Let $u_-<u_+\in (0, +\infty)$. 
For every $f\in C^3((0,+\infty)\times(1, +\infty))$, for $u\in [u_-, u_+]$ and $r\geq 1$, 
\begin{multline}\label{e:IGReduced}
    \widehat{\cal{A}}_\omega f(u,r)
		=F_\omega(u)\frac{\partial f}{\partial u}(u,r)
		+rG_\omega(u)\frac{\partial f}{\partial r}(u,r)\\
		+\frac{(1+u^2)^{3/2}}{2r}D_\omega(u)\frac{\partial^2 f}{\partial u^2}(u,r)+\mathcal E_f(u,r),
\end{multline}
where $\mathcal E_f$ contains correction terms, of lower order when $r\to +\infty$.

It can be bounded as follows: there exists $C_K>0$ such that 
\begin{multline*}
    |\cal{E}_f(u, r)|\leq C_K\big(
		r^{-1}\left|\frac{\partial f}{\partial u}\right|
    +\left|\frac{\partial f}{\partial r}\right|
		+r^{-2}\left|\frac{\partial^2 f}{\partial u^2}\right|
    \\
		+\sum_{\substack{(i, j)\in\{0, \ldots 3\}^2\setminus\{(2,0)\}:\\ 2\leq i+j\leq 3}} r^{1-i}\left|\frac{\partial^{i+j} f}{\partial u^{i}\partial r^{j}}\right|
		\big).
\end{multline*}
\end{proposition}

\begin{proof}
	This is simply a careful Taylor expansion of the generator. Let $f(u,r)$ be a $\cal{C}^3$ function. 
	
	For $\xi\in \R^2$, $\|\xi\|\leq K_0$, 
	\begin{align*}
		&u(x(u, r)+\xi)-u=  (\xi_2-\xi_1u)\frac{\sqrt{1+u^2}}{r}+\frac{\eps_u(u,r)}{r^2}\\
		&r(x(u,r)+\xi)-r= \frac{\xi_1+\xi_2u}{\sqrt{1+u^2}} +\frac{\eps_r(u,r)}{r}
	\end{align*}
	where $\eps_u(u,r)$ and $\eps_r(u,r)$ are bounded on $(u, r)\in [u_-, u_+]\times [1, +\infty)$. 
	
	Therefore, using Taylor's expansion, we get 
	\begin{multline*}
		\sqrt{1+u^2}r[f(u(x(u, r)+\xi), r(x(u,r)+\xi))-f(u, r)]\\
		 = (\xi_1+\xi_2u)r\frac{\partial f}{\partial r}+ (1+u^2)(\xi_2-\xi_1u)\frac{\partial f}{\partial u}\\
		 + \frac{1}{2r}(1+u^2)^{3/2}(\xi_2-\xi_1u)^2 \frac{\partial^2 f}{\partial u^2} + \cal{E}_{f, \xi}(u,r) 
	\end{multline*}
	where for some $C>0$ large enough, we have 
	\begin{multline*}
		|\cal{E}_{f, \xi}(u,r)|
		\leq C\left(\frac{1}{r}\left|\frac{\partial f}{\partial u}\right| + \left|\frac{\partial f}{\partial r}\right| 
		+\frac{1}{r^2}\left|\frac{\partial^2 f}{\partial u^2}\right|\right) \\
		+C\sum_{\substack{(i,j)\in\{0, \ldots 3\}^2\setminus\{(2,0)\}:\\ 2\leq i+j\leq 3}} r^{1-i}\left|\frac{\partial^{i+j} f}{\partial u^i\partial r^{j}}\right|
	\end{multline*}
	For a reaction $\nu=y_-\to y_+$, we define the rate after the time change:
	\[
		\widehat{\lambda}_\nu(u, r):= \frac{(1+u^2)^{(\rho_0+1)/2}}{r^{\rho_0-1}} \lambda_\nu(x(u, r))
	\]
	If $\nu\in \cal{R}_\omega$ and $\nu'\in \cal{R}\setminus \cal{R}_\omega$ it verifies 
	\[
		\widehat{\lambda}_\nu(u, r)= \kappa_\nu r\sqrt{1+u^2} u^{y_2^-}(1+\frac{\eps_\nu(u, r)}{r})\quad \text{and}\quad \widehat{\lambda}_{\nu'}(u,r)=\eps_{\nu'}(u, r),
	\]
	with $\eps_\nu(u,r)$ and $\eps_{\nu'}(u,r)$ bounded.
	Combining these results, we get the approximation. 
\end{proof}

We now apply this approximation to the function $f_a^{u_*}$ from Definition~\ref{d:fauStar}.
\begin{lemma}
If $u_-<u_*<u_+\in (0, +\infty)$ and $a>0$, uniformly for $u\in [u_-, u_+]\setminus\{u^*\}$, when $r\to +\infty$, provided that $r|u-u_*|\to\infty$, 
\begin{multline}\label{e:InfGenfau11}
    \widehat{\cal{A}}_\omega f_a^{u_*}(u,r)=\frac{r^b}{|u-u^*|^a}
		\bigg[-a\frac{F_\omega(u)}{u-u^*}+bG_\omega(u)  \\
		+a(a+1)\frac{(1+u^2)^{3/2}}{2r(u-u^*)^2}D_\omega(u)
		+O\left(\frac{1}{r|u-u^*|}+\frac{1}{r^2|u-u^*|^3}\right)\bigg]
\end{multline}
holds. 
\end{lemma}

\begin{proof}
	This is a direct application of Proposition~\ref{p:ReducedGenerator11}. 
\end{proof}

\subsection{Repulsive, convergent line}
This case is the easiest. We can use the decrease of the norm (i.e. the fact that $G_\omega(u_*)<0$) to show the dissipation near $u_*$. 

For a choice of $a>0$, $\delta>0$, we introduce the candidate energy function on $\cal{H}_\omega(u_-, u_+)$: 
\begin{equation*}
	\widehat{V}_{u_*}:(u,r)\mapsto 
	\begin{cases}
		\frac{r^b}{\delta^a} \quad &\text{for   } |u-u^*|<\delta\\
		f_{a}^{u_*}(u,r)\quad &\text{for   } |u-u^*|\geq \delta
	\end{cases}\, . 
\end{equation*}

The following proposition shows that when $G_\omega(u_*)<0$, i.e. when the invariant line is convergent, this function is point-wise dissipated. 
\begin{proposition}\label{p:DissipRC11}
	Let $u_*=u_i$ as in Proposition~\ref{p:RepLine11}, and fix some $(u_-, u_+)$ such that $u_{i-1}<u_-< u_*<u_+< u_{i+1}$. 
	If $G(u_*)<0$, for $\delta, \beta>0$ small enough and $a>0$ large enough,
	the following inequality holds for $u\in[u_-, u_+]$ and $r$ large enough: 
	\[
		\widehat{\cal{A}}_\omega \widehat{V}_{u_*}(u,r)\leq -\beta r^b.
	\]
\end{proposition}

\begin{proof}
The proof simply uses the results of Proposition~\ref{p:ReducedGenerator11}. 

Because of the assumption on $F_\omega'(u_*)$ and because $F_\omega$ has no root except $u_*$ in $[u_-, u_+]$, 
\[
	\min\left\{\frac{F_\omega(u)}{(u-u_*)}, \, u\in [u_-, u_*-\delta]\cup[u_*+\delta, u_+]\right\}>0.
\]
Therefore, there exists $a>0$ such that
\[
	\min\left\{a\frac{F_\omega(u)}{(u-u_*)} -bG_\omega(u), \, u\in [u_-, u_*-\delta/2]\cup[u_*+\delta/2, u_+]\right\}>0. 
\]
Choosing $\beta>0$ small enough, Equation~\eqref{e:InfGenfau11} and the definition of $a$ lead to 
\[
	\widehat{\cal{A}}_\omega(f_a^{u_*})(u,r)\leq -\beta r^b, 
\] 
for $u\in [u_-, u_*-3\delta/4]\cup[u_*+3\delta/4, u_+]$ and $r$ large enough.\\
Since $G_\omega(u_*)<0$, there exist a $\delta>0$ such that, for all $u$ such that $|u-u_*|<2\delta$, $G_\omega(u)<G_\omega(u_*)/2<0$. 
Using Equation~\eqref{e:IGReduced}, for $|u-u_*|<4\delta/3$,
\[
	\widehat{\cal{A}}_\omega((u,r)\mapsto r^b/\delta^a)(u,r)\leq -\beta r^b, 
\] 
where $\beta$ can be chosen smaller if needed. 
We conclude using Lemma~\ref{l:GluyingJonathan} of Section~\ref{Tools}
\end{proof}

\subsection{Repulsive, divergent line}
\subsubsection{Point-wise dissipation in the outer cones}
For a choice of parameters $(\delta, a, a', \eta_0)$ such that $\delta<|u_*-u_+|\wedge |u_*-u_-|$, $a'>a>1$ and $\eta_0>0$, we define a candidate energy function $\widehat{V}_{u_*}$ as 
\begin{equation*}
		\widehat{V}_{u_*}:(u,r)\mapsto 
		\begin{cases}
			\delta^{a'-a}f_{a'}^{u_*}(u,r) \quad &\text{for   } u\in [u_-, u_+]: \, |u-u_*|\geq \delta,\\
			f_{a}^{u_*}(u,r)\quad &\text{for   } u\in [u_-, u_+]: \, \frac{\eta_0}{2\sqrt{r}}\leq|u-u_*|<\delta ,\\
			2^a\eta_0^{-a}r^{b+a/2}\quad &\text{for   } u\in [u_-, u_+]: \,  |u-u_*|\leq \frac{\eta_0}{2\sqrt{r}},
		\end{cases}
\end{equation*}
where $f_{a}^{u_*}$ is defined in Relation~\eqref{d:fauStar}.

It is defined in three steps because of the following reasons: first, we need it to decrease when $|u-u_*|$ increases, but to be bounded at $u_*$, which leads to the definition in the inner cone, when $|u-u_*|\leq \frac{\eta_0}{2\sqrt{r}}$. Then, we use a function of the form $f^{u_*}_a$. In the inner cone, this function is of order $r^{b+a/2}$. Because the deviation from the axis takes a while, while it deviates the norm increases, we want the function in the cone $|u-u_*|<\delta$ to have the smallest $a>1$ possible. But then the point-wise dissipation on the side of the set may fail. So we first guarantee dissipation in the small cone with a small $a>1$, then take an $a'\geq a$ as large as needed to guarantee the point-wise dissipation when $|u-u_*|\geq\delta$. 

The following proposition shows that there exist $(\delta, a, a', \eta_0)$ such that $\widehat{V}_{u_*}$ is point-wise dissipated by the process associated to the generator $\widehat{\cal{A}}_\omega$ \emph{outside an inner cone}. 
\begin{proposition}\label{p:DissipRD11}
	In the setting of Proposition~\ref{p:RepLine11}, if $G_\omega(u_*)>0$, we can find parameters $(\delta, a, a', \eta_0)$ and a $\beta>0$ such that, for $u\in[u_-, u_+]$, $|u-u_*|\geq \eta_0 r^{-1/2}$ and $r$ large enough: 
	\[
		\widehat{\cal{A}}_\omega \widehat{V}_{u_*}(u,r)\leq -\beta r^b.
	\]
\end{proposition}

\begin{proof}
Let $\eps>0$. 
Using the last inequality in Assumption~\ref{a:OmegaEq1}, choosing $a>1$ small enough so that 
\[
	\frac{F'(u_*)}{G(u_*)}>\frac{b}{a}, 
\]
 we can set $\eps>0$ small enough such that
\[
	-aF'_\omega(u^*) +bG_\omega(u^*)<-\eps.
\]
We can choose $\delta$ small enough so that, 
\begin{equation*}
	\sup_{u: \, |u-u_*|\leq \delta}\left|\frac{F_\omega(u)}{(u-u^*)}-F'_\omega(u^*)\right|<\eps/2. 
\end{equation*}
Combining the above relations with the approximation~\eqref{e:InfGenfau11}, for $u\in [u_-, u_+]\setminus\{u_*\}$ such that $|u-u_*|\leq \delta$,  
\[
	\frac{\widehat{\cal{A}}_\omega(f^{u_*}_a)(u,r)}{f_a^{u_*}}\leq -\frac{\eps}{2} + a(a+1)\frac{(1+u^2)^{3/2}}{2r(u-u^*)^2}D_\omega(u)
\]
Also note that when $\eta_0 r^{-1/2}\leq |u-u^*|\leq \delta$, for some $C>0$,  
\[
	f^{u_*}_a(u,r)\geq C r^b.
\]
Therefore, when $\eta_0/\sqrt{r}\leq |u-u^*|\leq \delta$ with $\eta_0>0$ large enough, 
\[
	\widehat{\cal{A}}_\omega(f^{u_*}_a)(u,r)\leq -\frac{\eps C}{4}r^b.
\]

Using the same arguments as in the proof of Proposition~\ref{p:DissipRC11}, we can define $a'>a$ such that the dissipation holds on the exterior cones, i.e. for $u\in [u_-, u_+]\setminus[u_*-3\delta/4, u_*+3\delta/4]$. 

To patch the regions together (excluding the inner cone for now), i.e. to show the dissipation inequality for states with $u(x)\approx u_*\pm \delta$, we use Lemma~\ref{l:GluyingJonathan} of Section~\ref{Tools}. 
\end{proof}

\subsubsection{The inner cone}

\begin{proposition}\label{p:DissipRDFilo11}
	Keeping the notations of Proposition~\ref{p:DissipRD11},
	there exist some $T>0$, $\eta_1>\eta_0$ and $R_0\geq 1$ such that, 
	for any initial point $(u,r)$ with $r\geq R_0$ and $|u-u_*|\sqrt{r}\leq \eta_0$, 
	\[
		\E_{(u,r)}\left(\widehat{V}_{u_*}(U(\tau\wedge T), R(\tau\wedge T))- \widehat{V}_{u_*}(u, r) \right)\leq -1, 
	\]
	where $(U(t), R(t))$ is the Markov process defined by the infinitesimal generator $\widehat{\cal{A}}$, see Relation~\ref{e:GenUR11}, and $\tau$ is the stopping time defined as 
	\begin{equation*}
		\tau:= \inf\{t\geq 0: |U(t)-u_*|\sqrt{R(t)}\geq \eta_1\}. 
	\end{equation*}
\end{proposition}
The proof of this proposition relies on several results. 
The first and most important one shows that the deviation from the repulsive line behaves like a divergent Ornstein Uhlenbeck process in the limit $r\to +\infty$. 
More precisely, introducing the function 
\begin{equation*}
	\ell:(u,r)\mapsto (u-u_*)\sqrt{r},
\end{equation*}
for the convergence in distribution, the process $(\ell(U(t), R(t)))$ starting near the line (when $|\ell(u_0, r_0)|\leq \eta_0$) converges in distribution to the process $(L(t))$ solution of the SDE 
\begin{equation}\label{e:OU11}
	\diff L(t)= \left(F'_\omega(u_*)+\frac{G_\omega(u_*)}{2} \right)L(t)\diff t + (1+u_*^2)^{3/4}\sqrt{D_\omega(u_*)}\diff W_t, 
\end{equation}
with $(W_t)$ a standard Wiener process. 

For technical reasons, the convergence is only shown up to the exit time of $[-\eta_1, \eta_1]$ by $\ell(U(t), R(t))$. 
For a process $(Z(t))$ and some $\eta_1>\eta_0$, we introduce the stopping time 
\begin{equation}\label{e:DefTauZ}
	\tau_{Z}:= \inf\{t\geq 0: \, |Z(t)|\geq \eta_1\}. 
\end{equation}

\begin{lemma}\label{l:OU11}
	Let $(u_N)\in \R^{\Z_{\geq 0}}$ such that $\ell(u_N, N)\to \ell_0$ for some $\ell_0\in [-\eta_0, \eta_0]$. 
	Let $(U^N(t), R^N(t))$ be the Markov process with infinitesimal generator given by Relation~\ref{e:GenUR11}, starting from $(u_N, N)$. We set $L^N(t)=\ell(U_N(t), R_N(t))$. 
	Then for $T>0$, for the convergence in distribution, 
	\begin{multline*}
		\lim_{N\to +\infty}\left(\left(L^N(t\wedge \tau_{L^N}), \frac{R_N(t\wedge \tau_{L^N})}{N}\right)_{t\in [0, T]}, \tau_{L_N}\wedge T\right) \\
		= ((L(t\wedge \tau_L), e^{G_\omega(u_*)(t\wedge \tau_L)})_{t\in [0, T]}, \tau_L\wedge T),
	\end{multline*}
	where $(L(t))$ is a process starting from $\ell_0$ and solution of~\eqref{e:OU11}. 

	Moreover, for $\alpha\geq 1$, for any $\eps>0$, we can find $N_0\geq 0$ such that, for all $N\geq N_0$,
	\begin{equation}\label{e:OU11Lim2}
		\E_{(u_N, N)}\left(\sup_{t\in [0, T\wedge \tau_{L^N}]}\left|\frac{R_N(t)^\alpha}{N^\alpha}-e^{\alpha G_\omega(u_*)t}\right|\right)<\eps.
	\end{equation}
\end{lemma}	

\begin{proof}
	See the Appendix, Section~\ref{s:ProofOU11}.
\end{proof}

The second result we need is some bounds on the exit time of $[-\eta_1, \eta_1]$ by the divergent Ornstein-Uhlenbeck process. We state the Lemma in a general setting because it will be used with different constants when dealing with the face $\omega\neq (1,1)$.  

\begin{lemma}\label{l:ConvOU}
	Let $\alpha, \sigma>0$, and $(L(t))$ be a divergent Ornstein-Uhlenbeck process starting at $\ell_0$, and solution of the SDE
	\[
		\diff L(t)= \alpha L(t)\diff t + \sigma \diff W(t)
	\]
	where $(W(t))$ is a standard Brownian motion. 
	
	For $\ell_1>|\ell_0|$, we introduce 
	\begin{equation*}
		\tau_L:= \inf\{t\geq 0: |L(t)|\geq \ell_1\}. 
	\end{equation*}
	
	Then, there exist $C_0>0$ independent of $\ell_0$ such that, for $t>0$, 
	\begin{equation}\label{e:ConvOUeq1}
		\P_{\ell_0}(\tau_L>t)\leq \frac{\ell_1 C_0}{\sqrt{e^{2\alpha t}-1}}, 
	\end{equation}
	and for any $0<\alpha'<\alpha$, there exists $C_L(\alpha')$ independent of $\ell_1$ such that
	\begin{equation}\label{e:ConvOUeq2}
		\E_{\ell_0} \left( e^{\alpha'\tau_L}\right)<C_L(\ell_1+1). 
	\end{equation}
\end{lemma}

\begin{proof}
	For $t\geq 0$, $L(t)$ is normally distributed with mean value $\ell_0e^{\alpha t}$ and variance $\sigma^2(e^{2\alpha t}-1)/2\alpha$. 
	Therefore, for $t>0$, 
		\[
			\P_{L(0)=\ell_0}(\tau_L>t)\leq \P_0(\sup_{s\leq t}|L(s)| <\ell_1)\leq \P_0(|L(t)|<\ell_1)\leq \frac{\ell_1 C_0}{\sqrt{e^{2\alpha t}-1}}, 
		\]	
		where $C_0=2\sqrt{\alpha}/(\sigma\sqrt{\pi})$, and in the last inequality we used the explicit expression of the distribution of $L(t)$. This leads to
	\begin{align*}
		\E_0(e^{\alpha'\tau_L})&=\int_0^{+\infty}\P_0(e^{\alpha'\tau_L}>u)\diff u\\
			&\leq 2+ \int_2^{+\infty} \frac{\ell_1 C_0}{\sqrt{u^{2\alpha/\alpha'}-1}}\diff u\\
			&\leq 2+\ell_1C_L. 
	\end{align*}. 
\end{proof}

\begin{proof}[Proof of Proposition~\ref{p:DissipRDFilo11}]
	We introduce 
	\[
		\alpha:=F'_\omega(u_*)+\frac{G_\omega(u_*)}{2}\quad\text{and}\quad \alpha':= (b+\frac{a}{2})G_\omega(u_*)
	\]
	By assumption (see Equation~\eqref{e:RepLine11}), we can choose $a>1$ sufficiently close to 1 such that $\alpha>\alpha'$. 
	
	We introduce $C_L$ as in Relation~\eqref{e:ConvOUeq2} (with $\alpha'$ given above), and set $\eta_1>\eta_0$ large enough so that 
	\[
		\frac{C_L(\eta_1+1)}{\eta_1^a}<\frac{1}{4\eta_0^a}.
	\]
	We introduce $C_0$ as in Relation~\eqref{e:ConvOUeq1}, and set $T>0$ large enough so that 
	\[
		\frac{\eta_1 2^a C_0 e^{\alpha'T}}{\sqrt{e^{2\alpha T}-1}}<\frac{1}{4}. 
	\]
	
	If the statement is false for our choice of constant $T, \eta_1$, there is a sequence $(y_N) =(u_N, r_N)$ of initial states that verify $(r_N)\to +\infty$ and $|\ell(u_N, r_N)|\leq \eta_0$, and such that, for all $N\geq 1$, 
	\begin{equation}\label{e:RD11Contradiction}
		\E_{(u_N,r_N)}\left(V_{u_*}(U(\tau\wedge T), R(\tau\wedge T))- V_{u_*}(u_N, r_N) \right)> -1. 
	\end{equation}
	For all $N\geq 1$, let $(Y_N(t))=(U_N(t), R_N(t))$ be the Markov process defined by the infinitesimal generator $\widehat{\cal{A}}_\omega$ and starting from $(u_N, r_N)$, and let $L^N(t):=\ell(U_N(t), R_N(t))$. Note that $\tau=\tau_{L^N}$. 
	
	By definition of $\widehat{V}_{u_*}$, we have:
 
	\begin{align*}
		\E_{y_N}\left( \frac{\widehat{V}_{u_*}(Y_N(\tau\wedge T))}{r_N^{b+a/2}}\right)&= \E_{y_N}\left( \frac{R_N(\tau\wedge T)^{b+a/2}}{r_N^{b+a/2}}(\frac{\eta_0}{2}\vee |L_N(\tau\wedge T)|)^{-a}\right)\\
		&\leq \left(\frac{2}{\eta_0}\right)^a\E_{y_N}\left(\sup_{t\leq T\wedge \tau}\left|\frac{R_N(t)^{b+a/2}}{r_N^{b+a/2}}-e^{\alpha't}\right|\right)\\
		&\quad +\E_{y_N}\left( \frac{e^{\alpha' (\tau\wedge T)}}{\eta_1^a}\right) + \frac{2^ae^{\alpha'T}}{\eta_0^a}\P_{y_N}(\tau> T). 
	\end{align*}
	Up to a subsequence, we can assume that $(r_N)$ is strictly increasing and that $\ell(u_N, r_N)\to \ell_0\in [-\eta_0, \eta_0]$. We can now use Lemma~\ref{l:OU11}. 
	
	Moreover, $(L^N(t))$ converges in distribution to the divergent Ornstein-Uhlenbeck process $(L(t))$ starting from $\ell_0$ and with SDE given in Relation~\eqref{e:OU11}. Therefore, when $N\to+\infty$, 
	\[
		\E_{y_N}\left( \frac{e^{\alpha' (\tau\wedge T)}}{\eta_1^a}\right) + \frac{2^ae^{\alpha'T}}{\eta_0^a}\P_{y_N}(\tau> T)
		\to \, \E_{\ell_0}\left( \frac{e^{\alpha' (\tau_L\wedge T)}}{\eta_1^a}\right) + \frac{2^ae^{\alpha'T}}{\eta_0^a}\P_{\ell_0}(\tau_L> T).
	\]
	By Equation~\eqref{e:ConvOUeq1} in Lemma~\ref{l:ConvOU}, for $C_0$ independent of $T$ and $\eta_1$, we have
	\[
		\frac{2^ae^{\alpha'T}}{\eta_0^a}\P_{\ell_0}(\tau_L> T)
		\leq \frac{\eta_12^a C_0 e^{\alpha'T}}{\eta_0^a\sqrt{e^{2\alpha T}-1}}\leq\frac{1}{4\eta_0^a}. 
	\]
	Furthermore, because of Relation~\eqref{e:ConvOUeq2} in the same Lemma, for some $C_L$ depending on $\alpha'$, we have 
	\[
		\E_{\ell_0}\left( \frac{e^{\alpha' (\tau_L\wedge T)}}{\eta_1^a}\right)\leq \frac{C_L(\eta_1+1)}{\eta_1^a}\leq \frac{1}{4\eta_0^a}. 
	\]
	Finally, the term on the second line is controlled by Relation~\eqref{e:OU11Lim2}. 
	Using also $|\ell(u_N, r_N)|\leq \eta_0$, for $\eps>0$, we can find $N_0\geq 1$, such that
	\[
		\E_{y_N}\left( \frac{V_{u_*}(Y_N(\tau\wedge T))}{r_N^{b+a/2}}\right)-\frac{V_{u_*}(y_N)}{r_N^{b+a/2}} \leq -\frac{1}{2\eta_0^a}+\eps
	\]
	where $\eps$ contains the error term between the terms in $N$ and their limit when $N\to +\infty$. Taking $\eps$ small enough, since $r_N\to +\infty$, this leads to a contradiction of the inequality~\eqref{e:RD11Contradiction}, which concludes the proof. 
\end{proof}

\subsection{Proof of Proposition~\ref{p:RepLine11}}\label{s:TimeChange}
It remains to show that the time change we made in Relation~\eqref{e:GenUR11} can be reversed without affecting the dissipation.
 
When the dissipation is point-wise, we just write the relation between the original infinitesimal generator $\cal{A}$ defined in Relation~\eqref{e:IGOriginal} and the generator $ \widehat{\cal{A}}_\omega$ of Relation~\eqref{e:GenUR11}. 

For $f\in \cal{C}^3$ and $x\in \cal{H}$, we have 
\begin{equation}\label{e:LinkAfHatAf}
	\cal{A}(f)(x)= \frac{r(x)^{\rho_0-1}}{(1+u(x)^2)^{(\rho_0+1)/2}}\widehat{\cal{A}}_\omega(f\circ \underline{x})(u(x), r(x))
\end{equation}
where $\underline{x}$ is defined in Relation~\eqref{e:RevChangeVarNot11}. 
Since $\rho_0\geq 1$, the point-wise dissipation for $\widehat{\cal{A}}_\omega$ can directly be translated to $\cal{A}$. 
When the dissipation happens after a stopping time $\tau$, i.e. in Proposition~\ref{p:DissipRDFilo11}, we have to be more careful. More precisely, the time change does not change the sample paths, just the speed at which the process covers them.

\subsubsection*{If $G_\omega(u^*)<0$}
	Defining $V_{u_*}(x):= \widehat{V}_{u_*}(u(x), r(x))$ with $\widehat{V}_{u_*}(u,r)$ given in Proposition~\ref{p:DissipRC11}, Relation~\eqref{e:LinkAfHatAf} directly proves Proposition~\ref{p:RepLine11}. 
	
\subsubsection*{If $G_\omega(u_*)>0$}
	Defining $V_{u_*}(x):= \widehat{V}_{u_*}(u(x), r(x))$ with $\widehat{V}_{u_*}(u,r)$ given in Proposition~\ref{p:DissipRD11}, Relation~\eqref{e:LinkAfHatAf} leads to the dissipation in Proposition~\ref{p:RepLine11} for all states $x$ large enough and such that $|u(x)-u_*|\geq \eta_0 r(x)^{-1/2}$. 
	For the state in the inner cone, i.e. states $x$ such that  $|u(x)-u_*|\leq \eta_0 r(x)^{-1/2}$ we use Proposition~\ref{p:DissipRDFilo11}. Let $(X(t))$ be the process associated to the original infinitesimal generator $\cal{A}$.
	Define for $t\geq 0$ the time-change 
	\[
		\phi(t)= \inf\left\{s>0: \int_0^s r(X(s'))^{\rho_0-1}(1+u(X(s'))^2)^{-(\rho_0+1)/2}\diff s'\geq t\right\},  
	\]
	and the process $(\tilde{U}(t), \tilde{R}(t))= (u(X(\phi(t)), r(X(\phi(t))))$, then $(\tilde{U}(t), \tilde{R}(t))$ has the same distribution as $(U(t), R(t))$. See Section~III.21 of~\cite{Rogers1}. 
	Now defining as in Proposition~\ref{p:DissipRDFilo11} the constant $T>0$ and the stopping time 
	\[
		\tilde{\tau}:=\inf\{t\geq 0: |\ell(\tilde{U}(t), \tilde{R}(t))|\geq \eta_1\},
	\]
	we get
	\[
		V_{u_*}(X(\phi(\tilde{\tau}\wedge T)))= \widehat{V}_{u_*}(\tilde{U}(\tilde{\tau}\wedge T), \tilde{R}(\tilde{\tau}\wedge T))
	\]
	and because of Proposition~\ref{p:DissipRDFilo11}, for $\|x\|$ large enough such that $|u(x)-u_*|\leq \eta_0 r(x)^{-1/2}$, we have 
	\[
			\E_{x}\left(V_{u_*}(X(\phi(\tilde{\tau}\wedge T)))\right)-V_{u_*}(x)\leq -1.
	\]
	
	Finally, because $\rho_0\geq 1$ and $u(X(s{-}))\approx u_*$ here, we can find a $C_{u_*}>0$ such that, almost surely  for all $s\geq 0$,  
	\[
		\int_0^s r(X(s'))^{\rho_0-1}(1+u(X(s'))^2)^{-(\rho_0+1)/2}\diff s'\geq s/C_{u_*}
	\]
	which means that almost surely for all $t\leq T$, $\phi(t)\leq tC_{u_*}$, and therefore 
	\[
		\E_x(\phi(\tilde{\tau}\wedge T))\leq C_{u_*}\E_x(\tilde{\tau}\wedge T)\leq C_{u_*}T, 
	\]
	which concludes the proof.

\section{Repulsive curve, $\omega\neq (1,1)$}\label{s:RNeq11}
The goal of this section is to show the following proposition: 
\begin{proposition}\label{p:RepLineNeq11}
	Let $b>0$. Assume $\omega=(1,\gamma)$, $\gamma\in (0,1)$, and that $u_*=u_i\in \cal{P}_\omega$ is such that 
	\begin{equation*}
		F'(u_*)>0.
	\end{equation*}
	Then for any choice of $u_-, u_+$ such that $u_{i-1}<u_-<u^*<u_+<u_{i+1}$, we can find an energy function $V_{u^*}^\omega$ that verifies the $b$-dominance at the boundary of the set (see Definition~\ref{d:b-dominance}), an integrable stopping time $\tau\neq 0$ and some constants $\beta, C_\tau, R_0>0$ such that, \\
	for all $x\in \cal{H}_\omega(u_-, u_+)$ with $\|x\|\geq R_0$, 
	\[
		\E_x\left( V_{u^*}^\omega(X(\tau))-V_{u^*}^\omega(x)\right)\leq -\beta\E_x(\tau)\quad \text{and} \quad \E_x(\tau)\leq C_\tau. 
	\]
\end{proposition}
This proposition is proven in several steps. First, similarly to the previous section, we make a change of variable $(u, r)$ to track the deviation to the axis and the norm of the system. Then, in these new coordinates, we introduce the ``local'' infinitesimal generator controlling the behavior of the system in this order of magnitude, and give its first order terms (see Proposition~\ref{p:ReducedGeneratorNeq11}). We define an energy function $\widehat{V}_{u_*}$ point-wise dissipated for $u$ far enough from $u_*$, see Proposition~\ref{p:RDExtConeNeq11}. Finally,  we show the dissipation in the inner cone using a scaling result, see Proposition~\ref{p:DissipRDFiloNeq11}.  

\subsection{Change of variables and infinitesimal generator}
\begin{definition}[Change of coordinates and generator]
	Let $\omega=(1, \gamma)$, $\gamma\in (0,1)$. 
	We introduce the following change of coordinates:
	\begin{equation}\label{e:ChangeVarNeq11}
		(u, r): (x_1, x_2)\mapsto \left(\frac{x_2}{x_1^\gamma}, x_1\right). 
	\end{equation}
	Conversely, we define the function $\underline{x}(u,r)$ as 
	\begin{equation*}
		\underline{x}: (u, r)\mapsto \left(r, ur^{\gamma}\right). 
	\end{equation*}
	We define the generator $\widehat{\mathcal A}_\omega$ of the process $(u(X(t)), r(X(t)))$ after a time-change. For any function $f\in \cal{C}^3$, and $(u, r)\in (0, +\infty)^2$, we set 
	\begin{multline}\label{e:GenURNeq11}
			 \widehat{\cal{A}}_\omega f(u,r):=\frac{1}{r^{\rho_0-\gamma}} \\
			 \sum_{\nu\in\mathcal R}\lambda_\nu(\underline{x}(u,r))\left[f\left(u(\underline{x}(u,r)+\xi^\nu),r(\underline{x}(u,r)+\xi^\nu)\right)-f(u,r)\right].
	\end{multline}
\end{definition}

Unlike in the diagonal case ($\omega=(1,1)$), in the case where $\omega=(1, \gamma)$ and $\gamma<1$, some lower order reactions may have to be taken into account in the Taylor expansion of the generator. Moreover, the invariant curve will not be exactly the curve $x_2=u_*x_1^\gamma$, but will be a curve $(h(r), r)$ where $h(r) \to u_*$ when $r\to +\infty$. 

We first introduce the relevant quantities used in the expansion of the generator.  
The first order involves the functions $F_\omega(u)$ and $G_\omega(u)$ given by the Poincar\'e-Lyapunov compactification, see Definition~\ref{d:CompNeq11}. 
The set of subdominant reactions that appear in the Taylor expansion is $\cal{R}_\omega'$ given by
\begin{equation*}
	\cal{R}_\omega':=\{y^-\to y^+\in \cal{R}\setminus \cal{R}_\omega: \, \croc{y^-, \omega}>\rho_0(\omega)-\gamma\}, 
\end{equation*}
and we introduce the functions $R_\omega(u,r)$ and $B_\omega(u,r)$ (polynomial in the variable $u$) as 
\begin{equation*}
	R_\omega(u,r):= \sum_{\nu \in \cal{R}_\omega'} \kappa_{\nu}r^{\croc{y^-_\nu, \omega}-\rho_0(\omega)}u^{y_{\nu, 2}^-}\left(\xi_2^\nu -\gamma u r^{\gamma-1}\xi_1^\nu\right). 
\end{equation*}
and 
\[
	B_\omega(u,r):= \sum_{\nu \in \cal{R}_\omega'} \kappa_{\nu}\xi_1^\nu r^{\croc{y^-_\nu, \omega}-\rho_0(\omega)}u^{y_{\nu, 2}^-}.
\]
By definition of $\cal{R}_\omega$ and $\cal{R}_\omega'$, $R_\omega(u,r)\to 0$ and $B_\omega(u,r)\to 0$ when $r\to +\infty$ and $u$ bounded. 

The noise term in the generator involves the function $D_\omega(u)$ defined as 
\begin{equation*}
	D_\omega(u)= \sum_{\nu\in \cal{R}_\omega}\kappa_\nu (\xi_2^\nu)^2 u^{y^-_{\nu,2}}. 
\end{equation*}

\begin{proposition}\label{p:ReducedGeneratorNeq11}
Let $u_-<u_+\in (0, +\infty)$. 
For every $f\in C^3$, for $u\in [u_-, u_+]$ and $r\geq 1$,
\begin{multline}\label{e:IGReducedNeq11}
    \widehat{\cal{A}}_\omega f(u,r)
		=\left(F_\omega(u)-\gamma u r^{\gamma-1}G_\omega(u) +R_\omega(u,r)\right)\frac{\partial f}{\partial u}(u,r)\\
		+r^{\gamma} (G_\omega(u)+B_\omega(u,r))\frac{\partial f}{\partial r}(u,r)
		+\frac{D_\omega(u)}{2r^\gamma}\frac{\partial^2 f}{\partial u^2}(u,r)+\mathcal E_f(u,r),
\end{multline}
where $\mathcal E_f$ contains correction terms, of lower order when $r\to +\infty$.

It can be bounded as follows: there exists $C_K>0$, and a non negative function $\eps(r)$ that goes to $0$ when $r\to +\infty$,  such that 
\begin{multline*}
    |\cal{E}_f(u, r)|\leq C_K\big(\frac{1}{r^{\gamma}}\left|\frac{\partial f}{\partial u}\right|
    +\left|\frac{\partial f}{\partial r}\right|
    +\frac{\eps(r)}{r^\gamma}\left|\frac{\partial^2 f}{\partial u^2}\right| \\
		+r^{\gamma}\sum_{\substack{(i,j)\in\{0, \ldots 3\}^2\setminus\{(2,0)\}:\\ 2\leq i+j\leq 3}} r^{-\gamma i}\left|\frac{\partial^{i+j} f}{\partial u^{i}\partial r^{j}}\right|\big).
\end{multline*}
\end{proposition}

\begin{proof}
The proof is very similar to the one of Proposition~\ref{p:ReducedGenerator11}. We just have to be very careful about next order reactions. 
Let $f(u,r)$ be a $\cal{C}^3$ function. 

For $\xi\in \R^2$, $\|\xi\|\leq K_0$, 
	\begin{align*}
		&u(x(u, r)+\xi)-u=  \xi_2r^{-\gamma}-\gamma\xi_1ur^{-1}+\frac{\eps_u(u,r)}{r^{1+\gamma}}\\
		&r(x(u,r)+\xi)-r= \xi_1,
	\end{align*}
	where until the end of the proof, for any index $i$, $\eps_i(u,r)$ is a bounded function on $(u, r)\in [u_-, u_+]\times [1, +\infty)$. 

Therefore, using Taylor's expansion, we get 
	\begin{multline*}
		r^{\gamma}(f(u(x(u, r)+\xi), r(x(u,r)+\xi))-f(u, r))\\
		 = \xi_1 r^{\gamma}\frac{\partial f}{\partial r}+ (\xi_2-\gamma\xi_1 u r^{\gamma-1})\frac{\partial f}{\partial u} + \frac{(\xi_2)^2}{2r^\gamma}\frac{\partial^2 f}{\partial u^2} + \cal{E}_{f, \xi}(u,r) 
	\end{multline*}
	where for some $C>0$ large enough, we have 
	\begin{multline*}
		|\cal{E}_{f, \xi}(u,r)|
		\leq C\left(\frac{1}{r}\left|\frac{\partial f}{\partial u}\right|  
		+\frac{1}{r}\left|\frac{\partial^2 f}{\partial u^2}\right|\right) \\
		+Cr^\gamma\sum_{\substack{(i,j)\in\{0, \ldots 3\}^2\setminus\{(2,0)\}:\\ 2\leq i+j\leq 3}} r^{-\gamma i}\left|\frac{\partial^{i+j} f}{\partial u^{i}\partial r^{j}}\right|
	\end{multline*}

	For a reaction $\nu=y_-\to y_+$, we define the rate after the time change:
	\[
		\widehat{\lambda}_\nu(u, r):= r^{\gamma-\rho_0}\lambda_\nu(x(u, r))
	\]

	If $\nu\in \cal{R}_\omega$ and $\nu' \in \cal{R}'_\omega$, we have 
	\begin{multline*}
		\widehat{\lambda}_\nu(u, r)= \kappa_\nu r^\gamma u^{y_{\nu, 2}^-}(1+\frac{\eps_\nu(u, r)}{r^\gamma})\\
		\text{and}\quad
		\widehat{\lambda}_{\nu'}(u, r)= \kappa_{\nu'} r^{\gamma+\croc{\omega, y^-_{\nu'}}-\rho_0} u^{y_{\nu', 2}^-}(1+\frac{\eps_{\nu'}(u, r)}{r^\gamma})
	\end{multline*}
	Finally, if $\nu\notin\cal{R}_\omega\cup \cal{R}_\omega'$, 
	\[
		\widehat{\lambda}_\nu(u, r)= \eps_\nu(u,r). 
	\]
	Combining these results, we get the approximation. 
\end{proof}

\subsection{Central curve of the cone}
In this case we can't use a simple energy function like $f_a^{u_*}$ as in Section~\ref{s:R11}, because $u=u_*$ is not the exact center of the cone, i.e., $x_2=u_*x_1^\gamma$ is not the invariant divergent curve for the deterministic system. 

Intuitively, the curve in the center of the cone is a set $\{(u,r): \, u=h(r)\}$ with $h$ a function that verifies $h(r)\to u_*$ when $r\to +\infty$, and which is invariant for the \emph{leading order deterministic vector field}. 

The ODE approximating the evolution of $(u,r)$ is given by the first order terms (in derivation) in the expansion of the generator, i.e., is given by 
	\[
		\begin{cases}
			\dot u&=F_\omega(u)-\gamma u r^{\gamma-1}G_\omega(u) +R_\omega(u,r)\\
			\dot r&= r^{\gamma} (G_\omega(u)+B_\omega(u,r))
		\end{cases}
	\]
	Note that this ODE is autonomous, and can be written as $\partial_t(u, r)=\Phi (u, r)$. An invariant curve $(h(r), r)$ must verify 
	\begin{equation}\label{e:EqhExact}
		(1, -h'(r))\cdot \Phi (h(r), r)=0. 
	\end{equation}
	
	Because our process $(U(t), R(t))$ is a \emph{jump process}, with jumps in $u$ of size $O(r^{-\gamma})$, we only need an approximation of the solution of Equation~\eqref{e:EqhExact} at the order $r^{-\gamma}$, i.e., we need a function $(h(r))$ such that, for $r\to +\infty$, 
	\begin{multline}\label{e:Defh}
		F_\omega(h(r))-\gamma r^{\gamma-1} h(r) G_\omega(h(r)) +R_\omega(h(r),r)\\
		- r^{\gamma}h'(r) (G_\omega(h(r))+B_\omega(h(r),r))=O(r^{-\gamma}). 
	\end{multline}

\begin{lemma}\label{l:Defh}
	There exist an integer $p\geq 0$, an ordered set $0<\gamma_1<\gamma_2 \ldots <\gamma_p<\gamma$ and some real numbers $(\alpha_i)\in \R^p$ such that, defining 
	\[
		h(r):= u_*+\sum_{i=1}^p \alpha_ir^{-\gamma_i}, 
	\]
	Relation~\eqref{e:Defh} holds.
	Moreover, if $p\geq 1$, when $r\to +\infty$,  
	\[
		\alpha_1r^{-\gamma_1}\sim\frac{\gamma r^{\gamma-1} h(r) G_\omega(h(r))-R_\omega(u_*,r)}{F'_\omega(u_*)}\quad \text{and}\quad r^{\gamma}h'(r)=o(1). 
	\]
\end{lemma}

\begin{proof}
	The proof can be done by induction, first identifying $\gamma_1$ and $\alpha_1$, and so on. The number of exponents $p$ is finite because the number of degrees in $R_\omega(u)$ is finite. Since $\gamma<1$, for $r\to +\infty$, we clearly have $r^\gamma h'(r)= O(r^{\gamma-\gamma_1-1})\to 0$. 
\end{proof}
Until the end of the section, we consider that $(h(r))$ is fixed, given by Lemma~\ref{l:Defh}. 
\begin{lemma}\label{l:TechnIneqH}
	Define the function $H(u, r)$ as 
	\begin{equation*}
		H(u,r):= F_\omega(u)-\gamma r^{\gamma-1} u G_\omega(u) +R_\omega(u,r)- r^{\gamma}h'(r) (G_\omega(u)+B_\omega(u,r))
	\end{equation*}
	Under the assumptions and notations of Proposition~\ref{p:RepLineNeq11}, there exist an $r_0>0$ and a $\eps>0$ such that, for all $u\in [u_-, u_+]$ that verify $|u-h(r)|r^{\gamma}\to +\infty$ when $r\to +\infty$, for all $r>r_0$,  
	\begin{equation}\label{e:TechnIneqH}
		\frac{H(u,r)}{u-h(r)}>\eps. 
	\end{equation}
\end{lemma}

\begin{proof}
	Note that by definition, $F_\omega(u)$, $G_\omega(u)$, $R_\omega(u, r)$ and $B_\omega(u,r)$ are all polynomial functions in the variable $u$, and hence are $C^2$. We carry out the proof in two steps. 

	\emph{Far enough from $h(r)$}: Consider $u\in [u_-, u_+]$ such that $|u-h(r)|>r^{-\gamma_1/2}$, where $\gamma_1$ is given in the definition of $h(r)$, see Lemma~\ref{l:Defh}.  For $r$ large enough, we have 
	\[
		\frac{u-u_*}{u-h(r)}>1/2.
	\]
	Because of the assumption $F'_\omega(u_*)>0$, and the definition of the $u_i$ as roots of the polynomial function $F_\omega$, we have also, for $u\neq u_*$ and some $\eps>0$ small enough, 
	\[
		\frac{F_\omega(u)}{u-u_*}>4\eps
	\]
	and hence 
	\[
		\frac{F_\omega(u)}{u-h(r)}=\frac{F_\omega(u)}{u-u_*}\frac{u-u_*}{u-h(r)}>2\eps.
	\]
	Moreover, by definition of $h(r)$, 
	\[
		\lim_{r\to +\infty}r^{\gamma_1/2}\left|-\gamma r^{\gamma-1} u G_\omega(u) +R_\omega(u,r)- r^{\gamma}h'(r) (G_\omega(u)+B_\omega(u,r))\right|=0, 
	\]
	which leads to the inequality~\eqref{e:TechnIneqH} for $u\in [u_-, u_+]$ such that $|u-h(r)|>r^{-\gamma_1/2}$. 
	
	\emph{Close to $h(r)$:} For $|u-h(r)|\leq r^{-\gamma_1/2}$, since by assumption $|u-h(r)|r^{\gamma}\to +\infty$ we have, when $r\to +\infty$
	\begin{align*}
		\frac{H(u, r)}{u-h(r)}&=\frac{H(h(r), r)}{u-h(r)} +\frac{\partial H}{\partial u}(h(r), r) +o(1)\\
		&=\frac{H(h(r), r)}{u-h(r)} +F'_\omega(h(r))-\gamma r^{\gamma-1}G_\omega(h(r)) -\gamma h(r)r^{\gamma-1}G'_\omega(h(r))\\
		&\quad +\partial_u R_\omega(h(r), r) -h'(r)r^\gamma (G'_\omega(h(r))+\partial_u B_\omega(h(r), r)) +o(1) \\
		&= \frac{O(r^{-\gamma})}{u-h(r)} + (F'_\omega(u_*) +O(r^{-\gamma_1})) +o(1)\\
		&=F'_\omega(u_*) +o(1)
	\end{align*}
	where in the second equality, we used $H(h(r), r)=O(r^{-\gamma})$, and $\gamma<1$. Since $F'_\omega(u^*)>0$, this concludes the proof. 
\end{proof}

\subsection{Dissipation in the outer cones}
As in the proof for $\omega=(1,1)$ (see Proposition~\ref{p:DissipRD11}), we first prove the dissipation point-wise outside an inner cone.  
For a choice of $\eta_0>0$ and $a>0$, we introduce the candidate energy function on $\{(u,r): u\in [u_-, u_+]\}$: 
\begin{equation}\label{e:EnergyOmegaNeq11}
	\widehat{V}_{u_*}: (u,r)\mapsto 
	\begin{cases}
		\frac{r^b}{|u-h(r)|^a} \quad &\text{for   } u\in [u_-, u_+]: \, |u-h(r)|r^{\gamma/2}\geq \eta_0/2,\\
		\frac{2^ar^{b+a\gamma/2}}{\eta_0^a}\quad &\text{for   } u\in [u_-, u_+]: \, |u-h(r)|r^{\gamma/2}< \eta_0/2. 
	\end{cases}
\end{equation} 

The following proposition shows that this Energy function is point-wise dissipated by the process when it is far enough from the invariant curve. 

\begin{proposition}\label{p:RDExtConeNeq11}
	Under the assumption of Proposition~\ref{p:RepLineNeq11}, we can find some $r_0>0$, some $\eta_0>0$ and $\beta>0$ such that, defining $\widehat{V}_{u_*}$ as in Relation~\eqref{e:EnergyOmegaNeq11}
	for all $r\geq r_0$ and $u\in [u_-, u_+]$ with $|u-h(r)|r^{\gamma/2}\geq \eta_0$, 
	\[
		\widehat{\cal{A}}_\omega \widehat{V}_{u_*}(u,r)\leq -\beta r^b,
	\]
	where $\widehat{\cal{A}}_\omega$ is defined in Relation~\eqref{e:GenURNeq11}. 
\end{proposition}

\begin{proof}
	This proof relies once again on a careful Taylor expansion of the generator. Using the approximation given in Relation~\eqref{e:IGReducedNeq11} with the function defined above, we get  
	\begin{align*}
		\frac{\widehat{\cal{A}}_\omega \widehat{V}_{u_*}(u,r)}{\widehat{V}_{u_*}(u,r)}
		=-a\frac{H(u,r)}{u-h(r)} +br^{\gamma-1}(G_\omega(u)+B_\omega(u,r)) +\frac{a(a+1)D_\omega(u)}{2r^\gamma|u-h(r)|^2}+\overline{\cal{E}}(u,r), 
	\end{align*}	
	where $H(u,r)$ has been defined in Lemma~\ref{l:TechnIneqH}, and
	\[
		|\overline{\cal{E}}(u,r)|\leq C_K\left(\frac{2}{r^\gamma|u-h(r)|} + 2 r^{\gamma-1}+\frac{\eps(r)}{r^\gamma|u-h(r)|^2}\right). 
	\]
	Lemma~\ref{l:TechnIneqH} and the lower bound $\widehat{V}_{u_*}(u,r)>Cr^b$ for some $C>0$ on the region studied conclude the proof, taking $\eta_0$ large enough,. 
\end{proof}

\subsection{Dissipation in the inner cone}

\begin{proposition}\label{p:DissipRDFiloNeq11}
	We keep the notations of Proposition~\ref{p:RDExtConeNeq11}. 
	We can find some $T>0$ and $\eta_1>\eta_0$ and $R_0\geq 1$ such that, 
	for any initial point $(u,r)$ with $r\geq R_0$ and $|u-h(r)|r^{\gamma/2}\leq \eta_0$, 
	\[
		\E_{(u,r)}\left(\widehat{V}_{u_*}(U(\tau\wedge T), R(\tau\wedge T))- \widehat{V}_{u_*}(u, r) \right)\leq -1, 
	\]
	where $(U(t), R(t))$ is the Markov process defined by the infinitesimal generator $\widehat{\cal{A}}_\omega$, see Relation~\ref{e:GenURNeq11}, and $\tau$ is the following stopping time:
	\begin{equation*}
		\tau:= \inf\{t\geq 0: |U(t)-h(R(t))|R(t)^{\gamma/2}\geq \eta_1\}. 
	\end{equation*}
\end{proposition}

We set the function $\ell$ as 
\begin{equation*}
	\ell:(u, r)\mapsto (u-h(r))r^{\gamma/2}. 
\end{equation*}

Similarly to the case of $\omega=(1,1)$, the proof of Proposition~\ref{p:DissipRDFiloNeq11} relies on a limit theorem for the sample paths.  
We show that when starting from $|\ell_0|\leq \eta_0$, the deviation to the axis, i.e., the process $(\ell(U(t), R(t)))$ converges to a divergent Ornstein-Uhlenbeck given by the SDE
	\begin{equation}\label{e:OUNeq11}
		\diff L(t)= F'_\omega(u_*)L(t)\diff t + \sqrt{D_\omega(u_*)}\diff W_t, 
	\end{equation}
	with $(W_t)$ a standard Wiener process. 

The convergence is shown until the first time the process exits the interval $[-\eta_1, \eta_1]$ which, given the stopping time chosen in Filonov's criterion in Proposition~\ref{p:DissipRDFiloNeq11}, is enough to estimate the energy.

For a process $(Z(t))$ on $\R$, we use the same notation for $\tau_Z$ as in Relation~\eqref{e:DefTauZ}: 
\[
	\tau_Z=\inf\{t\geq 0: |Z(t)|\geq \eta_1\}. 
\] 

\begin{lemma}\label{l:OUNeq11}
	Set $L^N(t)=\ell(U^N(t),R^N(t))$.
	Let $(u_N)\in \R^{\Z_{\geq 0}}$ such that $\ell(u_N, N)\to \ell_0$ for some $\ell_0\in [-\eta_0, \eta_0]$. 
	Let $(U^N(t), R^N(t))$ be the Markov process with infinitesimal generator given by Relation~\ref{e:GenURNeq11}, starting from $(u_N, N)$. 
	Then for $T>0$, for the convergence in distribution, 
	\begin{multline*}
		\lim_{N\to +\infty}\left(\left(L^N(t\wedge \tau_{L^N}), \frac{R_N(t\wedge \tau_{L^N})}{N}\right)_{t\in [0, T]}, \tau_{L_N}\wedge T\right) \\
		= ((L(t\wedge \tau_L), 1)_{t\in [0, T]}, \tau_L\wedge T),
	\end{multline*}
	where $(L(t))$ is a divergent Ornstein-Uhlenbeck process starting from $\ell_0$ and solution of the SDE~\eqref{e:OUNeq11}. 

	Moreover, for $\alpha\geq 1$, for any $\eps>0$, we can find $N_0\geq 0$ such that, for all $N\geq N_0$,
	\begin{equation}\label{e:OUNeq11Lim2}
		\E_{(u_N, N)}\left(\sup_{t\in [0, T\wedge\tau_{L^N}]}\left|\frac{R_N(t)^\alpha}{N^\alpha}-1\right|\right)<\eps.
	\end{equation}	
\end{lemma}

The proofs of Lemma~\ref{l:OUNeq11} and Proposition~\ref{p:DissipRDFiloNeq11} are very similar to the ones of the ``equivalent'' results for $\omega=(1,1)$, Lemma~\ref{l:OU11} and Proposition~\ref{p:DissipRDFilo11} respectively. They are even simpler since the norm doesn't have time to increase before $\tau_{L^N}$, and the control of $(R^N(t\wedge \tau^N))$ is therefore straightforward, see Equation~\eqref{e:OUNeq11Lim2}.
\section{Attractive curve}\label{s:A}

The goal of this section is to show that along the attractive, convergent curves, the process has time to decrease significantly. It would be possible (and quite easy) to show only point-wise dissipation in this region, but this stronger result will significantly facilitate the ``patching'' between regions of the compactification. 
 
\begin{proposition}\label{p:DissipAC}
	Let $b\geq 1$. Fix $(1, 0)<\omega\leq (1,1)$. Assume that $u_*=u_i\in \cal{P}_\omega$ is such that 
	\begin{equation*}
		F'(u_*)<0\quad \text{and}\quad G(u_*)<0.
	\end{equation*}
	Then we can find some $u_-, u_+$ with $u_{i-1}<u_-<u^*<u_+<u_{i+1}$ such that for any $\eps\in (0,1)$, we can find an integrable stopping time $\tau$ and some constants $R, C>0$ such that for all $x\in \cal{H}_\omega(u_-, u_+)$ with $\|x\|\geq R$, 
	\[
		\E_x\left( \|X(\tau)\|^b\right)\leq \eps^b\|x\|^b \quad \text{and}\quad \E_x(\tau)\leq C
	\]
	and almost surely, $(X(t\wedge\tau))$ starting from $x$ stays in $\cal{H}_\omega(u_-, u_+)$. 
\end{proposition}

The consequence of this proposition is that in this region, we can define the energy function freely as soon as it is $b$-dominated. In particular, in the proof of Theorem~\ref{t:Main}, we will not care if the Energy function makes a big jump of order $\|x\|^b$ in these regions, as long as the order of magnitude of the function is preserved. 
\begin{corollary}\label{c:AttractivePatching}
	Under the assumptions of Proposition~\ref{p:DissipAC}, if $W$ defines an energy function on $\cal{H}_\omega(u_-, u_+)$ such that, for some constant $0<c<C$, for all $x\in \cal{H}_\omega(u_-, u_+)$,
	\[
		c\leq \frac{W(x)}{\|x\|^b}\leq C, 
	\]
	then we can find an integrable stopping time $\tau$ and some constants $R,\beta>0$ such that for all $x\in \cal{H}_\omega(u_-, u_+)$ with $\|x\|\geq R$, 
	\[
		\E_x\left( W(X(\tau))\right)-W(x)\leq -\beta \E_x(\tau),
	\]
	and almost surely, $(X(t\wedge\tau))$ starting from $x$ stays in $\cal{H}_\omega(u_-, u_+)$. 
\end{corollary}
The proof of this corollary is straightforward, taking $\eps>0$ small enough in Proposition~\ref{p:DissipAC}.

As in the previous sections, we make a change of variables and study the evolution of the deviation from the axis and the norm of the system. 
We introduce the Markov process $(Y(t))=(U(t), R(t))$ with infinitesimal generator given in Relation~\eqref{e:GenUR11} for $\omega=(1,1)$, respectively~\eqref{e:GenURNeq11} for $\omega\neq (1,1)$. 

Let's assume that $(U(0), R(0))=(u_*, N)$. We first need to understand at which timescale the process $(R(t)/N)$ decreases significantly. This can be understood from the term in $\partial r$ of the approximated infinitesimal generators. Relations~\eqref{e:GenUR11} and~\eqref{e:GenURNeq11} tell us that on the timescale $t\mapsto t$, the decrease of $R(t)$ is of order $N^{\omega_2}$. Therefore, to see $(R(t)/N)$ decrease significantly, we need to study the process at an accelerated timescale, and more precisely at the timescale $t\mapsto N^{1-\omega_2}t$. 
The difficulty here is to make sure that on this timescale, the process $(U(t))$ stays close to $u_*$. This is done in the first part of this Section. In the second part, we show the decrease of the process.  

\subsection{Staying close to the curve}
For $R_0>0$ and $\eta>0$, we define: 
\begin{align*}
	\tau_{R_0}^r &:=\inf\{t\geq 0: \, R(t)\leq R_0\},\\
	\tau_{\eta}^u &:= \inf\{t\geq 0: |U(t)-u_*|\geq \eta\}. 
\end{align*}

This first lemma controls the process $(U(t))$, first on a finite time interval $[0,T]$ in Equation~\eqref{e:ACControlU1}, and then on any polynomial timescale. The proof uses two key ideas. First, assuming that $U(0)\in (u_*-\eta, u_*+\eta)$, a uniform control of the martingale part of $(U(t))$ shows that, with high probability, the process does not exit the interval $(u_*-2\eta, u_*+2\eta)$. Then, using a lower bound on the (deterministic) drift that brings $U(t)$ back to $u_*$, we show that, on the same event, at time $T$ the process is back in the interval $(u_*-\eta, u_*+\eta)$. See~\cite{Katzenberger} for a similar model. 
The extension to an interval $[0, R(0)^{\alpha}T]$ simply uses the Markov property. 

\begin{lemma}\label{l:ACControlU}
	We can find some $0<\eta<u_*/2$ and some $T, C, R_0>0$ large enough such that, starting from $(u_0, r_0)$ with $r_0>R_0$ and $u_0\in {(u_*-\eta, u_*+\eta)}$, 
	\begin{multline}\label{e:ACControlU1}
			\P_{(u_0, r_0)}\left(\sup_{t\in [0,T\wedge \tau^r_{R_0}]} |U(t)-u_*|<2\eta, \quad  |U(T)-u_*|\ind{T<\tau_{R_0}^r}<\eta  \right)\\
			\geq 1-\exp\left(-CR_0^{\omega_2}\right). 
	\end{multline}
	
	Therefore, for any $\alpha, \alpha'>0$, for $R_0$ large enough, the following inequality holds: 
	\begin{equation}\label{e:ACControlU2}
		\P_{(u_0, r_0)}\left( \tau_{2\eta}^u< TR_0^\alpha \wedge \tau_{R_0}^r\right) \leq R_0^{-\alpha'}. 
	\end{equation}
\end{lemma}

\begin{proof}
	We write the proof for $\omega<(1,1)$. For $\omega=(1,1)$, it can be carried out the same way. 
	Assume that $\omega=(1, \gamma)$ with $\gamma\in(0,1)$.
	
	The result is shown up to $\tau_{R_0}^r$, and therefore until the end of the proof, we can consider that $r\geq R_0$.
	Let $f(u, r)=u$. 
	The process $(U(t))$ is solution of the equation, for $t\geq 0$, 
	\begin{equation}\label{e:DecompU}
		U(t)=u_0+\int_0^t \widehat{\cal{A}}_\omega(f)(Y(s))\diff s + M_u(t), 
	\end{equation}
	where $(M_u(t))$ is a square integrable martingale, with increasing process given by 
	\[
		\croc{M_u}(t)= \int_0^t \widehat{\cal{A}}_\omega(f^2)(Y(s))-2f(Y(s))\widehat{\cal{A}}_\omega(f)(Y(s)) \diff s. 
	\]
	See Appendix, Section~\ref{s:ProofOU11} for further details on these calculations.
	
	Using Proposition~\ref{p:ReducedGeneratorNeq11}, we can find some $C>0$ such that for $u\in [u_*-2\eta, u_*+2\eta]$ and $r\geq R_0$
	\begin{align}
		&|\widehat{\cal{A}}_\omega(f)(u,r)-F_\omega(u)|\leq C (\log r)^{-1}\label{e:BoundAf1} \\ 
		&|\widehat{\cal{A}}_\omega(f^2)(u,r)-2f(u,r)\widehat{\cal{A}}_\omega(f)(u,r)|\leq Cr^{-\gamma}\label{e:BoundAf2}
	\end{align}
	
	Let $v(t)$ be the (deterministic) solution of the ODE
	\[
		\dot v(t)= F_\omega(v(t)), 
	\]
	starting from $u_0$. 
	Given the definition of $F_\omega$ and the assumption on $F'_\omega(u_*)$, taking $\eta>0$ small enough, we can find some $\beta, L>0$ such that, for all $u, u'\in [u_*-2\eta, u_*+2\eta]$,
	\begin{equation}\label{e:FOmegaLipsch}
		 (u-u_*)F_\omega(u)\leq -\beta(u-u_*)^2\quad \text{and}\quad  |F_\omega(u)-F_\omega(u')|\leq L|u-u'|. 
	\end{equation}
	Therefore, since $u_0\in [u_*-\eta, u_*+\eta]$, $v(t)$ stays in the same interval for any $t\geq 0$, and  
	\[
		|v(t)-u_*|\leq e^{-\beta t}|u_0-u_*|\leq \eta e^{-\beta t}\leq \eta. 
	\]
	We set $T>0$ such that $e^{-\beta T}<1/4$.  
	
	For $t\in [0, T]$, assuming $t\leq \tau^r_{R_0}\wedge \tau^u_{4\eta}$, 
	\begin{align*}
		|U(t)-v(t)|&\leq \int_0^t L|U(s)-v(s)|\diff s + CT (\log R_0)^{-1}+ \sup_{t\in[0, T]}|M_u(t)|\\
		&\leq e^{Lt}\left(C T (\log R_0)^{-1}+ \sup_{t\leq T} |M_u(t)|\right),
	\end{align*}
	where we first used Equations~\eqref{e:DecompU},~\eqref{e:BoundAf1} and~\eqref{e:FOmegaLipsch} and the definition of $(v(t))$, then Gronwall's Lemma in the second inequality. 
	We want to upper bound the probability of the event $\{\sup_{t\leq T\wedge \tau^r_{R_0}\wedge \tau^u_{4\eta}} |U(t)-v(t)| >\eta/2\}$. 
	Taking $R_0$ large enough,
	\[
		C T (\log R_0)^{-1} \leq \frac{\eta}{4e^{LT}}.
	\]
	We will now upper bound the probability of the event 
	\[
		\cal{E}_{(u_0, r_0)}^T:=\left\{\sup_{t\leq T} |M_u(t\wedge \tau^r_{R_0}\wedge \tau^u_{4\eta})|>\frac{\eta}{4e^{LT}}\right\}. 
	\]
	We control the martingale using an inequality known as Freedman's inequality, see Proposition~\ref{p:Freedman} in Section~\ref{Tools} of the Appendix, or the reference~\cite{Shorack2009} (page {899-900}). This inequality bounds the probability that $\sup_{t\leq T} |M_t|\geq a$, given that $\croc{M}(T)\leq b$ and \emph{the jumps of the Martingale are upper bounded} by $c$. 
	
	Here, the jumps of the martingale, i.e., the elements of the set $\{|M_u(t)-M_u(t{-})|, \, t\in [0, T]\}$ are uniformly upper bounded by $C_1R_0^{-\omega_2}$ for some $C_1>0$ large enough. Moreover, Equation~\ref{e:BoundAf2} gives, almost surely,
	\[
		\croc{M_u}(T\wedge \tau_{R_0}^r\wedge \tau_{2\eta}^u )\leq \frac{C T}{R_0^{\omega_2}}  
	\]
	Therefore, using Equation~\eqref{e:Freedman} of Proposition~\ref{p:Freedman}, we get for $\eps>0$, 
	\[
		\P\left( \sup_{t\leq T} |M_u(t\wedge \tau_{R_0}^r\wedge \tau_{2\eta}^u)|\geq \eps\right) \leq 2\exp\left( -\frac{\eps^2R_0^{\omega_2}}{C_2T} \psi\left( \frac{\eps C_1}{C T} \right)\right), 
	\]
where for $x>0$, 
\[
	\psi(x):=\frac{2}{x^2}\int_0^x \log(1+y)\diff y\geq \frac{1}{1+x/3}.  
\]
We can find a $C'>0$ small enough, independent of $R_0$, such that for $R_0$ large enough, 
	\[
		\P\left( \sup_{t\leq T} |M_u(t\wedge \tau_{R_0}^r\wedge \tau_{2\eta}^u)|\geq \frac{\eta}{4e^{LT}}\right) \leq e^{-C'R_0^{\omega_2}},
	\]
	which leads to 
	\[
		\P_{(u_0, r_0)}(\cal{E}_{(u_0, r_0)}^T)\leq e^{-C'R_0^{\omega_2}}. 
	\]
	On the event $( \cal{E}_{(u_0, r_0)}^T)^c$, for $t\leq T\wedge \tau_{R_0}^r$, 
	\[
		|U(t)-u_*|\leq |U(t)-v(t)|+|v(t)-u_*| \leq \frac{\eta}{2}+\eta<2\eta
	\]
	and at $T$, if $T<\tau_{R_0}^r$, using $e^{-\beta T}<1/4$, we get
	\[
		|U(T)-u_*|\leq |U(T)-v(T)|+|v(T)-u_*| \leq \frac{\eta}{2}+\frac{\eta}{4}<\eta, 
	\]
	which concludes the proof of Equation~\eqref{e:ACControlU1}.

	The second equation is a consequence of the Markov property of $(U(t), R(t))$. Let $\alpha, \alpha'>0$.
	For $t\geq 0$, we introduce the event
	\[
		\widetilde{\cal{E}}_{t}=\{|U(t)-u_*|\ind{t<\tau_{R_0}^r}<\eta\}. 
	\]
	
	We have 
	\begin{align*}
		\P_{(u_0, r_0)}&\left( \tau_{2\eta}^u\geq TR_0^\alpha \wedge \tau_{R_0}^r\right)\\
		&\geq \P_{(u_0, r_0)}\left(\bigcap_{k=0}^{\lfloor R_0^\alpha\rfloor}\{\sup_{t\in [0,T]} |U(t+kT)-u_*|<2\eta\}\cap \widetilde{\cal{E}}_{(k+1)T} \right)\\
		&\geq \E_{(u_0,r_0)}\left[\hspace{-0.1cm}\prod_{k=0}^{\lfloor R_0^\alpha\rfloor\wedge \lceil\tau_{R_0}^r/T\rceil}
		\P\left(\{\sup_{t\in [0,T]} |U(t+kT)-u_*|<2\eta\}\cap \widetilde{\cal{E}}_{(k+1)T} \Big| \, \widetilde{\cal{E}}_{kT}\right)\right]\\
		&\geq \left[1-\exp\left(-CR_0^{\omega_2}\right)\right]^{\lfloor R_0^\alpha\rfloor+1}\\
		&\geq 1-R_0^{-\alpha'}
	\end{align*}
	where in the third inequality, we used Markov property and Equation~\eqref{e:ACControlU1}, and the last inequality is valid for $R_0$ large enough. This concludes the proof of Relation~\eqref{e:ACControlU2}. 
\end{proof}

\subsection{Proof of dissipation}
\begin{proof}[Proof of Proposition~\ref{p:DissipAC}]
	We keep the notations of Lemma~\ref{l:ACControlU}, and take $\eta>0$ small enough so that we can find a $\beta'>0$ such that 
	\[
		\forall u\in [u_*-2\eta, u_*+2\eta]: \quad G_\omega(u)<-\beta'. 
	\]
	We introduce the scaled process 
	\[
		\overline{R}(t):=\frac{R(tr_0^{1-\omega_2})}{r_0}
	\]
	and the stopping times $\overline\tau_{2\eta}^u:=\tau_{2\eta}^u/r_0^{1-\omega_2}$, $\overline\tau_{\eps r_0}^r:=\tau_{\eps r_0}^r/r_0^{1-\omega_2}$ and $\overline \tau=\overline\tau_{2\eta}^u\wedge \overline\tau_{\eps r_0}^r$. 
	
	For $t\leq \overline\tau$, 
	\[
		\overline{R}(t)^b= 1 +\int_0^t b\overline{R}(s)^{b-1+\omega_2}G_\omega(U(sr_0^{1-\omega_2}))\diff s +\int_0^t E_r(U(sr_0^{1-\omega_2}), R(sr_0^{1-\omega_2}))\diff s+ M_r^b(t), 
	\]
	where $E_r(u,r)$ contains the lower order terms, $E_r(u,r)\to 0$ for $r$ large enough and uniformly in $u$, and $(M_r^b(t))$ is a martingale, also converging to $0$ when $r\to +\infty$. Therefore, for $t\leq \overline \tau$, 
	\[
		\overline{R}(t)^b\leq  1 - \beta'b\eps^{b-1+\omega_2} t +\int_0^t E_r(U(sr_0^{1-\omega_2}), R(sr_0^{1-\omega_2}))\diff s+ M_r^b(t), 
	\]
	Choosing $t_0>2(1-\eps^b)/b\beta'\eps^{b-1+\omega_2}$, the last inequality gives 
	\[
		\lim_{r_0\to +\infty}\P\left(\overline\tau_{\eps r_0}^r\wedge \overline\tau_{2\eta}^u<t_0 \right)=1
	\]
	and because of Lemma~\ref{l:ACControlU}, 
	\[
		\lim_{r_0\to +\infty}\P\left(\overline\tau_{2\eta}^u>t_0 \right)=1. 
	\]
	Combining these two limits concludes the proof of dissipation for the process $(U(t), R(t))$.

	To conclude the Proof of Proposition~\ref{p:DissipAC}, we only have to reverse the time change introduced in the infinitesimal generator~\eqref{e:GenURNeq11}. This can be done using the same arguments as in Section~\ref{s:TimeChange}. 
\end{proof}

\section{Drift regions}\label{s:Drift}

In this section, we show that in the regions where a simple drift dominates (henceforth called drift regions, defined below), a well-chosen linear energy function is dissipated. 
Let $\omega_A<\omega_B\in \Omega$ and $m_A, m_B>0$. The set $\cal{T}(\omega_A, \omega_B; m_A, m_B)$ is a \emph{drift region} if
\begin{itemize}
	\item for all $\omega\in (\omega_A, \omega_B)$, $\cal{P}_\omega=\emptyset$
	\item if $\cal{P}_{\omega_A}\neq \emptyset$, $m_A>\max\{\cal{P}_{\omega_A}\}$,
	\item if $\cal{P}_{\omega_B}\neq \emptyset$, $m_B<\min\{\cal{P}_{\omega_B}\}$.
\end{itemize}
In a drift region, the stochastic system behaves similarly to the deterministic system resulting from a leading order expansion of the generator. Our goal is to show that we can find one linear function that is point-wise dissipated in the whole drift region. The region $\omega=(1,1)$ has again a specific behavior, and must be dealt with separately. We consider first a drift region \emph{below the diagonal}, i.e., such that $\omega_B\leq (1,1)$ and if $\omega_B=(1,1)$, $m_B$ chosen as small as necessary. 
Recalling Remark~\ref{r:NonTrivialCompDrift}, note that except for $\omega=(1,1)$, $\cal{P}_\omega\neq \emptyset$ implies $\omega\in \cal{P}$. The first condition in the definition of drift region only has to be checked for a \emph{finite number of $\omega$}. 

We now show that for each such drift region we can find a $c=(1, c_2)\in \R^2$ such that the linear function $f_c(x)=\croc{x, c}$ is dissipated in the whole region. 
\begin{proposition}\label{p:DissipDrift}
		Under Assumptions~\ref{a:Delta2Neq0} and~\ref{a:DeltaNeq0_11}, if $(\omega_A, \omega_B, m_A, m_B)$ defines a drift region with $\omega_B\leq (1,1)$, we can find some $c=(1, c_2)\in \R^2$, $m_B'\in(0, m_B]$ and $K, K'>0$ such that for all $x\in \cal{T}(\omega_A, \omega_B; m_A, m_B)$, if $\|x\|\geq K$ and $K'\leq x_2\leq m_B'x_1$, 
		\[
			f_{c}(x)\geq 0 \quad \text{and} \quad \cal{A}(f_c)(x)\leq -1. 
		\]
\end{proposition}
When the drift region reaches $\omega_B=(1,1)$, to ensure the positivity of $f_c(x)$ when $c_2<0$, we might have to restrict it to $\cal{T}(\omega_A, \omega_B; m_A, m_B')$ with $m_B'<1/|c_2|$. 

\subsection{Structural properties of the drift region and definition of $c$}
Recall that $\cal{C}_\omega=\arg\max\{\croc{y, \omega}, \, y\in \cal{C}^-\}$. 
For $\omega\notin \cal{P}$, the dominant transitions of the Markov process in $\cal{H}_\omega$ are the transitions starting from the \emph{only} complex in $\cal{C}_\omega$. The following Lemma specifies the expression of this complex and relates it to the surrounding regions. 
\begin{lemma}\label{l:BetweenCompact}
	Let $\omega_1<\omega_2\in \Omega$ such that $\cal{P}\cap (\omega_1, \omega_2)=\emptyset$. 
	Then for all $\omega\in (\omega_1, \omega_2)$, $\cal{C}_\omega=\{y\}=\cal{C}_{\omega_1}\cap\cal{C}_{\omega_2}$. 
	Moreover, if $\omega_2\leq (1,1)$, 
	\[
		y=\arg\max\{y_2, \, y\in \cal{C}_{\omega_1}\} = \arg\min\{y_2, \, y\in \cal{C}_{\omega_2}\}
	\]
\end{lemma}
The proof relies on a convexity argument. 
\begin{proof}[Proof of Lemma~\ref{l:BetweenCompact}]
Since $\cal{C}^-$ is finite, a unique maximizer remains unique in a neighborhood of its direction. Since $\cal{P}\cap(\omega_1,\omega_2) =\emptyset$, the function $\omega\mapsto\cal{C}_\omega$ is locally constant on the connected interval $(\omega_1,\omega_2)$. Therefore,
\[
    \forall \omega\in(\omega_1,\omega_2),\quad \cal{C}_\omega=\{y\},
\]
for some $y\in\cal{C}^-$. By continuity, for $i\in \{1,2\}$ and $z\in\cal{C}^-$, 
\[
    \croc{y-z,\omega_i}\geq0,
\]
so $y\in\cal{C}_{\omega_1}\cap\cal{C}_{\omega_2}$. Conversely, every $z$ in this intersection satisfies
\[
    \croc{y-z,\omega_1}=\croc{y-z,\omega_2}=0.
\]
Distinct directions in $\Omega$ are linearly independent, hence $z=y$.

If $\omega_2\leq(1,1)$, write $\omega_i=(1,\gamma_i)$ with $\gamma_1<\gamma_2$, and fix $\gamma\in(\gamma_1,\gamma_2)$. 
For $i=1,2$ and $z\in\cal{C}_{\omega_i}\setminus\{y\}$,
\[
    0<\croc{y-z,(1,\gamma)}
    =\croc{y-z,\omega_i}+(\gamma-\gamma_i)(y_2-z_2)
    =(\gamma-\gamma_i)(y_2-z_2).
\]
Thus, $y_2>z_2$ on $\cal{C}_{\omega_1}\setminus\{y\}$ and $y_2<z_2$ on $\cal{C}_{\omega_2}\setminus\{y\}$, proving both extremal characterizations.
\end{proof}

Assume we are in the setting of Proposition~\ref{p:DissipDrift}. In the drift regions, the system's evolution is well approximated by first order terms of the generator. At a given scaling $\omega$, this term is given by the ODE in Relation~\eqref{ODEOmega}. 
We first give some properties of the function $\Delta^\omega$ when $\cal{P}_\omega=\emptyset$. 
\begin{lemma}\label{l:PropDeltaDrift}
	The function $\Delta^\omega(x,y)$ defined in Relation~\eqref{e:fDelta} satisfies the following properties: 
	\begin{enumerate} 
	\item For $\omega'\notin \cal{P}$, note $\cal{C}_{\omega'}=\{y_{\omega'}\}$, then 
	\[
		\Delta^{\omega'}(x_1, x_2)= x_1^{y_{\omega', 1}}x_2^{y_{\omega', 2}}\Delta^{\omega'}(1,1). 
	\]
	\item Given $\omega_-<\omega<\omega_+\leq (1,1)$ such that $[\omega_-, \omega_+]\cap \cal{P}=\{\omega\}$, 
	\[
		\Delta^\omega(1, u)= \Delta^{\omega_-}(1,1) \, u^{k_-^\omega} +Q_\omega(u) + \Delta^{\omega_+}(1,1)\, u^{k_+^\omega}, 
	\]
	where $k_-^\omega:= \min\{y_2, y\in \cal{C}_\omega\}$ and $k_+^\omega:= \max\{y_2, y\in \cal{C}_\omega\}$, and $Q_\omega(u)$ is a polynomial function constituted of monomials of degrees strictly between $k_-^\omega$ and $k_+^\omega$. 
	\item Under the assumptions of (b),  
	\[
		\Delta_{2}^{\omega_-}(1,1)\neq 0 \quad \Longleftrightarrow \quad \lim_{u\to 0} \frac{\Delta_1^\omega(1,u)}{\Delta_2^\omega (1,u)}= \frac{\Delta_1^{\omega_-}(1,1)}{\Delta_2^{\omega_-}(1,1)}\in \R,
	\]
	and 
	\[
		\Delta_{2}^{\omega_+}(1,1)\neq 0 \quad \Longleftrightarrow \quad \lim_{u\to +\infty} \frac{\Delta_1^\omega(1,u)}{\Delta_2^\omega (1,u)}= \frac{\Delta_1^{\omega_+}(1,1)}{\Delta_2^{\omega_+}(1,1)}\in \R. 
	\]
	\item Under Assumption~\ref{a:Delta2Neq0}, for the drift region defined in Proposition~\ref{p:DissipDrift}, $\omega{\mapsto}\Delta_2^\omega(1, 1)$ is of constant sign on $(\omega_A, \omega_B)$ and for all $\omega\in (\omega_A, \omega_B)$, $\Delta^{\omega}_2(1,1){\neq} 0$. 
	\end{enumerate} 
\end{lemma}

\begin{proof}
	Results (a), (b) and (b) $\Rightarrow$ (c) are consequences of the definition of $\Delta^\omega$ given in Relation~\eqref{e:fDelta} and Lemma~\ref{l:BetweenCompact}. We show (d). Because of Lemma~\ref{l:BetweenCompact}, for $\omega_1<\omega_2$ such that $(\omega_1, \omega_2)\cap \cal{P}=\emptyset$, $\cal{C}_{\omega_1}=\cal{C}_{\omega_2}=\{y^*\}$, and 
	\[
		\Delta^{\omega_1}_2(1,1)=\sum_{y^*\to y'\in \cal{R}} \kappa_{y^*\to y'} (y'_2-y^*_2)=\Delta^{\omega_2}_2(1,1).
	\]
	Therefore, between two elements of $\cal{P}$, $\omega\mapsto\Delta^\omega_2(1,1)$ doesn't change sign. If $\omega\in \cal{P}$, since $\cal{P}_\omega=\emptyset$, $\Delta_2^\omega(1, u)\neq 0$ for all $u\in (0, +\infty)$. Using (c), this shows that $\omega\mapsto\Delta^\omega_2(1,1)$ doesn't change sign even through $\omega\in \cal{P}$. 
	Assumption~\ref{a:Delta2Neq0} and (c) guarantees that for $\omega\notin \cal{P}$, $\Delta^{\omega}_2(1,1)\neq 0$. 
\end{proof}

Using these properties, we can now define the coefficient $c\in \R^2$ of the energy function. Since $c$ is defined up to a constant, we can set $c_1=1$. We choose $c_2$ of opposite sign to the sign of $\omega \mapsto \Delta_2^\omega(1,1)$. Finally, we choose $|c_2|$ such that the leading order terms applied to $x\mapsto \croc{c, x}$ in the generator are negative. 
In the direction $\omega$, for $\omega \in (\omega_A, \omega_B)$, the generator can be approximated by 
\begin{itemize}
	\item If $\omega\notin \cal{P}$, for $x\in \cal{H}^\omega$ we have 
	\[
		\cal{A}(f_c)(x)\approx x^{y_\omega} (\Delta_1^{\omega}(1,1)- |c_2\Delta_2^\omega(1,1)|)
	\]
	\item If $\omega\in \cal{P}$, $\omega<(1,1)$, for $x\in \cal{H}^\omega$, setting $u=x_2/x_1^{\omega_2}$ we have 
	\[
		\cal{A}(f_c)(x)\approx x_1^{\rho_0(\omega)}\left(\Delta_1^{\omega}\left(1, u\right)-|c_2 \Delta_2^{\omega}\left(1, u\right)|\right)
	\]
\end{itemize}
This leads us to set $|c_2|>1$ such that the following inequalities hold:
\begin{equation}\label{e:Defa}
	\begin{aligned}
		|c_2|&> 1+\sup_{\omega\in(\omega_A, \omega_B)} \sup_{u\in \R_+}\frac{\Delta_1^{\omega}\left(1, u\right)}{|\Delta_2^{\omega}\left(1, u\right)|}, \\
		|c_2|&>1+\sup_{u\in (0, m_B]}\frac{\Delta_1^{\omega_B}\left(1, u\right)}{|\Delta_2^{\omega_B}\left(1, u\right)|},\\
		|c_2|&>1+\sup_{u\in [m_A, +\infty)}\frac{\Delta_1^{\omega_A}\left(1, u\right)}{|\Delta_2^{\omega_A}\left(1, u\right)|}\quad \text{if}\quad  \omega_A\neq (1,0),\\
		|c_2|&>1+\sup_{u\in [m_A\vee K', +\infty)}\frac{\Delta_1^{\omega_A}\left(1, u\right)}{|\Delta_2^{\omega_A}\left(1, u\right)|}\quad \text{if}\quad  \omega_A= (1,0)\\
	\end{aligned}
\end{equation}
where $K'$ will be set later as large as needed, so that the last inequality will be well-defined. The existence of such a $|c_2|<+\infty$ is a consequence of Lemma~\ref{l:PropDeltaDrift}.

\subsection{Point-wise dissipation under the diagonal}

\begin{proof}[Proof of Proposition~\ref{p:DissipDrift}]
If $c_2\geq 0$, set $m_B'=m_B$.
	If $c_2<0$, set $m_B'=\min\{m_B,1/(2|c_2|)\}$, and write 
	\[
		\overline{m}_B=m_B\delta_{\omega_B<(1,1)}+ m_B'\delta_{\omega_B=(1,1)}.
	\] 
	We prove the proposition by contradiction. Assume there is some sequence $(x_n)\in \cal{T}(\omega_A, \omega_B; m_A, \overline{m}_B)$ such that $\|x_n\|\to +\infty$, $x_{n,2}\leq m_B'x_{n,1}$ and for all $n\geq 1$, $\cal{A}(f_c)(x_n)>-1$.
	We show that a subsequence of $(x_n)$ verifies $\lim_{n\to +\infty}\cal{A}(f_c)(x_n)=-\infty$.
	
	We introduce the sequence $(\rho_n)=(\log(x_{n,2})/\log(x_{n,1}))$, with value in $[\omega_{A,2}+o(1), \omega_{B, 2}+o(1)]$. Up to taking a subsequence, we can consider that $(\rho_n)$ converges, to $\rho$. We now study separately the cases depending on the value of $\rho$. 

	If $\omega=(1,\rho)\notin \cal{P}\cup \{(1,0)\}$, set $y_\omega\in \cal{C}_\omega$. 
	For $y\in \cal{C}^-$ with $y\neq y_\omega$,  
	\begin{multline}\label{e:DominantComplex}
		\frac{x_n^y}{x_n^{y_\omega}}=\exp\left(\log(x_{n,1})(\croc{y, (1, \rho_n)}-\croc{y_\omega, (1, \rho_n)}) \right)\\
					=\exp\left(\log(x_{n,1})(\underbrace{\croc{y-y_{\omega}, \omega}}_{<0} +(y_{2}-y_{\omega, 2})\underbrace{(\rho_n-\rho)}_{\to 0}) \right)\to 0
	\end{multline}
	where in the last line, we take the limit $n\to +\infty$. 
	Therefore, 
	\begin{align*}
		\left|\frac{\cal{A}(f_c)(x_n)}{x_n^{y_\omega}}-\croc{\Delta^\omega(1, 1), c}\right|&\leq 2\|c\|K_0\sum_{y\to y'\in \cal{R}\setminus\cal{R}_\omega}\kappa_{y\to y'}\frac{x_n^y}{x_n^{y_\omega}}+O\left(x_{n,1}^{-1}+x_{n,2}^{-1}\right),\\
			&\to 0, 
	\end{align*}
	where in the second line, the $O(x_{n, 1}^{-1}+ x_{n, 2}^{-1})$ stands for the correction between the discrete rate and $x^y$.
	Since $\croc{\Delta^\omega(1, 1), c}<0$ by definition of $c$, see Relation~\eqref{e:Defa}, this leads to a contradiction. 
	
	We now assume that $\omega=(1, \rho)\in \cal{P}\cup\{(1,1)\}$. Note that the case $\omega=(1,1)$ can only happen if $\omega_B=(1,1)$. 
	We set $y^*=\arg\min\{y_2, \, y\in \cal{C}_\omega\}$. Using the decomposition given in (b) of Lemma~\ref{l:PropDeltaDrift}, we have 
	\[
		|\Delta_{2}^\omega(x_n)|= c_*x_n^{y^*}(1+ R(\frac{x_{n,2}}{x_{n,1}^{\rho}}))
	\]
	where $c_*>0$ and $1+R(u)$ is lower bounded by some $\eta>0$. 
	For the same reasons as in Relation~\eqref{e:DominantComplex}, for $y\notin \cal{C}_\omega$, $(x_n^{y-y^*})\to 0$ when $n\to +\infty$, which leads to 
	\[
		\lim_{n\to +\infty}\frac{x_n^y}{|\Delta^\omega_2(x_n)|}=0.
	\]
	Therefore, 
	\begin{align*}
		\frac{\cal{A}(f_c)(x_n)}{|\Delta^\omega_2(x_{n,1}, x_{n,2})|}&=-|c_2|+\frac{\Delta^\omega_1(1, x_{n,2}/x_{n,1}^{\rho})}{|\Delta^\omega_2(1, x_{n,2}/x_{n,1}^\rho)|} +o(1)\\
			&\leq -1,
	\end{align*}
	where we used the definition of $c$ given in Relation~\eqref{e:Defa} for the last line. Since $|\Delta_{2}^\omega(x_n)|\to +\infty$ when $n\to +\infty$, this leads to a contradiction. 
	
	Finally, when $\rho=0$ and $\omega=(1, \rho)$, we set $y^*=\arg\max\{y_2, y\in \cal{C}_\omega\}$ and take $\omega^+$ such that $((1,0), \omega^+]\cap\cal{P}=\emptyset$. 
	Here, all the $y\in \cal{C}_\omega$ share the same $y_1$ by definition. Therefore, $\Delta^\omega(x_n)= x_{n,1}^{y_1^*}\Delta^\omega(1, x_{n,2})$.  
	Using the decomposition given in (b) of Lemma~\ref{l:PropDeltaDrift}, for every $\eps>0$ we can find $K'$ large enough such that for all $u\geq K'$, 
	\[
		\left\|\frac{\Delta^\omega(1, u)}{u^{y_2^*}}-\Delta^{\omega^+}(1,1)\right\|<\eps. 
	\]
	Therefore, when $x_{n,2}\geq K'$, 
	\begin{align*}
		\frac{\cal{A}(f_c)(x_n)}{x_{n,1}^{y_1^*}}&\leq  x_{n,2}^{y_2^*}(\croc{\Delta^{\omega^+}(1,1), c} +\eps) +o(1)
	\end{align*}
	where the ``o(1)'' contains the lower order reactions that can be dealt with as before. Using the definition of $c$ given in Relation~\eqref{e:Defa} and taking $\eps$ small enough, this leads again to a contradiction, which concludes the proof of Proposition~\ref{p:DissipDrift}. 
\end{proof}

\subsubsection*{$b$-dominated energy function}
In Theorem~\ref{t:Main}, the energy function is of order $\|x\|^b$. The following straightforward Corollary extends the dissipation in drift regions to the function $g^b_c(x)$ defined as follows:
\begin{equation*}
    g^b_c(x)= (\croc{x, c})^b. 
\end{equation*}

\begin{corollary}\label{c:DriftPowerb}
    For $b\geq 1$, for some $\beta>0$, up to an increase of the constants $K, K'>0$, for all $x\in \cal{T}(\omega_A, \omega_B; m_A, m_B)$, if $\|x\|\geq K$ and $K'\leq x_2\leq m_B'x_1$, 
		\[
			\cal{A}(g^b_c)(x)\leq -\beta. 
		\]
\end{corollary}

\subsection*{Dissipation on the boundaries}
To facilitate the patching of the drift regions and the non-trivial compactifications, we show a stronger dissipation inequality near the boundary of the drift region. 

For $\omega= (1,1)$, we use the change of variable $(u, r)$ introduced in Relation~\eqref{e:ChangeVar11} Section~\ref{s:R11}, associated to the local infinitesimal generator~\eqref{e:IGReduced} (after time-change). 
For $(1,0)<\omega<(1,1)$, we use the change of variable $(u,r)$ of Relation~\eqref{e:ChangeVarNeq11} of Section~\ref{s:RNeq11} associated to the local infinitesimal generator~\eqref{e:IGReducedNeq11}. 
In both case, and as defined in the respective sections, $\underline{x}$ is the function reversing the change of coordinates: $\underline{x}(u(x), r(x))=x$. 

\begin{lemma}\label{l:DriftBoundary}
	Under the conditions of Proposition~\ref{p:DissipDrift}, we can find some $\beta>0$ small enough such that 
	\begin{itemize}
		\item if $\omega_A>(1, 0)$, for $M_A>m_A$, on $\cal{H}_{\omega_A}(m_A, M_A)$, for $r$ large enough, 
		\[
			\widehat{\cal{A}}_{\omega_A}(g^b_c\circ \underline{x})\leq -\beta r^{b+\omega_{A, 2}-1}
		\]
		\item for $M_B<m_B'$, on $\cal{H}_{\omega_B}(M_B, m_B')$, for $r$ large enough, 
		\[
			\widehat{\cal{A}}_{\omega_B}(g^b_c \circ \underline{x})\leq -\beta r^{b+\omega_{B, 2}-1}
		\]
	\end{itemize}
\end{lemma}

\begin{proof}
	This is a direct application of the definition of $c$ given in Relation~\eqref{e:Defa}, and the approximation of the generator $\widehat{\cal{A}}_\omega$ given in Proposition~\ref{p:ReducedGenerator11} and~\ref{p:ReducedGeneratorNeq11}. 
\end{proof}

\subsection{The diagonal region}
In the diagonal region, once again, we have to proceed differently. We assume here that $\cal{P}_{(1,1)}=\emptyset$. 
In $\cal{H}_{(1,1)}$, the order of magnitude of the states is fixed, we can therefore use the reduced generator introduced in Relation~\eqref{e:IGReduced} (after time-change), associated to the change of variable $(u, r)$ introduced in Relation~\eqref{e:ChangeVar11} Section~\ref{s:R11}.

Our goal is to show that the system exits the drift region. 
Until now, dissipation was captured by the monotonic (increasing or decreasing) behavior of the quantity $x_2$ in time. Here, the natural variable capturing dissipation through its monotonic evolution is the \emph{angle}. 
Note that the constant sign of $\Delta_2^{(1,1)}(1,u)$ is not enough to guarantee this. 

The easiest and most intuitive way to construct the energy function here is to take a function $V_d(u,r)=r^b\Phi(u)$ where $\Phi$ is a monotonic function of the angle $u$. More specifically, the function $u\mapsto \Phi(u)$ should be increasing if the $u(t)$ decreases, i.e., if the system turns clock-wise on the Poincar\'e disk, and decreasing if the system turns anti-clock-wise on the Poincar\'e disk. 

Taking the constant $M$ of Theorem~\ref{t:Main}, we choose the energy function $\Phi$ as $\Phi(u)=(1+u)^a$ with $a$ such that $aF_\omega(1)<0$, and 
\begin{equation*}
	|a|\geq 1+ \max\left\{0, \sup_{u\in (0, M+1]}\frac{(1+u)bG_\omega(u)}{|F_\omega(u)|}\right\}. 
\end{equation*}

\begin{lemma}\label{l:DriftDiag}
	Let $b>0$. If  $\cal{P}_\omega=\emptyset$ for $\omega=(1, 1)$ and Assumption~\ref{a:DeltaNeq0_11} holds, then for any $m\in (0, M)$ we can find $\beta, K>0$ such that, for all $(u,r)$ with $u\in[m, M]$ and $r\geq K$, 
	\[
		\widehat{\cal{A}}_\omega(V_d)(u,r)\leq - \beta r^{b}. 
	\]
\end{lemma}

\begin{proof}
	$a$ is well-defined because of Assumption~\ref{a:DeltaNeq0_11}. Analougously to what was done in the proof of Proposition~\ref{p:DissipDrift}, the inequality is a direct application of the approximation of the generator given in Relation~\eqref{e:IGReduced}. 
\end{proof}

\section{Patching of the different regions}\label{Gluing}

In this section, we patch the different regions together, completing the final steps in the proof of Theorem~\ref{t:Main}. Recall that $b\geq 1$ is fixed. 
We patch the regions starting from the diagonal $\omega=(1,1)$ down to the boundary $\omega=(1,0)$. This procedure can be carried out on the other half of the state space by exchanging the two species, see Remark~\ref{r:SwitchSpecies}.  

In the spirit of what has been done in Section~4 of~\cite{Agazzi25}, the idea is to multiply by a positive constant the energy function of the ``lower'' set, so that the energies can be patched near the boundary between the lower and upper set. When we meet an attractive, convergent curve, the dissipation is so strong that the patching of the function does not necessitate much adjustment, see Proposition~\ref{p:DissipAC}. When we meet a repulsive curve however, the dissipation has to be shown more carefully. In this case, we use a smoothing argument, see Lemma~\ref{l:SmoothingGluing}. 

The patching is always carried out in a narrow region of the form $\cal{H}_\omega(p_1, p_2)$ for fixed $\omega$. Therefore, we can show the dissipation in the region of the patching using the ``local'' generator, defined by Proposition~\ref{p:ReducedGenerator11} (for $\omega=(1,1)$) and Proposition~\ref{p:ReducedGeneratorNeq11} (for $\omega<(1,1)$).

We use $F_\omega,G_\omega$ from Definitions~\ref{d:Comp11} and~\ref{d:CompNeq11}, and the generator $\widehat{\cal A}_\omega$ after the time change of Sections~\ref{s:R11} and~\ref{s:RNeq11}. For $\omega=(1,\gamma)$, $0<\gamma\leq 1$, recall the change of coordinates: 
\[
    u=\frac{x_2}{x_1^\gamma},\qquad
    r=\begin{cases}
        x_1,&\gamma<1,\\
        \|x\|_2,&\gamma=1.
    \end{cases}
\]

\subsection{Patching with a repulsive region: smoothing lemma}
We state a lemma that will be used to carry out the patching between a repulsive region and a drift region. It is stated in a general setting so that it can be used on the diagonal region as well as sub-diagonal region. The smoothing is always carried out in a cone $\cal{H}_\omega(p_1, p_2)$ for some $0<p_1<p_2$. 

Introduce the interval $I:=[p_1, p_2]$ and an interval $J:=[q_1, q_2]$ such that $I\subset (q_1, q_2)$ and 
\[
    \min_{u\in J}|F_\omega(u)|>0.
\]
Let $V_1,V_2$ be positive $C^2$ functions of $(u,r)$ and $\beta_0,\beta_1,\beta_2,C$ positive constants.

\begin{lemma}\label{l:SmoothingGluing}
If, for $j=1,2$, $u\in J$ and for $r$ large enough, it holds
\begin{equation}\label{e:SmoothingGluing}
    \begin{aligned}
    \widehat{\cal A}_\omega V_j(u, r)\leq-\beta_jr^{b+\gamma-1},\\
    V_1(u,r)-V_2(u,r)\geq\beta_0r^b,\\
    |V_j|+|\partial_uV_j|+r|\partial_rV_j|\leq Cr^b,
    \end{aligned}
\end{equation}
then we can define a function $V$, positive and $C^2$, that verifies, for $u\in J$ and $r$ large enough,
\[
    V_2(u, r)\leq V(u,r)\leq V_1(u,r),\qquad
    V(u,r)=
    \begin{cases}
        V_1(u,r),&u\leq p_1,\\
        V_2(u,r),&u\geq p_2,
    \end{cases}
    \quad\text{if }F_\omega>0,
\]
with the two last equalities interchanged if $F_\omega<0$, and such that for $u\in I$ and $r$ large enough, 
\[
    \widehat{\cal A}_\omega (V)(u, r)\leq-\beta r^{b+\gamma-1},
\]
for some $\beta>0$. This inequality also holds for states $(u,r)$ that can reach $I\times \R_+$ in one jump. 
\end{lemma}

When $F_\omega>0$ on the smoothing region, the vector field tends to increase $u$, and the system would cross the patching region from the lower part ($u=p_1$) to the above part ($u=p_2$). In this case, $V_1$ is the function below the region of patching, and $V_2$ the function above. We chose $V_2\leq V_1$, and therefore the vector field brings us towards lower energy values. 
When $F_\omega<0$, the function with lowest values is on the bottom, and we again go toward the low energy region. 

When using this lemma, we will see that the expression of $V_1$ and $V_2$ as well as the first inequality required in~\eqref{e:SmoothingGluing} are consequences of the work we have done until now. The last inequality can be verified using explicit definitions of the functions $V_i$, whereas the $b$-dominated condition gives the comparison of their values. The second inequality will be obtained by multiplying the lower-function by a positive constant.  

\begin{proof}
We set $I'\subset J$ such that $I\subset I'\setminus\partial I'$, to ensure that the dissipation inequality holds also for states at the boundary of the set $I$ that can reach $I^c$ in one jump. 
For $\nu\in \cal{R}$ and function $H$, we write $\partial_\nu H= H(x+\xi_\nu)-H(x)$. 
Recall the estimations (and notations) calculated in the proofs of Propositions~\ref{p:ReducedGenerator11} and~\ref{p:ReducedGeneratorNeq11}: uniformly on $J$ and for all $\nu\in \cal{R}$, 
\[
    |\partial_\nu u|\leq cr^{-\gamma}, \quad |\partial_\nu r|\leq c, \quad  \widehat{\lambda}_\nu(u,r) \leq cr^\gamma
    \quad\text{and} \quad \widehat{\cal{A}}_\omega u =F_\omega(u) +o(1).
\]  
Set for $u\in J$ and $r\geq 1$,  
\[
    V(u,r)=\chi(u) V_1(u,r)+(1-\chi(u))V_2(u,r),
\]
where $\chi$ is chosen so that $0\leq\chi\leq 1$, and $F_\omega\chi'\leq0$ (which ensures the right ``order'' for $V_1$ and $V_2$, as mentioned before) and 
\[
    |\chi'(u)|\propto
    \begin{cases}
        \exp[-1/((u-p_1)(p_2-u))],&p_1<u<p_2,\\
        0,&u\notin(p_1,p_2),
    \end{cases}
    \qquad \int_{p_1}^{p_2}|\chi'(u)|\diff u=1.
\]
For $u\in I'$ and $r$ large enough, 
\begin{equation}\label{e:GenDecomposition1}
    \widehat{\cal A}_\omega V(u,r)= \chi \widehat{\cal A}_\omega V_1(u,r)+(1-\chi)\widehat{\cal A}_\omega V_2(u,r) +(V_1-V_2)\widehat{\cal A}_\omega\chi(u,r) +E_1(u,r),
\end{equation}
where $E_1(u,r)$ is defined as  
\[
    E_1(u,r):=\sum_{\nu\in \cal{R}}\widehat{\lambda}_\nu(u,r)[(\partial_\nu \chi)(\partial_\nu (V_1-V_2)) ], 
\]
Therefore, Taylor expansion and the last inequality of the Lemma lead to 
\[
    |E_1(u,r)|\leq C'r^{\gamma}(cr^{-\gamma}\sup_{|v-u|\leq Cr^{-\gamma}}|\chi'(v)|)\left(Ccr^{b-\gamma}+Ccr^{b-1}\right)\leq C''r^{b-\gamma}\sup_{|v-u|\leq Cr^{-\gamma}}|\chi'(v)|. 
\]
Moreover, using again a Taylor expansion and the approximations recalled above, we have 
\[
    (V_1-V_2)\widehat{\cal A}_\omega\chi(u,r)=(V_1-V_2)\chi'\widehat{\cal A}_\omega u +E_2(u,r), 
\]
where 
\[
    E_2(u,r):=(V_1-V_2)\sum_{\nu\in \cal{R}}\widehat{\lambda}_\nu(u,r)(\partial_\nu\chi-\chi'\partial_\nu u), 
\]
and can be upper bounded using Taylor expansion and the last inequality of the Lemma as: 
\[
    |E_2(u,r)|\leq 2Cr^b\left(cr^\gamma\right)(\frac{(cr^{-\gamma})^2}{2}\sup_{|v-u|\leq Cr^{-\gamma}}|\chi''(v)|)\leq C'r^{b-\gamma}\sup_{|v-u|\leq Cr^{-\gamma}}|\chi''(v)|. 
\]
Therefore, using the assumptions of the Lemma, gathering all these inequalities in Relation~\eqref{e:GenDecomposition1}, we get: 

\begin{multline}\label{e:neg1}
    \widehat{\cal A}_\omega V(u,r)\leq -(\beta_1\wedge \beta_2)r^{b+\gamma-1} -\frac{\beta_0}{2}\min_{v\in J}|F_\omega(v)|r^b |\chi'(u)| \\
    +C'r^{b-\gamma}(\sup_{|v-u|\leq Cr^{-\gamma}}|\chi'(v)|+\sup_{|v-u|\leq Cr^{-\gamma}}|\chi''(v)|). 
\end{multline}
For $p_1<u<p_2$,
\[
    \frac{\chi''(u)}{\chi'(u)}=\frac{p_1+p_2-2u}{(u-p_1)^2(p_2-u)^2}.
\]
The exponential factor vanishes faster than any power at either endpoint. Thus, for every integer $n\geq2$, we can find $C_n>0$ such that for $u\in I$, $|\chi''(u)|{\leq} C_n|\chi'(u)|^{1-1/n}$, and therefore, on I, 
\[
    \left|\frac{\diff}{\diff u}|\chi'|^{1/n}\right|
    =\frac{|\chi''|}{n|\chi'|^{1-1/n}}\leq\frac{C_n}{n}.
\]
Extending $|\chi'|^{1/n}$ by zero outside $I$ therefore gives
\[
    \bigl||\chi'(v)|^{1/n}-|\chi'(u)|^{1/n}\bigr|\leq C_n|v-u|.
\]
This leads to 
\[
    \sup_{|v-u|\leq Cr^{-\gamma}}|\chi'(v)|^{1/n}\leq |\chi'(u)|^{1/n}+C_nr^{-\gamma}
\]
and therefore
\[
    \sup_{|v-u|\leq Cr^{-\gamma}}|\chi'(v)|\leq C_n'(|\chi'(u)|^{1-1/n}+r^{-(n-1)\gamma}).
\]
Using also the boundedness of $|\chi'(v)|$, we have 
\[
     \sup_{|v-u|\leq Cr^{-\gamma}}|\chi''(v)|\leq C_n \sup_{|v-u|\leq Cr^{-\gamma}}|\chi'(v)|^{1-1/n}\leq C_n'(|\chi'(u)|^{1-1/n}+r^{-(n-1)\gamma}). 
\]
Since
\[
		r^{b-\gamma}|\chi'|^{1-1/n} =(r^b|\chi'|)^{(n-1)/n}(r^{b-n\gamma})^{1/n},
\]
Young's inequality with exponents $n/(n-1)$ and $n$ gives, for any $\eps>0$, 
\[
C'r^{b-\gamma}(\sup_{|v-u|\leq Cr^{-\gamma}}|\chi'(v)|+\sup_{|v-u|\leq Cr^{-\gamma}}|\chi''(v)|)\leq\eps r^b|\chi'(u)|+C_{\eps,n}r^{b-n\gamma}. 
\]
This remains valid when $\chi'(u)=0$, including starting points in $I'\setminus I$. 

Choose $\eps<\beta_0\min_J|F_\omega|/4$ and $n\gamma>1-\gamma$. For $r$ large enough,
\[
    \widehat{\cal A}_\omega V
    \leq-\frac{1}{2}\min\{\beta_1,\beta_2\}r^{b+\gamma-1},
\]
which concludes the proof. 
\end{proof}

\subsection{Proof of Theorem~\ref{t:Main}}
The final proof consists in patching step by step the different dissipation results we have shown until now. Let $b\geq 1$ and $M$ be given by Theorem~\ref{t:Main}. 

We start from the $S_2$ boundary (states where $x_2\leq K_2$), increasing $\omega$ until the diagonal region. 
We introduce $\omega_1<\omega_2<\ldots<\omega_q$ such that 
\[
	\{\omega_i\}=\{\omega:\, (1, 0)<\omega\leq (1,1)\quad \text{and}\quad \cal{P}_\omega\neq \emptyset\}. 
\]
The first region is a drift region defined as $\cal{T}((1,0), \omega_1; K_2, v_-)$. Proposition~\ref{p:DissipDrift} gives a linear function $V_B$ (``B'' for ``bottom'') dissipated in this region. 

First assume that $\omega_1\neq(1,1)$. This case is treated below.
We define $u_*=\min\{\cal{P}_{\omega_1}\}$. Because $\omega_1\neq (1,1)$, see Proposition~\ref{p:DissipDrift}, we can choose $v_-$ as close as $u_*$ as needed. 
We distinguish two cases: if $u_*$ corresponds to an attractive curve (i.e., $F'_{\omega_1}(u_*)<0$), we use \emph{Method 1}. If it corresponds to a repulsive curve (i.e., $F'_{\omega_1}(u_*)>0$), we use \emph{Method 2}. 

\smallbreak

\emph{Method 1:} The region of an attractive curve \emph{is used as a region of patching}: we will patch the function $V_B$ to the function defined by the region above using the strong dissipation of the attractive convergent curve. Looking above this curve, we encounter either a repulsive curve (if $\cal{P}_{\omega_1}\setminus\{u_*\}\neq 0$) or another drift region (if $\cal{P}_{\omega_1}=\emptyset$). Both situations allow us to define a function $V_T$ (``T'' for ``top'') that dissipates above the attractive curve, until $v_+>u_*$ as close to $u_*$ as needed. See Proposition~\ref{p:RepLineNeq11} if we meet a repulsive line, and Proposition~\ref{p:DissipDrift} if we meet a drift region. 
Summing up these results, we can find $u_-<v_-<u_*<v_+<u_+$ such that 
\begin{itemize}
	\item The function  $V_B$ dissipates for $u\leq v_-$
	\item The function $V_T$ dissipates for $u\geq v_+$
	\item Both functions can be extended until $u_*$, and are of order $r^b$ in $\cal{H}^{(1,1)}(u_-, u_+)$. 
\end{itemize}
The function $W$ on $\cal{H}^{(1,1)}(u_-, u_+)$ defined as 
\[
	W(x)= V_T(x)\ind{u(x)\geq u_*} +V_B(x)\ind{u(x)< u_*},  
\]
verifies the conditions of Corollary~\ref{c:AttractivePatching}, which guarantees dissipation using Filonov's criterion. 

\smallbreak

\emph{Method 2:} We use here Lemma~\ref{l:SmoothingGluing}. This time, the patching occurs \emph{between} the region of drift and the next region (not around an invariant line: below the curve $u_*$). The function $V_T$ is defined in Proposition~\ref{p:RepLineNeq11}, and is dissipative until $v_-<u_*$, chosen as far as $u_*$ as necessary. Summing up these results, we can find $u_-<v_-<u_*$ such that, defining the patching region $I\subset(u_-,v_-)$ in Lemma~\ref{l:SmoothingGluing}, the inequalities~\eqref{e:SmoothingGluing} are verified:
\begin{itemize}
    \item The first inequalities are given by Proposition~\ref{p:DissipRC11} for $V_T$ and Lemma~\ref{l:DriftBoundary} for $V_B$. 
    \item Multiplying $V_T$ by a constant $\lambda$ ($\lambda<1$ if the system goes towards larger $u$ and $\lambda>1$ otherwise), we get the second inequality. 
    \item The last inequality is given by the explicit definitions of the functions $V_B$ and $V_T$. 
\end{itemize}
The two regions are now patched with a consistent dissipation inequality. 

\smallbreak

This procedure can be iterated until we reach the diagonal. More precisely, in a compactification, repulsive and attractive lines alternate, and we can always use Method~1 here. Patching the top of the compactification to the next drift region uses Method 1 if $u_*'=\max\{\cal{P}_{\omega_1}\}$ is attractive, and Method 2 if $u_*'$ is repulsive. 

When we reach the diagonal region, we are in a drift zone. Write $V_B$ the linear function given by Proposition~\ref{p:DissipDrift}, after all the previous patching have been done (which corresponds to the function of Proposition~\ref{p:DissipDrift}, maybe multiplied by a constant). If $\cal{P}_{(1,1)}\neq \emptyset$, we can use exactly the same methods as before, replacing the Propositions of Section~\ref{s:RNeq11} by the corresponding propositions of Section~\ref{s:R11}. 

If $\cal{P}_{(1,1)}= \emptyset$, we use Method 2 to patch the two drift regions together. 
The function $V_B$ is given by Proposition~\ref{p:DissipDrift}, and is dissipative until $u=m'$ small enough. 
The function $V_T$ is given by Lemma~\ref{l:DriftDiag}, and dissipate until $u=m>0$ as small as needed. We fix  $m<m'$ and the interval of patching $I\subset (m, m')$. The first inequalities of Relation~\eqref{e:SmoothingGluing} are given by Lemma~\ref{l:DriftDiag} and Lemma~\ref{l:DriftBoundary}. 

In the case where $q=0$, the diagonal is patched with the drift region that covers the whole half space. 
The $b$-dominance on the boundaries of the state space is preserved by definition of the energy functions. 

This concludes the proof of Theorem~\ref{t:Main}.\qed

\section*{Acknowledgements}
The authors thank Lukas Baumgartner, Jonathan C. Mattingly and Philippe Robert for helpful discussions. AA acknowledges partial support of the Swiss National Science Foundation project grant 2000-1-243054, and of PRIN 2022 project ConStRAINeD (2022XRWY7W). AA and LL are grateful to the Mathematics Department at the University of Pisa for the hospitality during the development of this project.

\section*{AI disclosure}
Artificial intelligence tools were used to spell-check and improve readability of the paper and to streamline the proof of Lemmas~\ref{l:BetweenCompact},~\ref{l:PropDeltaDrift}, and~\ref{l:SmoothingGluing}. The authors take full responsibility for the content and the presentation of this work.

\bibliographystyle{plain}
\bibliography{ref}

\appendix
\section{Technical tools}\label{Tools}

\subsection{Energy function and Lyapunov-type criteria}\label{s:Lyapunov}

 An \emph{energy function} $V: \N^2\to \R_+$ is a function that verifies
	\begin{equation*}
		\lim_{\|x\|_1\to +\infty} V(x)=+\infty. 
	\end{equation*}

Let $(X(t))$ be a CTMC on $\N^2$ with infinitesimal generator $\cal{A}$. 

\begin{proposition}[Foster-Lyapunov criterion]
	If there exists an energy function $V$ and some constants $\beta, K>0$ such that 
	\[
		\cal{A}(V)(x)\leq -\beta \qquad \text{for all } \|x\|\geq K\,,
	\]
	then the process $(X(t))$ is positive recurrent on each closed irreducible class. 
\end{proposition}

We introduce the first instant of jump of the process as follow: 
\[
	t_1:=\inf\{t>0: X(t)\neq X(t{-})\}
\]
\begin{proposition}[Filonov's criterion]
	If there exists an energy function $V$, some constants $\eta, \beta, K>0$ and an integrable stopping time $\tau\geq t_1\wedge \eta$, such that 
	\[
		\E_x\left(V(X(\tau))\right)-V(x)\leq -\beta\E_x(\tau) \qquad \text{for all } \|x\|\geq K\,,
	\]
	then the process $(X(t))$ is positive recurrent on each closed irreducible class. 
\end{proposition}
See~\cite{Filonov1989} for the historical reference, and~\cite{LR23} for the introduction of this criterion for CRNs. 

Note that when we set $\tau=t_1$ we get: 
\[
	\frac{\E_x\left(V(X(\tau))\right)-V(x)}{\E_x(\tau)}=\cal{A}(V)(x). 
\]
Foster-Lyapunov inequality corresponds to Filonov's inequality when $\tau=t_1$. When the first one is verified, we say that \emph{the dissipation is point-wise}.

\subsection{Control of martingales: Freedman's Inequality}
We state the result as in~\cite{Shorack2009} (page {899-900}). See also~\cite{Dzhaparidze2001}, or~\cite{Freedman1975} for the historical reference, with discrete time martingales. 
\begin{proposition}\label{p:Freedman}
	If $(M(t))$ is a real valued local Martingale starting from $0$ such that 
	\[
		\sup_{t\geq 0}|M(t)-M(t{-})|\leq c
	\]
	for some $c>0$, then for all $t\geq 0$, $\lambda>0$ and $\sigma>0$, we have 
	\begin{equation}\label{e:Freedman}
		\P\left(\sup_{s\leq t}|M(s)|\geq \lambda \right)
		\leq 2\exp\left(-\frac{\lambda^2}{2\sigma}\psi\left(\frac{\lambda c}{\sigma}\right)\right) +\P\left(\croc{M}(t)\geq \sigma\right), 
	\end{equation}
	where for $x> 0$, 
	\[
		\psi(x):=\frac{2}{x^2}\int_0^x \log(1+y)\diff y.  
	\]
\end{proposition}

Note that the proof of this result consists in showing that the process $(\exp(rM(t)-\phi_c(r)\croc{M}(t)))$ where $\phi_c(r)=(e^{rc}-1-rc)/c^2$, for $r\geq 0$, is a super-martingale.

\subsection{Patching two functions easily}\label{PatchingJonathan}
The following Lemma gives a simple argument to patch energy functions near a curve where their value match, and conserve the dissipation property. It can only be applied when the functions are ordered the right way on each side of the curve. 
\begin{lemma}\label{l:GluyingJonathan}
	Assume there is a curve $x_2=h(x_1)$ strictly increasing, going to infinity, separating $\Z_{\geq 0}^2$ in two domains, $L$ and $R$. 
	Moreover, assume that 
	\begin{itemize}
		\item there is an Energy function $V_L$ on $L\cup \partial^{K_0}(L)$ such that for all $x\in L\cup \partial^{K_0}(L)$ with $\|x\|$large enough, $\cal{A}(V_L)(x)\leq -1$,
		\item there is an Energy function $V_R$ on $R\cup \partial^{K_0}(R)$ such that for all $x\in R\cup \partial^{K_0}(R)$ with $\|x\|$ large enough, $\cal{A}(V_R)(x)\leq -1$, 
		\item for $x\in \partial^{K_0}(L)\cap R$ with $\|x\|$ large enough, $V_R(x)\leq V_L(x)$
		\item for $x\in \partial^{K_0}(R)\cap L$ with $\|x\|$ large enough, $V_L(x)\leq V_R(x)$. 
	\end{itemize}
	then setting the energy function 
	\[
		V(x)=\begin{cases}
			V_L(x) \quad \text{when} \quad x\in L\\
			V_R(x) \quad \text{when} \quad x\in R
		\end{cases}\, , 
	\]
	for all $x\in \Z_{\geq 0}^2$ with $\|x\|$ large enough, $\cal{A}(V)(x)\leq -1$. 
\end{lemma}

\begin{proof}
	The proof is given in the proof of Lemma 4.2. of \cite{Agazzi25}. 
\end{proof}

\section{Proof of Lemma~\ref{l:OU11}}\label{s:ProofOU11}

With the notations of Lemma~\ref{l:OU11}, we set 
\[
	(Y^N(t))=(U^N(t), R^N(t)), \quad y_N=(u_N, N) \quad \text{and}\quad \tau_N=\tau_{L^N}.
\]
We introduce the stopped processes as follows: for all $t\geq 0$, 
\begin{equation*}
	L^N_{\eta}(t):=L^N(t\wedge \tau_N)\quad \text{and}\quad \overline{R}_\eta^N(t):=\frac{R^N(t\wedge \tau_N)}{N}. 
\end{equation*}	

Then, for $t\geq 0$, by Theorem 4.6.1 in \cite{KurtzEthier}, we can write 
\begin{equation}\label{e:sdeL}
	L^N_\eta(t)=L^N_\eta(0)+\int_0^{t\wedge \tau_N} \widehat{\cal{A}}_\omega(\ell)(Y^N(s))\diff s +M_L^N(t), 
\end{equation}
where $(M_L^N(t))$ is a square integrable martingale with increasing process given by, for $t\geq 0$, 
\[
	\croc{M_L^N}(t)=\int_0^{t\wedge \tau_N} \Lambda(\ell)(Y^N(s))\diff s, 
\]
with $\Lambda$ that verifies, for $f\in \cal{C}^1$, 
\begin{align*}
	\Lambda(f)(u, r)&:= \frac{(1+u^2)^{(\rho_0+1)/2}}{r^{\rho_0-1}} \sum_{\nu\in\mathcal R}\lambda_\nu(x(u, r))\left[f\left(u(x(u, r)+\xi^\nu),r(x(u, r)+\xi^\nu)\right)-f(u,r)\right]^2\\
	&= \widehat{\cal{A}}_\omega (f^2)(u,r)-2f(u, r)\widehat{\cal{A}}_\omega (f)(u,r). 
\end{align*}
In the last line we used the decomposition
\[
	(f(u', v')-f(u,v))^2=(f(u', v')^2-f(u,v)^2)-2f(u,v)(f(u', v')-f(u,v)).
\]

Similarly, for $\beta\in \R\setminus\{0\}$, setting $g_\beta(u, r)=r^\beta$, 
\begin{equation}\label{e:sdeR}
	(\overline{R}_\eta^N(t))^\beta = 1+ \frac{1}{N^\beta}\int_0^{t\wedge \tau_N} \widehat{\cal{A}}_\omega(g_\beta)(Y^N(s))\diff s +M_{R^\beta}^N(t), 
\end{equation}
where $(M_{R^\beta}^N(t))$ is a square integrable martingale with increasing process given by, for $t\geq 0$, 
\begin{equation}\label{e:sdeMR}
	\croc{M_{R^\beta}^N}(t)=\frac{1}{N^{2\beta}}\int_0^{t\wedge \tau_N} \Lambda(g_\beta)(Y^N(s))\diff s. 
\end{equation}

\subsection{Control of $\overline{R}_\eta^N(t)$}
\begin{lemma}\label{l:ControlR11}
	For $\beta\in \R\setminus\{0\}$, the following limit holds:
	\begin{equation}\label{e:EsupMR}
		\lim_{N\to +\infty}\E_{y_N}\left(\sup_{t\leq T}|M_{R^\beta}^N(t)|\right)=0. 
	\end{equation}
	Moreover, we can find a $C_\beta\geq 0$ such that
	\begin{equation}\label{e:EsupR}
		\sup_{N\geq 1}\E_{y_N}\left(\sup_{t\leq T}(\overline{R}_\eta^N(t))^{\beta} \right)\leq 2e^{C_\beta T}. 
	\end{equation}
\end{lemma}

Note that when $\beta<0$, this controls excursions towards zero of the process $(\overline{R}_\eta^N(t))$. 

\begin{proof}
Let $\beta\in \R$, $\beta\neq 0$. 
Since until $\tau_N$, the state $(u,r)$ of the process verifies $|\ell(u, r)|\leq \eta_1$, we have $|u-u_*|\leq \eta_1/\sqrt{r}$. Therefore, using Proposition~\ref{p:ReducedGenerator11} and the fact that $G_\omega(u)$ is a polynomial function, we can find $C_G, C_{\cal{E}}, C_\beta'>0$ such that 
\begin{align}
	\begin{split}\label{e:AgRleq}
		\widehat{\cal{A}}_\omega(g_\beta)(u, r) 
		&=\beta r^\beta G_\omega(u) +\cal{E}_{g_\beta}(u,r)\\
		& \leq r^\beta(\beta G_\omega(u_*) + C_G\eta_1r^{-1/2}+C_{\cal{E}}r^{-1/2})\\
		&\leq C_\beta r^\beta,
	\end{split} 
	\\
	\Lambda(g_\beta)(u, r)&\leq C_{\beta}' r^{2\beta-1},\nonumber
\end{align}
where $C_\beta:=|\beta| G_\omega(u_*) + C_G\eta_1+ C_{\beta}'$, and we used the fact that $r\geq 1$. 

Therefore, for $t\geq 0$, Relation~\eqref{e:sdeR} leads to 
\[
	\sup_{t\leq T}\E_{y_N}\left(\overline{R}_\eta^N(t)^\beta \right)\leq 1 + C_\beta \int_0^T \sup_{s\leq t}\E_{y_N}\left(\overline{R}_\eta^N(s)^\beta \right)\diff t
\]
and Gronwall's Lemma gives 
\[
	\sup_{t\leq T}\E_{y_N}\left(\overline{R}_\eta^N(t)^\beta \right)\leq e^{C_\beta T}. 
\]

Equation~\eqref{e:sdeMR} gives 
\[
	\E_{y_N}\left(\croc{M_{R^\beta}^N}(T)\right)\leq \frac{C_\beta'}{N}\int_0^T \E_{y_N}\left((\overline{R}_\eta^N(s))^{2\beta-1}\right)\diff s\leq \frac{3C_{\beta}'Te^{C_{2\beta-1}T}}{N}\to 0, 
\]
and using Doob's inequality, we get Relation~\eqref{e:EsupMR}.

Applying Equation~\eqref{e:AgRleq} to Relation~\eqref{e:sdeR} again, we get for all $t\in [0,T]$,
\[
	\sup_{s\leq t}(\overline{R}_\eta^N(s))^\beta \leq 1+C_\beta \int_0^t \sup_{s'\leq s}\overline{R}_\eta^N(s')^\beta\diff s+\sup_{s\leq t}|M_{R^\beta}^N(s)|, 
\]
and using the previous limit and Gronwall's Lemma once again, we get Relation~\eqref{e:EsupR}. 
\end{proof}

\subsection{Tightness}
Set $\theta_N=\tau_N\wedge T$.

\begin{lemma}\label{l:Tightness11}
	The sequence $((L^N_\eta(t), \overline{R}^N_\eta(t))_{t\in[0,T]},\theta_N)$ is tight for the product of the $J_1$-Skorokhod topology on the path coordinate and the usual topology on $[0,T]$.
	Moreover, any limit of the process is continuous. 
\end{lemma}

Note that we include $(\theta_N)$ in the sequence to be able to choose a subsequence that converges for all these processes, but the tightness of $(\theta_N)$ is trivial since the sequence lives in $[0, T]$ compact.  

\begin{proof}
We use the criterion of continuity to show the tightness of $(\overline{R}^N_\eta(t))$. 
For $\delta>0$, define
\[
	\omega_R^N(\delta):= \sup\{|\overline{R}^N_\eta(t)-\overline{R}^N_\eta(s)|, \, s, t\in [0, T], \, |s-t|<\delta\}. 
\] 

For $0\leq t<t'\leq T$ with $|t-t'|<\delta$, Relations~\eqref{e:sdeR} and~\eqref{e:AgRleq} give
\begin{align*}
	|\overline{R}^N_\eta(t')-\overline{R}^N_\eta(t)|
	&\leq \frac{1}{N}\int_{t\wedge \tau_N}^{t'\wedge \tau_N} |\widehat{\cal{A}}_\omega(g_1)(Y^N(s))|\diff s + 2\sup_{s\leq T}|M_R^N(s)|\\
	&\leq \delta C_R \sup_{t\leq T}\overline{R}_\eta^N(t)+ 2\sup_{s\leq T}|M_R^N(s)|
\end{align*}

Hence, 
\begin{align*}
	\P_{y_N}\left(\omega_R^N(\delta)>\eps\right)&\leq 
	\P_{y_N}\left(\sup_{t\leq T}\overline{R}_\eta^N(t)> \frac{\eps}{2\delta C_R}\right)+ \P_{(u_N, N)}\left(\sup_{s\leq T}|M_R^N(s)|>\eps/4\right)\\
	&\leq \frac{4\delta C_Re^{C_RT}}{\eps}+ \frac{4}{\eps}\E_{y_N}\left(\sup_{t\leq T}|M_R^N(t)|\right)
\end{align*}
and using Lemma~\ref{l:ControlR11}, for any $\eps, \eps'>0$, we can find some $\delta_0>0$, $N_0\geq 1$ such that, for $\delta\in (0, \delta_0)$ and $N\geq N_0$, 
\[
	\P_{y_N}\left(\omega_R^N(\delta)>\eps\right)\leq \eps'.  
\]
This result and Relation~\eqref{e:EsupR} allow us to apply Theorem 15.1 in~\cite{Billingsley} which gives the tightness of the stopped process $(\overline{R}_\eta^N(t))$, and guarantees that any limit of this process is continuous. 

We now show that $(L_\eta^N(t))$ is tight, using Aldous' criterion. 
We set 
\[	
	\alpha:= F'_\omega(u_*)+\frac{G_\omega(u_*)}{2}\quad \text{and}\quad \sigma^2= (1+u_*^2)^{3/2}D_\omega(u_*). 
\]
Note that $\alpha>0$ and $\sigma>0$ by assumption.

Applying Proposition~\ref{p:ReducedGenerator11} to the functions $\ell$ and $\ell^2$, we can find some $C_L>0$ such that, for $r\geq 1$, and $u$ such that $|\ell(u, r)|\leq \eta_1$, 
\begin{align*}
	&|\widehat{\cal{A}}_\omega(\ell)(u,r)-\alpha\ell(u,r)|\leq \frac{C_L}{\sqrt{r}},\\
	&|\Lambda(\ell)(u, r)-\sigma^2|\leq \frac{C_L}{\sqrt{r}}.
\end{align*} 

For a stopping time $\bar\tau\leq T$, for $\eps>0$ and $\delta>0$, we have 
\begin{align*}
	&\P_{y_N}\left(|M_L^N(\bar\tau+\delta)-M_L^N(\bar\tau)| > \eps\, |\, \cal{F}_{\bar\tau} \right)\\
	&\quad \leq \frac{1}{\eps^2}\E_{y_N}\left(\int_0^\delta |\Lambda(\ell)(Y^N((s+\bar \tau)\wedge \tau_N))|\diff s \, \big{|} \, \cal{F}_{\bar\tau}\right) \\
	&\quad \leq \frac{\delta}{\eps^2}(\sigma^2+C_L),
\end{align*} 
where we first used Doob's inequality, then the estimation of $\Lambda(\ell)$ given above. 

Therefore, 
\begin{align*}
	\P&\left(|L_\eta^N(\bar\tau+\delta)-L_\eta^N(\bar\tau)|\geq \eps\right)\\
	&\quad \leq \P\left(\int_{\bar\tau\wedge \tau_N}^{(\bar\tau+\delta)\wedge \tau_N}|\widehat{\cal{A}}_\omega(\ell)(U_N(s), R_N(s))|\diff s\geq \eps/2\right) \\
	&\quad \quad+ \E_{y_N}\left(\P_{y_N}\left(|M_L^N(\bar\tau+\delta)-M_L^N(\bar\tau)| > \eps\, |\, \cal{F}_{\bar\tau} \right)\right)\\
	&\quad \leq \frac{2\delta}{\eps}(\alpha \eta_1 + C_L) + \frac{4\delta}{\eps^2}(\sigma^2+C_L),
\end{align*}
where we used first Relation~\eqref{e:sdeL}, then Markov's inequality and the estimation on $\widehat{\cal{A}}_\omega(\ell)$ given above, to deal with the drift part. This proves Aldous' criterion as stated in (16.23) of~\cite{Billingsley}. 

Given the fact that for all $t\leq T$, $|L_\eta^N(t)|\leq \eta$ almost surely, Theorem 16.10 in~\cite{Billingsley} proves that the process is tight for the $J_1$-Skorokhod topology. To ensure the convergence to a continuous process, we simply have to show that 
$$(\sup_{t\leq T}|L^N_\eta(t)-L^N_\eta(t{-})|)$$
converges in probability to $0$ (see for example Proposition VI.3.26 in \cite{Jacod}), which is immediate given that at $(u, r)$ the height of the jump of $(L^N_\eta(t))$ is of order $r^{-1/2}$. This concludes the proof of Lemma~\ref{l:Tightness11}. 
\end{proof}

\subsection{Identification of the Ornstein-Uhlenbeck}

We introduce the infinitesimal generator $\cal{L}_{u_*}$ as, for $f\in \cal{C}^2$, and $x\in \R$, 
\[
	\cal{L}_{u_*}(f)(x)= \alpha xf'(x)+\frac{\sigma^2}{2}f''(x). 
\]

Because of Lemma~\ref{l:Tightness11}, for some subsequence $(N_k)$, for the convergence in distribution, 
\[
	\lim_{k\to +\infty} ((L^{N_k}_\eta(t), \overline{R}^{N_k}_\eta(t))_{[0, T]}, \theta_{N_k})=((L_\infty(t), R(t)), \theta). 
\]
We now identify the limit.

\begin{lemma}\label{l:IdLOU}
	The limit process $(L_\infty(t))$ verifies the following stopped martingale  problem: for all $f\in \cal{C}^\infty_c$, 
	\[
		f(L_\infty(t))-f(\ell_0)-\int_0^{t\wedge \theta} \cal{L}_{u_*}(f)(L_\infty(s))\diff s
	\]
	is a martingale. 
\end{lemma}

\begin{proof}
	We use a similar setting as in Theorem 4.8.2 in~\cite{KurtzEthier}. Let $k\geq 0$, $0\leq t_1<\ldots <t_k\leq t<t'$, and $f\in\cal C_c^\infty(\R)$ and $h_1,\ldots,h_k\in\cal C_c(\R)$.
	
	Using Skorokhod representation theorem we can assume that $((L^{N_k}_\eta(t), R^{N_k}_\eta(t))_{[0, T]}, \theta_{N_k})$ converges almost surely to $((L_\infty(t), R(t)), \theta)$, and that they all live in the same probability space. We now work on this subsequence, but the index $k$ is dropped to avoid heavy notations. 
	Since $(L_\infty(t))$ is continuous, almost surely, 
	\begin{equation}\label{e:LConvCInfty}
		\lim_{N\to +\infty}\sup_{t\leq T}|L_\eta^N(t)-L_\infty(t)|=0. 
	\end{equation}
	
	The following holds: 
	\begin{align*}
		f(L_\eta^N(t'))-&f(L_\eta^N(t))-\int_{t\wedge \theta}^{t'\wedge\theta} \cal{L}_{u_*}(f)(L^N_\eta(s))\diff s\\
		&\, =f(L_\eta^N(t'))-f(L_\eta^N(t))-\int_{t\wedge \theta_N}^{t'\wedge\theta_N} \widehat{\cal{A}}_\omega(f\circ \ell)(Y^N(s))\diff s\\
		&\quad + \int_{t\wedge \theta_N}^{t'\wedge\theta_N} (\widehat{\cal{A}}_\omega(f\circ \ell)(Y^N(s))-\cal{L}_{u_*}(f)(L^N_\eta(s)))\diff s + \cal{E}_\theta^N
	\end{align*}
	where $\cal{E}_\theta^N$ stands for the corrective term due to the difference between $\theta_N$ and $\theta$. 
	Since $\cal{L}_{u_*}(f)\in \cal{C}_c(\R)$, we have 
	\[
		|\cal{E}_\theta^N|\leq \|\cal{L}_{u_*}(f)\|_\infty |\theta_N-\theta|. 
	\]
	
	Relation~\ref{e:IGReduced} yields the local convergence of the generator: we can find $C_f>0$ such that, for $r\geq 1$ and $u$ s.t. $|\ell(u, r)|\leq \eta_1$, 
	\[
		|\widehat{\cal{A}}_\omega(f\circ \ell)(u,r)-\cal{L}_{u_*}(f)(\ell(u,r))|\leq \frac{C_f}{\sqrt{r}}. 
	\]
	Therefore, 
	\[
		\int_{t\wedge \theta_N}^{t'\wedge\theta_N} \left|\widehat{\cal{A}}_\omega(f\circ \ell)(Y^N(s))-\cal{L}_{u_*}(f)(L^N_\eta(s))\right|\diff s
		\leq \frac{(t'-t)C_f}{\sqrt{N}}\sup_{t\leq T} \frac{1}{\sqrt{\overline{R}^N_\eta(t)}}. 
	\]
	
	Hence, we have
	\begin{align*}
		&\left|\E_{y_N}\left[\left(f(L_\eta^N(t'))-f(L_\eta^N(t))-\int_{t\wedge \theta}^{t'\wedge\theta} \cal{L}_{u_*}(f)(L^N_\eta(s))\diff s\right)\prod_{i=1}^k h_i(L^N_\eta(t_i))\right]\right|\\
		&\leq \left|\E_{y_N}\left[\left(f(L_\eta^N(t'))-f(L_\eta^N(t))-\int_{t\wedge \theta_N}^{t'\wedge\theta_N} \widehat{\cal{A}}_\omega(f\circ \ell)(U^N(s), R^N(s))\diff s\right)\prod_{i=1}^k h_i(L^N_\eta(t_i))\right]\right|\\
		&\quad + \frac{(t'-t)C_f}{\sqrt{N}}\prod_{i=1}^k\|h_i\|_\infty \E_{y_N}\left(\sup_{t\leq T} (\overline{R}^N_\eta(t))^{-1/2}\right) \\
		&\quad+\|\cal{L}_{u_*}(f)\|_\infty \prod_{i=1}^k\|h_i\|_\infty \E_{y_N}\left(|\theta_N-\theta|\right).
	\end{align*}
	The first term of the right-hand side is equal to $0$ by definition of $(L^N_\eta(t))$. The second term goes to $0$ when $N\to +\infty$ by Lemma~\ref{l:ControlR11}, and so does the last one by convergence almost sure of $(\theta_N)$ to $\theta$. 
	
	Finally, since the function  
	\[
		x\to \left(f(x(t'))-f(x(t))-\int_t^{t'}\cal{L}_{u_*}(f)(x(s))\diff s\right)\prod_{i=1}^k h_i(x(t_i))
	\]
	is continuous for the topology of the uniform norm, Relation~\eqref{e:LConvCInfty} leads to
	\[
		\E\left[\left(f(L_\infty(t'))-f(L_\infty(t))-\int_{t\wedge \theta}^{t'\wedge\theta} \cal{L}_{u_*}(f)(L_\infty(s))\diff s\right)\prod_{i=1}^k h_i(L_\infty(t_i))\right]=0
	\]
	and the result is shown.
\end{proof}

\subsection{Identification of the stopping time $\theta$}

\begin{lemma}
	Defining $\tau:=\inf\{t\geq 0: |L_\infty(t)|\geq \eta_1\}$, we have $\theta=T\wedge \tau$. 
\end{lemma}

\begin{proof}
First show that $\theta\geq T\wedge \tau$. 
Since $\theta^N=\tau_N\wedge T$, for all $N\geq 1$, $\theta_N\leq T$ almost surely, and therefore $\theta\leq T$. 
Now assume that $\theta<T$. We have to show that $\tau\leq\theta$. It is sufficient to show that $|L_\infty(\theta)|=\eta_1$. 

Keeping the setting of the proof of Lemma~\ref{l:IdLOU} (with the Skorokhod representation), on the event $\{\theta<T\}$, we have: 
\begin{equation}\label{e:DecompLInfty}
	||L_\infty(\theta)|-\eta_1|\leq |L_\infty(\theta)-L_\infty(\theta^N)| +|L_\infty(\theta^N)-L_\eta^N(\theta^N)|+ ||L_\eta^N(\theta^N)|-\eta_1|
\end{equation}
By continuity of $(L_\infty(t))$, the first term on the right-hand side goes to $0$ when $N\to +\infty$. The second term also goes to $0$ because of Equation~\eqref{e:LConvCInfty}. Finally, because almost surely, $\theta_N\to\theta$, for $N$ large enough, we have $\theta_N<T$ on the event $\{\theta<T\}$. Therefore, $||L_\eta^N(\theta^N)|-\eta_1|\leq \sup_{t\leq T}|L^N_\eta(t)-L^N_\eta(t{-})|\to 0$, and we proved the first inequality. 

Now let's show that $\theta\leq \tau\wedge T$. If $\tau<\theta$, since $(L_\infty(t))$ verifies the stopped martingale problem of a divergent O.U. until $\theta$, we have for any $\eps>0$,  
\[
	\sup_{s\in [\tau, \tau +\eps]}|L_\infty(s)|>\eta_1.
\] 
But using the same decomposition as in Relation~\eqref{e:DecompLInfty}, we can show that $|L_\infty(t)|\leq \eta_1$ for all $t\geq 0$, which leads to a contradiction. Therefore, $\theta\leq \tau$, which concludes the proof. 
\end{proof}

\subsection{Limit of $(\overline{R}^N_\eta(t))$}

The proof of the convergence of $(\overline{R}^N_\eta(t))$ to $(e^{G_\omega(u_*) (t\wedge \tau)})$ can be deduced from Equation~\eqref{e:sdeR}, the bound~\eqref{e:EsupMR} and the convergence of the stopping time. It is left to the reader. 

Finally, Relation~\eqref{e:OU11Lim2} comes from this convergence and Relation~\eqref{e:EsupR}.

\end{document}